%% file: BnB_PEP_manuscript.tex
\documentclass[english, 11pt]{article}
\usepackage{xcolor}
\usepackage{verbatim}
\usepackage{float}
\usepackage{mathtools}
\usepackage{url}
\usepackage{amsmath}
\usepackage{amsthm}
\usepackage{amssymb}
\usepackage{booktabs}

\makeatletter

\providecommand{\tabularnewline}{\\}
\floatstyle{ruled}
\newfloat{algorithm}{tbp}{loa}
\providecommand{\algorithmname}{Algorithm}
\floatname{algorithm}{\protect\algorithmname}

\input{shuvos_preamble.tex}

\author{
  \small{\textbf{Shuvomoy Das Gupta}}\\
  \small{MIT Operations Research Center} \\
  \small{\texttt{sdgupta@mit.edu}} \\
  \and
  \small{\textbf{Bart P.G. Van Parys}} \\
  \small{MIT Sloan School of Management} \\
  \small{\texttt{vanparys@mit.edu}} \\
  \and
  \small{\textbf{Ernest K. Ryu}}\\
  \small{Seoul National University} \\
  \small{\texttt{ernestryu@snu.ac.kr}} \\
}

\date{}

\makeatother

\usepackage{babel}

\begin{document}

\title{Branch-and-Bound Performance Estimation Programming: A Unified Methodology for Constructing Optimal Optimization Methods}

\maketitle
\input{shuvos_macros.tex}

\input{Sections/abstract.tex}

\newpage

\setcounter{tocdepth}{2}

\tableofcontents

\newpage

\input{Sections/section_1.tex}

\input{Sections/section_2.tex}

\input{Sections/section_3.tex}

\input{Sections/section_4.tex}

\input{Sections/acceleration_without_momentum.tex}

\input{Sections/smooth_nonconvex_gradient_reduction.tex}

\input{Sections/optimal_algorithm_with_respect_to_a_potential_for_weakly_convex_gradient_reduction.tex}

\input{Sections/conclusions.tex}

\bibliographystyle{plain}
\bibliography{BnB_PEP_manuscript}

\input{Sections/supplementary_stuff.tex}

\end{document}

%% file: shuvos_preamble.tex
\usepackage{graphicx,psfrag,amsfonts,verbatim,fullpage,mathtools}
\usepackage[hidelinks]{hyperref}

\usepackage{xurl}

\usepackage{bbm}

\usepackage{dsfont}

\usepackage[framemethod=default]{mdframed}

\global\mdfdefinestyle{vert_line_shuvo}
{
linecolor=black, 
linewidth=2pt, 
topline=false, 
nobreak=true,
bottomline=false, 
rightline=false
} 

\global\mdfdefinestyle{algorithm_box_shuvo}
{
linecolor=black, 
linewidth=1pt, 
innerleftmargin=1pt,
innertopmargin=5pt,
innerbottommargin=5pt,
innerrightmargin=1pt,
nobreak=true,
}

\usepackage{ upgreek } %

\usepackage[bbgreekl]{mathbbol}
\DeclareSymbolFontAlphabet{\mathbbm}{bbold}
\DeclareSymbolFontAlphabet{\mathbb}{AMSb}%
\usepackage[bb=px]{mathalfa}
\usepackage{bm}

\PassOptionsToPackage{obeyFinal}{todonotes}
\usepackage{todonotes}

\usepackage{abstract} %

\theoremstyle{plain}
\newtheorem{thm}{\protect\theoremname}
\theoremstyle{plain}
\newtheorem{cor}{\protect\corollaryname}
\theoremstyle{plain}
\newtheorem{lem}{\protect\lemmaname}
\theoremstyle{plain}
\newtheorem{assumption}{Assumption}
\theoremstyle{plain}

\theoremstyle{definition}

\theoremstyle{plain}

\theoremstyle{remark}

\theoremstyle{plain}
\newtheorem{conjecture}{Conjecture}

\usepackage{babel}
\providecommand{\lemmaname}{Lemma}
\providecommand{\propositionname}{Proposition}
\providecommand{\theoremname}{Theorem}
\providecommand{\corollaryname}{Corollary}

\usepackage{parskip}
\begingroup
\makeatletter
   \@for\theoremstyle:=definition,remark,plain\do{%
     \expandafter\g@addto@macro\csname th@\theoremstyle\endcsname{%
        \addtolength\thm@preskip\parskip
     }%
   }
\endgroup
\usepackage{parskip}

\usepackage{diagbox}

\definecolor{commentcolor}{RGB}{74,112,35}

\usepackage{titlesec}
\def\sectionfont{\sffamily\Large\bfseries\boldmath}
\def\subsectionfont{\sffamily\large\bfseries\boldmath}
\def\paragraphfont{\sffamily\normalsize\bfseries\boldmath}
\titleformat*{\section}{\sectionfont}
\titleformat*{\subsection}{\subsectionfont}
\titleformat*{\subsubsection}{\paragraphfont}
\titleformat*{\paragraph}{\paragraphfont}
\titleformat*{\subparagraph}{\paragraphfont}
\usepackage[small,labelfont={bf,sf}]{caption}  %
\usepackage{multirow}
\usepackage{enumitem}
\setlist{nolistsep}

\usepackage{booktabs}           %
\usepackage{adjustbox} %

\usepackage{titling}
\pretitle{\begin{center}\LARGE \bfseries\sffamily}
\posttitle{\par\end{center}\vskip 2em}
\preauthor{\begin{center}
\large \lineskip 2em%
\begin{tabular}[t]{c}}
\postauthor{\end{tabular}\par\end{center}}
\predate{\begin{center}\large}
\postdate{\par\end{center}}

\usepackage{tikz}
\usetikzlibrary{fit, arrows, shapes, positioning, shadows, matrix, calc}
\usepackage[]{subfig}

\RequirePackage{luatex85}
\usepackage{pgfplots}
\pgfplotsset{compat=newest}
\usepgfplotslibrary{groupplots}
\usepgfplotslibrary{polar}
\usepgfplotslibrary{smithchart}
\usepgfplotslibrary{statistics}
\usepgfplotslibrary{dateplot}
\allowdisplaybreaks

%% file: shuvos_macros.tex
\global\long\def\rl{{\mathbb R}}%

\global\long\def\eu{\mathbb{E}}%

\global\long\def\ig{\mathbb{Z}}%

\global\long\def\nt{\mathbb N}%

\global\long\def\limsu{\mathop{\overline{\mathrm{lim}}}}%

\global\long\def\limin{\mathop{\underline{\mathrm{lim}}}}%

\global\long\def\epi{\mathop{{\bf epi}}}%

\global\long\def\intr{\mathop{{\bf int}}}%

\global\long\def\dom{\mathop{{\bf dom}}}%

\global\long\def\ran{\mathop{{\bf ran}}}%

\global\long\def\tr{\mathop{{\bf tr}}}%

\global\long\def\zer{\mathop{{\bf zer}}}%

\global\long\def\conv{\mathop{{\bf conv}}}%

\global\long\def\prox{\mathbf{prox}}%

\global\long\def\argmin{\mathop{{\rm argmin}}}%

\global\long\def\gra{\mathop{{\bf gra}}}%

\global\long\def\fix{\mathop{\mathop{{\bf fix}}}}%

\global\long\def\rank{\mathop{{\bf rank}}}%

\global\long\def\card{\mathop{{\bf card}}}%

\global\long\def\proj{\mathop{{\bf \Pi}}}%

\global\long\def\ones{\mathbf{1}}%

\global\long\def\zeros{\mathbf{0}}%

\global\long\def\regf{f}%

\global\long\def\ind{\iota}%

\global\long\def\sg{\widetilde{\nabla}}%

\global\long\def\morf{\prescript{\gamma}{}{f}}%

\global\long\def\regind{\prescript{\mu}{}{\bbiota}}%

\global\long\def\opA{\mathbb{A}}%

\global\long\def\subg{\widetilde{\nabla}}%

\global\long\def\sgf{f^{\prime}}%

\global\long\def\gf{\nabla f}%

\global\long\def\ft{\tilde{f}}%

\global\long\def\sgft{\tilde{f}^{\prime}}%

\global\long\def\gft{\nabla\tilde{f}}%

\global\long\def\opT{\mathbb{T}}%

\global\long\def\opJ{\mathbb{J}}%

\global\long\def\opR{\mathbb{R}}%

\global\long\def\opS{\mathbb{S}}%

\global\long\def\id{\mathbb{I}}%

\global\long\def\cpc{\mathcal{F}_{0,\infty}}%

\global\long\def\cpcLSbgBd{\mathcal{F}_{0,\infty}^{L}}%

\global\long\def\LSmthCvx{\mathcal{F}_{0,L}}%

\global\long\def\LMinusMuSmthCvx{\mathcal{F}_{0,L-\mu}}%

\global\long\def\LSmthMuStCvx{\mathcal{F}_{\mu,L}}%

\global\long\def\LSmth{\mathcal{F}_{-L,L}}%

\global\long\def\rhoLWcvx{\mathcal{W}_{\rho,L}}%

%% file: Sections/abstract.tex
\begin{abstract}
We present the Branch-and-Bound Performance Estimation Programming (BnB-PEP), a unified methodology for constructing optimal first-order methods for convex and nonconvex optimization. BnB-PEP poses the problem of finding the optimal optimization method as a nonconvex but practically tractable quadratically constrained quadratic optimization problem and solves it to certifiable global optimality using a customized branch-and-bound algorithm. By directly confronting the nonconvexity, BnB-PEP offers significantly more flexibility and removes the many limitations of the prior methodologies. Our customized branch-and-bound algorithm, through exploiting specific problem structures, outperforms the latest off-the-shelf implementations by orders of magnitude, accelerating the solution time from hours to seconds and weeks to minutes. We apply BnB-PEP to several setups for which the prior methodologies do not apply and obtain methods with bounds that improve upon prior state-of-the-art results. Finally, we use the BnB-PEP methodology to find proofs with potential function structures, thereby systematically generating analytical convergence proofs.
\end{abstract}

%% file: Sections/section_1.tex
\section{Introduction \label{sec:Introduction}}

Since the pioneering work of Nesterov and Nemirovsky on accelerated
gradient methods \cite{Nest83} and information-based complexity \cite{nemirovsky1992information,nemirovski1995information},
finding efficient and optimal first-order methods has been the focus
in the study of large-scale optimization. Recently, renewed vitality
was injected into this classical line of research by the emergence
of computer-assisted methodologies following the Performance
Estimation Problem (PEP) of Drori and Teboulle \cite{drori2014performance}. The celebrated
accelerated gradient method by Nesterov was improved by a constant
factor in \cite{kim2016optimized,van2017fastest,taylor2021optimal},
and entirely novel acceleration mechanisms, distinct from Nesterov's,
have been discovered \cite{kim2021optimizing,kim2021accelerated,lieder2021convergence,yoon2021accelerated}.
These computer-assisted methodologies, roughly speaking, pose the
problem of analyzing an efficient method as a convex semidefinite program,
and the convexity provides certain algorithmic guarantees.

However, the convexity in the formulation simultaneously serves as
a limitation. The aforementioned works presented several ingenious
changes of variables, relaxations, and reformulations to retain convexity,
but such efforts cover only a handful of setups. When these techniques
do not apply, the prior methodologies become inapplicable.

\paragraph{Contribution.}
This work presents 
the Branch-and-Bound Performance
Estimation Programming (BnB-PEP), a methodology for constructing optimal first-order
methods for convex and nonconvex optimization in a tractable and unified
manner. We formulate the problem of finding the optimal optimization
method as a nonconvex quadratically constrained quadratic problem
(QCQP).
By directly confronting the nonconvexity of the QCQPs in consideration, BnB-PEP offers
significantly more flexibility and removes 
the many limitations of
the prior PEP-based methodologies.
We then provide  a customized spatial branch-and-bound algorithm that enables us to solve such QCQPs to certifiable global optimality in a practical time scale.
The customization speeds up the branch-and-bound algorithm, compared to the latest off-the-shelf implementations, by orders of magnitude, reducing runtimes from hours to seconds and weeks to minutes. %
We apply the BnB-PEP methodology to several setups for which the prior methodologies do not apply
and construct methods with bounds improving upon prior state-of-the-art
results. 
Finally, we use the BnB-PEP methodology to find proofs with potential function structures, thereby systematically generating analytical convergence proofs.

\subsection{Prior work}

The performance estimation methodology, initiated by Drori and Teboulle \cite{drori2014performance}, formulates the worst case-performance of an optimization method as an optimization problem itself and upper bounds this performance through a semidefinite program (SDP) ``relaxation''. Taylor, Hendrickx, and Glineur then showed that the SDP formulation is, in fact, tight (not a relaxation) through the notion of convex interpolation \cite{taylor2017smooth}. Lessard, Recht, and Packard combined the notion of the performance estimation methodology with control-theoretic notions through their integral quadratic constraints (IQC) formulation \cite{lessard2016analysis}. Taylor, Van Scoy, and Lessard then showed that IQCs could be seen as a feasible solution to performance estimation problems finding optimal linear convergence rate through Lyapunov functions \cite{taylor2018lyapunov}. Taylor and Bach extended this observation through a methodology that uses the performance estimation approach to find the optimal sublinear rates through potential functions \cite{taylor2019stochastic}.

This advancement in the the performance estimation methodology has led to the discovery of many novel methods and analyses.
Drori and Teboulle numerically constructed the optimized gradient method (OGM) \cite{drori2014performance} and Kim and Fessler found its analytical description \cite{kim2016optimized}.
OGM surpasses Nesterov's fast gradient method by a constant factor and Drori showed that OGM is exactly optimal through an exact matching complexity lower bound \cite{drori2017exact}.
Drori and Taylor constructed efficient first-order methods that utilize 1D and 3D exact line searches using a span-search variant of the performance estimation methodology \cite{drori2020efficient}.
Van Scoy, Freeman, and Lynch constructed the triple momentum method, a method that surpasses the Nesterov's fast gradient method for the strongly convex setup by a constant factor, using the IQC methodology \cite{van2017fastest}.
Taylor and Drori constructed the information-theoretic exact method (ITEM), further improving upon the triple momentum method \cite{taylor2021optimal}. Drori and Taylor showed that ITEM is exactly optimal through an exact matching complexity lower bound \cite{drori2022oracle}.
Kim and Fessler constructed OGM-G, which has the best known rate for reducing the gradient magnitude in the smooth convex setup \cite{kim2021optimizing}.
Finally, the performance estimation methodology has also been utilized for constructing methods with inexact evaluations \cite{cyrus2018,barre2020,de2020worst}, analyzing methods in the composite minimization setup \cite{kim2018another,taylor2018exact}, analyzing methods with exact line search \cite{de2017worst}, analyzing monotone operator and splitting methods \cite{ryu2020operator,gu2020tight,lieder2021convergence,kim2021accelerated}, and analyzing acceleration in the mirror descent setup \cite{Dragomir2021}.

Prior work constructing efficient methods based on the performance estimation methodology relies on two conceptual stages. The first stage poses the \emph{inner} problem as finding the worst-case performance of a given method and formulates the inner problem as a convex SDP through some ingenious change of variables and SDP duality. The second stage constructs an \emph{outer} problem that minimizes the aforementioned worst-case performance as a function of the method. An equivalent view is that the inner problem finds a convergence proof and the outer problem finds the algorithm with the smallest (best) guarantee established by the convergence proof of the inner problem. While the inner optimization problem is convex, setups for which the outer minimization problem is convex are quite rare. As we detail in $\mathsection$\ref{subsec:Converting-minmax-into-a-bnb-pep}, prior work circumvents this nonconvexity within the scope of convex optimization through relaxations and heuristics. However, these prior techniques do not always apply, especially when the underlying optimization problem is nonconvex.
The setups of $\mathsection$\ref{subsec:Smooth-nonconvex-gradient-reduction} and $\mathsection$\ref{subsec:pot-bnb-pep-wcvx-Moreau} of this work are such examples.

\subsection{Organization}
This paper is organized as follows.
In $\mathsection$\ref{sec:background}, we present the necessary background and describe our problem setup.
In $\mathsection$\ref{sec:QCQO-framework-for-template-problem}, we illustrate the BnB-PEP methodology by applying it on a concrete problem instance of constructing the optimal fixed-step first-order method for reducing the gradient of a strongly convex and smooth function.
Our discussion up to $\mathsection$\ref{par:Converting-the-inner-problem-into-SDP} follows prior approaches,
and our novel contribution starts in $\mathsection$\ref{subsec:Converting-minmax-into-a-bnb-pep}.
In $\mathsection$\ref{sec:Efficient-implementation-and-enhancement}, we present customizations of the spatial branch-and-bound algorithm that enables us to solve the QCQPs that construct optimal optimization methods in a practical time scale.
In $\mathsection$\ref{subsec:Generalization-of-QCQO-framework}, we present the generalized formulation of our methodology.
In $\mathsection$\ref{sec:applications}, we demonstrate the effectiveness of our methodology through several applications.
In $\mathsection$\ref{subsec:AWM}, we construct the optimal gradient method without momentum for reducing function value in the smooth convex setup and demonstrate that it outperforms the best known method without momentum. 
In $\mathsection$\ref{subsec:Smooth-nonconvex-gradient-reduction}, we construct the optimal method for reducing gradient norm of smooth nonconvex functions and demonstrate that it outperforms the prior best known method \cite{abbaszadehpeivasti2021exact}. 
In $\mathsection$\ref{subsec:pot-bnb-pep-wcvx-Moreau}, we design an optimized first-order method with
respect to a suitable potential function for reducing the (sub)gradient
norm of nonsmooth weakly convex functions and demonstrate that it outperforms the prior best known method \cite[Theorem 3.1]{davis2019stochastic}.
Additionally, in $\mathsection$\ref{subsec:Numerical-results-ncvx-moreau}, we present a systematic approach to generate analytical proofs from the solutions obtained through our methodology, extending
the approach of Taylor and Bach \cite{taylor2019stochastic} to nonconvex and nonsmooth setups.

\subsection{Computational setup} \label{subsec:computational_setup}

For scientific reproducibility, we open-source our codes to generate all the numerical results presented in this paper at the link: 

\begin{center}
\url{https://github.com/Shuvomoy/BnB-PEP-code}
\par\end{center}

Unless otherwise specified, we performed our numerical
experiments on a laptop computer running Windows 10 Pro with Intel
Core i7-8650U CPU with 16 GB of RAM. We used \texttt{JuMP}---a domain-specific
modeling language for mathematical optimization embedded in the open-source
programming language \texttt{Julia} \cite{JuMPDunningHuchetteLubin2017}---to model the optimization problems. Our proposed algorithm uses
the following solvers: \texttt{Mosek 9.3 }\cite{mosek} (free for
academic use), \texttt{Ipopt 3.12.11} \cite{IpoptLocConv} (open-source), \texttt{KNITRO 13.0.0} \cite{byrd2006k} (free for academic use),
and \texttt{Gurobi 10} \cite{Gurobi95} (free for academic use).

\section{Background and problem setup\label{sec:background}}

Write $\rl^{d}$ for the underlying Euclidean space, even though
our results and formulations extend to the setup where the underlying
setup is a Hilbert space \cite{ryu2020operator}. Write $\left\langle \cdot\mid\cdot\right\rangle $
and $\|\cdot\|$ to denote the standard inner product and norm on
$\rl^{d}$. For $a,b\in\nt$, denote $[a:b]=\{a,a+1,a+2,\ldots,b-1,b\}$.
Write $\rl^{m\times n}$ for the set of $m\times n$ matrices,
$\mathbb{S}^{n}$ for the set of $n\times n$ symmetric matrices,
and $\mathbb{S}_{+}^{n}$ for the set of $n\times n$ positive-semidefinite matrices. We use the standard notation $e_i \in \rl^d$ for the unit vector having a single 1 as
its $i$-th component.
Write $\left(\cdot\odot\cdot\right)\colon\rl^{d}\times\rl^{d}\to\rl^{d\times d}$
to denote the symmetric outer product, that is, for any $x,y\in\rl^{d}$:
\[
x\odot y=\frac{1}{2}\left(xy^{\top}+yx^{\top}\right).
\]

We follow standard convex-analytical definitions \cite{boyd2004convex,nesterov2003introductory,rockafellar2009variational,RyuYin2022_largescale}.
A set $S\subseteq\rl^{d}$ is convex if for any $x,y\in S$ and $\theta\in[0,1],$ we have $\theta x+(1-\theta)y\in S$.
A function $f\colon\rl^{d}\to\rl$ is convex if
\[
f\left(\theta x+(1-\theta)y\right)\leq\theta f(x)+(1-\theta)f(y)
\]
 for all $x,y\in\rl^d$ and $\theta\in(0,1)$.
 
The abstract subdifferential of $f\colon\rl^{d}\to\rl$ at $x$, denoted by $\partial f(x)$,
is defined to satisfy the following properties \cite{Bauschke2021}:
\begin{itemize}
\item[(i)] If $f$ is convex, then the abstract subdifferential is the usual convex subdifferential, i.e., 
\[
\partial f(x) = \{g\in \rl^n\,|\,f(y) \geq f(x) + \langle g \mid y-x\rangle,
\,\forall\,y\in \rl^n
\}.
\]
\item[(ii)] If $f$ is continuously differentiable at $x$, then its abstract subdifferential at $x$ just contains the gradient $\nabla f(x)$, i.e., $\partial f(x)=\{\nabla f(x)\}.$
\item[(iii)] If $f$ attains a local minimum at $x$, 
then $0\in\partial f(x).$ 
\item[(iv)] For all $y\in\rl^{d}$ and $\beta\in\rl$,
\[
\partial\left(f(\cdot)+\frac{\beta}{2}\|\cdot-y\|^{2}\right)=\partial f(\cdot)+\beta(\cdot-y).
\]
\end{itemize}
The Clarke--Rockafellar subdifferential \cite[$\mathsection$1.2]{clarke1990optimization}, Mordukhovich subdifferential \cite[$\mathsection$1.3]{mordukhovich2006PartI}, and Fr\'echet subdifferential  \cite[page 132]{BorweinLewis2006}
are all instances of the abstract subdifferential \cite[page 70]{Bauschke2021}.
Whenever we say \emph{subdifferential} in this paper, we are referring to the abstract subdifferential.
Because our analyses use only the properties of the abstract subdifferential, our results apply to all instances of the abstract subdifferential.
We write $f^{\prime}(x)$ to denote an element of $\partial f(x)$.

We say a function $f$ is $L$-smooth if it is differentiable everywhere and $\nabla f$ is $L$-Lipschitz continuous.
We say a function $f$ is $\mu$-strongly convex if $f(\cdot)-(\mu/2)\|\cdot\|^2$ is convex. {We say a function $f$ is $\rho$-weakly convex if $f+(\rho/2)\|\cdot\|^2$ is convex.}
We say a function $f$ has $L$-bounded subgradients if $\|g\|\le L$ for all $g\in \partial f(x)$ and $x\in \rl^d$.

\subsection{Quadratically constrained quadratic program (QCQP)\label{subsec:Preliminaries-on-QCQP}}

A QCQP is defined as: 
\begin{equation}
p^{\star}=\left(\begin{array}{ll}
\underset{x\in\rl^{q}}{\mbox{minimize}} & c^{\top}x+x^{\top}Q_{0}x\\
\textrm{subject to} & a_{i}^{\top}x+x^{\top}Q_{i}x\leq b_{i},\quad i\in[1:m],\\
 & a_{j}^{\top}x+x^{\top}Q_{j}x=b_{j},\quad j\in[m+1:p],
\end{array}\right)\label{eq:QCQO-template}
\end{equation}
where $x\in\rl^{q}$ is the decision variable. The matrices $Q_{0},Q_{1},\ldots,Q_{p}\in\rl^{q\times q}$
are symmetric, but not necessarily positive-semidefinite. Therefore,
this problem is nonconvex.

\paragraph{Practical tractability of QCQPs.}

The QCQP problem class is NP-hard and therefore has no known polynomial-time
algorithm \cite[pp. 565--567]{bertsimas2005optimization}. However,
such theoretical worst-case intractability does not necessarily imply
that specific problem instances are not \emph{practically tractable}
\cite[Chapter 1]{bertsimas2019machine}.

Branch-and-bound solvers have experienced astounding speedup in the
past few decades. In the last thirty years, branch-and-bound solvers
for mixed-integer optimization (MIO) problems have achieved an algorithmic
speedup of approximately 1,250,000 and a hardware speedup of approximately
1,560,000, resulting in an overall speedup factor of approximately
2 trillion \cite[page 5]{bertsimas2019machine}. While these speedup
factors are for MIO and not for\textsf{ }QCQP, the speedup factors
for QCQP solvers have followed a similar trend since the recent (2019)
incorporation of QCQPs in commercial solvers \cite{Gurobi}. For example,
in less than three years,\texttt{ Gurobi}'s spatial branch-and-bound
algorithm has achieved a machine-independent speedup factor of 175.5
\cite{Gurobi91,Gurobi95}.

This remarkable speedup has rendered previously intractable problems
practically tractable. Furthermore, one can often significantly speed
up the spatial branch-and-bound algorithm by customizing it to exploit
specific problem structure. We present such customizations for the
BnB-PEP Algorithm in $\mathsection$\ref{sec:Efficient-implementation-and-enhancement}
and $\mathsection$\ref{subsec:Generalization-of-QCQO-framework}, and demonstrate
that the speedup is absolutely essential for the BnB-PEP Algorithm to be used practically.

\paragraph{QCQP solvers for local and global solutions.}

Since QCQPs have twice-continuously differentiable objectives and
constraints, one can use interior-point solvers such as \texttt{KNITRO}
\cite{byrd2006k} or \texttt{Ipopt} \cite{wachter2006implementation}
to compute locally optimal solutions under certain regularity conditions
\cite[Theorem 4]{fiacco1990nonlinear}\cite[$\mathsection$3.2]{liberti2008introduction}.
On the other hand, one can use the spatial branch-and-bound algorithm
implemented in solvers such as \texttt{Gurobi} \cite{Gurobi95}
to find globally optimal solutions of nonconvex QCQPs and to certify
their optimality in finite time.

\subsection{Problem setup\label{subsec:Preliminaries-on-performance-estimation-prob}}

Consider the unconstrained minimization problem 
\begin{equation}
\begin{array}{ll}
\underset{x\in\mathbb{R}^{d}}{\mbox{minimize}} & f(x),\end{array}\label{eq:main-opt}
\end{equation}
where $f$ is smooth or nonsmooth, convex or nonconvex. %

For the sake of simplicity, we assume that $f$ has a global minimizer $x_{\star}$ (not necessarily unique).
Optimization problems with further structure, such as problems with
constraints and problems whose objective are sums of functions, can
also be considered. However, we restrict our discussion to this setup
of unconstrained minimization for the sake of simplicity.

\paragraph{Function class $\mathcal{F}$.}
An optimization method is usually designed for a specific class of functions.
In this work, we use the BnB-PEP methodology with function classes listed in Table~\ref{tab:Considered-function-classes}.
More generally, we can use the BnB-PEP methodology with quadratically representable function classes, a notion we further discuss in $\mathsection$\ref{appendix:function-class} of the appendix.

\begin{table}

\begin{centering}
{\footnotesize{}}%
\begin{tabular}{cc}
\toprule 
{\footnotesize{}Function class} & {\footnotesize{}Notation}\tabularnewline
\midrule
\midrule 
{\footnotesize{}$L$-smooth Convex ($0<L<\infty$)} & {\footnotesize{}$\mathcal{F}_{0,L}$}\tabularnewline
\midrule 
{\footnotesize{}$L$-smooth and $\mu$-strongly convex ($0\leq\mu<L<\infty$) } & {\footnotesize{}$\mathcal{F}_{\mu,L}$}\tabularnewline
\midrule 
{\footnotesize{}$L$-smooth Nonconvex ($0<L<\infty$)} & {\footnotesize{}$\mathcal{F}_{-L,L}$}\tabularnewline
\midrule 
{\footnotesize{}$\rho$-weakly convex with $L$-bounded subgradients
($\rho>0,L>0$)} & {\footnotesize{}$\mathcal{W}_{\rho,L}$}\tabularnewline
\bottomrule
\end{tabular}
\par\end{centering}

\caption{Function classes considered in this paper.\label{tab:Considered-function-classes}}
\end{table}

\paragraph{Fixed-step first-order method $\mathcal{M}_{N}$.}

We consider fixed-step first-order methods, which include most subgradient
methods and accelerated gradient methods \cite{taylor2021optimal}.
A method is said to be a \emph{fixed-step first-order method} (FSFOM)
with $N$ steps if it takes in
a function $f$ and a starting point $x_{0}\in\mathbb{R}^{d}$
as input and produces its iterates with: 
\begin{equation}
x_{i}=x_{i-1}-\sum_{j=0}^{i-1}s_{i,j}\sgf(x_{j})\label{eq:fsfom}
\end{equation}
for $i\in[1:N]$, where $\sgf(x_{j})\in\partial f(x_{j})$ is a subgradient
of $f$ at $x_{j}$ for $j\in[0:N-1]$. We can equivalently express
the FSFOM \eqref{eq:fsfom} with 
\[
x_{i}=x_{0}-\sum_{j=0}^{i-1}\overline{s}_{i,j}f^{\prime}(x_{j}),
\]
where $\{s_{i,j}\}_{0\leq j<i\leq N}$ and $\{\overline{s}_{i,j}\}_{0\leq j<i\leq N}$ are related by 
\begin{equation}
\overline{s}_{i,j}=\begin{cases}
s_{i,i-1}, & \textrm{if }j=i-1,\\
\overline{s}_{i-1,j}+s_{i,j}, & \textrm{if }j\in[0:i-2]
\end{cases}\label{eq:TriLinSys-General}
\end{equation}
for $0\leq j<i\leq N$. Write $s=\{s_{i,j}\}_{0\leq j<i\leq N}$ and
$\overline{s}=\{\overline{s}_{i,j}\}_{0\leq j<i\leq N}$ to denote
the collection of stepsizes. The stepsizes $s$ or $\overline{s}$
may depend on the function class $\mathcal{F}$ and the value of $N$,
but are otherwise predetermined. In particular, they may not depend
on function values or gradients observed throughout the method. Write
$\mathcal{M}_{N}$ to denote the set of all FSFOMs with $N$ steps.
We will soon formulate the problem of finding
an optimal FSFOM in $\mathcal{M}_{N}$ as an optimization problem itself, and the stepsizes $s$ or $\overline{s}$ will serve as the decision variables.

The notion of fixed-step \emph{linear} first-order methods extend these definitions to accommodate proximal
methods and conditional gradient methods \cite[pp. 118-119]{taylor2017convex}.
Our BnB-PEP methodology also directly applies to these generalizations,
but we restrict our discussion to FSFOMs for the sake of simplicity.

\paragraph{Performance measure $\mathcal{E}$ and initial condition $\mathcal{C}$.}

For notational convenience, define the index sets
\begin{equation*}
I_{N}=\{0,1,\ldots,N\},
\qquad
I_{N}^{\star}=\{0,1,\ldots,N,\star\}.
\end{equation*}
Throughout this paper, we will use $\star$ as the index corresponding
to the optimal point. Write $\mathcal{E}$ to denote the performance
measure that evaluates a method $M\in\mathcal{M}_{N}$ on a specific
function $f\in\mathcal{F}$ with a starting point $x_{0}$. We require
that $\mathcal{E}$ depends only on iterates $\left\{ x_{0},\ldots,x_{N}\right\} $,
a globally optimal solution $x_{\star}$ to \eqref{eq:main-opt},
and the values and (sub)gradients of $f$ at the points $x_{0},x_{1},\ldots,x_{N},x_{\star}$. In other words, $\mathcal{E}$ may depend on the
solution $x_{\star}$ and zero- and first-order information the FSFOM
observes, but may not depend on other unobserved information of $f$.
Commonly considered performance measures are 
\[
\mathcal{E}\left(\{x_{i},\sgf(x_{i}),f(x_{i})\}_{i\in I_{N}^{\star}}\right)=f(x_{N})-f(x_{\star})
\]
or 
\[
\mathcal{E}\left(\{x_{i},\nabla f(x_{i}),f(x_{i})\}_{i\in I_{N}^{\star}}\right)=\|\nabla f(x_{N})\|^{2}
\]
when $f$ is differentiable.

To obtain a meaningful rate on the methods, we impose a suitable condition
on the initial iterate $x_{0}$, which we abstractly express as 
\begin{equation*}
\mathcal{C}\left(\{x_{i},\sgf(x_{i}),f(x_{i})\}_{i\in I_{N}^{\star}}\right)\leq0.
\end{equation*}
Commonly considered initial conditions are 
\[
\mathcal{C}\left(\{x_{i},\sgf(x_{i}),f(x_{i})\}_{i\in I_{N}^{\star}}\right)=\|x_{0}-x_{\star}\|^{2}-R^{2}
\]
or 
\[
\mathcal{C}\left(\{x_{i},\sgf(x_{i}),f(x_{i})\}_{i\in I_{N}^{\star}}\right)=f(x_{0})-f(x_{\star})-R^{2},
\]
where $R>0$.

\paragraph{Worst-case performance $\mathcal{R}$.}

The worst-case performance or the rate of the method $M\in\mathcal{M}_{N}$
is obtained by maximizing $\mathcal{E}$ over functions in $\mathcal{F}$.
More formally, we define 
\begin{align}
 & \mathcal{R}\left(M,\mathcal{E},\mathcal{F},\mathcal{C}\right)\nonumber \\
 & =\left(\begin{array}{ll}
\textrm{maximize}\quad\mathcal{E}\left(\{x_{i},\sgf(x_{i}),f(x_{i})\}_{i\in I_{N}^{\star}}\right)\\
\textrm{subject to}\\
f\in\mathcal{F},\\
x_{\star}\textrm{ is a globally optimal solution to }\eqref{eq:main-opt},\\
\{x_{i}\}_{i\in[1:N]}\textrm{ {is generated by FSFOM \ensuremath{M} }with initial point \ensuremath{x_{0}}},\\
\mathcal{C}\left(\{x_{i},\sgf(x_{i}),f(x_{i})\}_{i\in I_{N}^{\star}}\right)\leq0,
\end{array}\right)\tag{\ensuremath{\mathcal{O}^{\textrm{inner}}}}\label{eq:worst-case-pfm}
\end{align}
where $f$, $x_{0},\ldots,x_{N}$, and $x_{\star}$ are the decision variables.
We set $x_{\star}=0$ and $f(x_{\star})=0$, which incurs no loss
of generality because the function classes in Table~\ref{tab:Considered-function-classes}
and the FSFOM in consideration are closed and invariant under shifting
variables and function values.

We will soon show that evaluating the worst-case performance of a
given method $M\in\mathcal{M}_{N}$ by solving \eqref{eq:worst-case-pfm}
can be represented as a (finite-dimensional convex) semidefinite-program.

\paragraph{Optimal FSFOM.}
An \emph{optimal} FSFOM $M^{\star}_N\in\mathcal{M}_{N}$ for a given
performance measure $\mathcal{E}$ over function class $\mathcal{F}$
subject to the initial condition $\mathcal{C}$ is a solution to the
following minimax optimization problem: 
\begin{equation}
\begin{array}{ll}
\mathcal{R}^{\star}\left(\mathcal{M}_{N},\mathcal{E},\mathcal{F},\mathcal{C}\right)=\underset{M\in\mathcal{M}_{N}}{\mbox{minimize}} & \mathcal{R}\left(M,\mathcal{E},\mathcal{F},\mathcal{C}\right).\end{array}\tag{\ensuremath{\mathcal{O}^{\textrm{outer}}}}\label{eq:main-min-max-problem}
\end{equation}

Finding an optimal FSFOM $M^{\star}_N\in\mathcal{M}_{N}$
by solving \eqref{eq:main-min-max-problem} is, in general, a nonconvex
problem. The BnB-PEP methodology formulates \eqref{eq:main-min-max-problem}
as a (nonconvex) QCQP and solves it to certifiable
global optimality using a spatial branch-and-bound algorithm
in a practically tractable manner. 
In $\mathsection$\ref{sec:QCQO-framework-for-template-problem}, we illustrate our methodology by describing
it for a concrete problem instance.
In $\mathsection$\ref{subsec:Generalization-of-QCQO-framework}, we present the general form of the the methodology.

%% file: Sections/section_2.tex
\section{BnB-PEP for strongly convex smooth minimization \label{sec:QCQO-framework-for-template-problem}}

This section demonstrates the BnB-PEP methodology on a concrete instance
for which prior methodologies do not apply. The general BnB-PEP methodology
is presented in $\mathsection$\ref{subsec:Generalization-of-QCQO-framework}.

Specifically, we find the optimal FSFOM for reducing the gradient
of $\mu$-strongly convex $L$-smooth functions, with $0\leq\mu<L\leq\infty$.
In other words, we choose the function class $\mathcal{F}=\mathcal{F}_{\mu,L}$
and performance measure $\mathcal{E}=\|\gf(x_{N})\|^{2}$. Further,
we choose the initial condition $\mathcal{C}=\|x_{0}-x_{\star}\|^{2}-R^{2}\leq0$
with $R>0$. Then, an optimal FSFOM is a solution of the following
instance of \eqref{eq:main-min-max-problem}: 
\[
\begin{array}{ll}
\mathcal{R}^{\star}\left(\mathcal{M}_{N},\mathcal{E},\mathcal{F},\mathcal{C}\right)=\underset{M\in\mathcal{M}_{N}}{\mbox{minimize}} & \mathcal{R}\left(M,\mathcal{E},\mathcal{F},\mathcal{C}\right).\end{array}
\]

\subsection{Optimal optimization method from BnB-PEP-QCQP
\label{subsec:Converting-into-QCQO-gradient-reduction-F-mu-L}}

We formulate the outer problem as a (nonconvex) QCQP, which we refer to as the BnB-PEP-QCQP, in the following two
steps. In $\mathsection$\ref{par:Converting-the-inner-problem-into-SDP},
we formulate the inner problem \eqref{eq:worst-case-pfm} as a convex
SDP. This first step follows the approach of \cite{drori2014performance,taylor2017exact}.
Then in $\mathsection$\ref{subsec:Converting-minmax-into-a-bnb-pep}
formulates the outer problem \eqref{eq:main-min-max-problem} as a
QCQP. This second step is is novel.

\subsubsection{Formulating the inner problem \eqref{eq:worst-case-pfm} as a convex
SDP \label{par:Converting-the-inner-problem-into-SDP}}

\paragraph{Infinite-dimensional inner optimization problem.}

By setting $s_{i,j}=\frac{h_{i,j}}{L}$ in \eqref{eq:fsfom}, parameterize FSFOMs in $\mathcal{M}_{N}$ as 
\begin{equation}
x_{i}=x_{i-1}-\frac{1}{L}\sum_{j=0}^{i-1}h_{i,j}\sgf(x_{j})\label{eq:FSFO-tidle}
\end{equation}
for $i\in[1:N]$. Write $h=\{h_{i,j}\}_{0\leq j<i\leq N}$.
Then \eqref{eq:worst-case-pfm} becomes 
\begin{align}
 & \mathcal{R}(M,\mathcal{E},\mathcal{F},\mathcal{C})\nonumber \\
 & =\left(\begin{array}{ll}
\textrm{maximize} & \|\gf(x_{N})\|^{2}\\
\textrm{subject to} & f\in\LSmthMuStCvx,\\
 & \nabla f(x_{\star})=0,\\
 & x_{i}=x_{i-1}-\frac{1}{L}\sum_{j=0}^{i-1}h_{i,j}\nabla f(x_{j}),\quad i\in[1:N],\\
 & \|x_{0}-x_{\star}\|^{2}\leq R^{2},\\
 & x_{\star}=0,\;f(x_{\star})=0,
\end{array}\right)\label{eq:worst-case-pfm-fmuL-1}
\end{align}
where $f,x_{0},\ldots,x_{N}$ are the decision variables. As is, $f$
is an infinite-dimensional decision variable.

\paragraph{Reparametrization from $\mathcal{F}_{\mu,L}$ to $\mathcal{F}_{0,L-\mu}$.}

Next, we use the following lemma to reparameterize \eqref{eq:worst-case-pfm},
defined with function class $\LSmthMuStCvx$, into an equivalent problem
with the $(L-\mu)$-smooth convex function class $\LMinusMuSmthCvx$.
The benefit of the reparametrization is that the final problem becomes
more compact.

\begin{lem}[{Reparametrization from $\mathcal{F}_{\mu,L}$ to $\mathcal{F}_{0,L-\mu}$\label{Reparametrization-lemma}
\cite[$\mathsection$3.2]{taylor2021optimal}}] Consider $f\in\LSmthMuStCvx$
where $0\leq\mu\le L\leq\infty$ with a minimizer $x_{\star}$. Consider
an FSFOM with $f$ and $\{h_{i,j}\}_{0\leq j<i\leq N}$
as defined in \eqref{eq:FSFO-tidle}. Define $\widetilde{f}\coloneqq f-(\mu/2)\|\cdot-x_{\star}\|^{2}$
and an array of parameters $\{\alpha_{i,j}\}_{0\leq j<i\leq N}$ 
\[
\alpha_{i,j}=\begin{cases}
h_{i,i-1}, & \textrm{if }j=i-1,\\
\alpha_{i-1,j}+h_{i,j}-\frac{\mu}{L}\sum_{k=j+1}^{i-1}h_{i,k}\alpha_{k,j}, & \textrm{if }j\in[0:i-2],
\end{cases}
\]
where $i\in[1:N]$ and $j\in[0:i-1]$. Then (i) $\widetilde{f}\in\LMinusMuSmthCvx$
if and only if $f\in\LSmthMuStCvx$, (ii) $x_{\star} \in \argmin\widetilde{f}$,
and (iii) the FSFOM \eqref{eq:FSFO-tidle} is equivalent to 
\begin{align*}
x_{i} & =x_{\star}+(x_{0}-x_{\star})\left(1-\frac{\mu}{L}\sum_{j=0}^{i-1}\alpha_{i,j}\right)-\sum_{j=0}^{i-1}\frac{\alpha_{i,j}}{L}\nabla\widetilde{f}(x_{j})
\end{align*}
for $i\in[1:N]$. \end{lem}

\paragraph{Reformulated infinite-dimensional maximization problem.}

Using Lemma~\ref{Reparametrization-lemma}, we reformulate \eqref{eq:worst-case-pfm}
as 
\begin{align*}
 & \mathcal{R}(M,\mathcal{E},\mathcal{F},\mathcal{C})\nonumber \\*
= & \left(\begin{array}{l}
\textrm{maximize}\quad\|\nabla\widetilde{f}(x_{N})\|^{2}+\mu^{2}\|x_{N}-x_{\star}\|^{2}+2\mu\left\langle \nabla\widetilde{f}(x_{N})\mid x_{N}-x_{\star}\right\rangle \\
\textrm{subject to}\\
\widetilde{f}\in\LMinusMuSmthCvx,\\
\nabla\widetilde{f}(x_{\star})=0,\\
x_{i}=x_{0}\left(1-\frac{\mu}{L}\sum_{j=0}^{i-1}\alpha_{i,j}\right)-\frac{1}{L}\sum_{j=0}^{i-1}\alpha_{i,j}\nabla \widetilde{f}(x_{j}),\quad i\in[1:N],\\
\|x_{0}-x_{\star}\|^{2}\leq R^{2},\\
x_{\star}=0,\widetilde{f}(x_{\star})=0,
\end{array}\right)
\end{align*}
where $\widetilde{f},x_{0},\ldots,x_{N}$ are the decision variables.
The decision variable $\widetilde{f}\in\LMinusMuSmthCvx$ is still
infinite-dimensional. Write $\alpha=\{\alpha_{i,j}\}_{0\leq j<i\leq N}$.

\paragraph{Interpolation argument.}

We now convert the infinite-dimensional optimization problem into
finite-dimensional one with the following lemma. 

\begin{lem}[{$\LSmthCvx$-interpolation \cite[Theorem 2]{taylor2021optimal}\label{Thm:Interpolation-inequality-FmuL}}]
Let $I$ be an index set,
and let $\{(x_{i},g_{i},f_{i})\}_{i\in I}\subseteq\mathbb{R}^{d}\times\mathbb{R}^{d}\times\mathbb{R}$.
Let $L>0$.
There exists $f\in\LSmthCvx$ satisfying
$f(x_{i})=f_{i}$ and $g_{i}\in\partial f(x_{i})$ for all $i\in I$
if and only if %
\footnote{
This can be viewed as a discretization of the following condition \cite[Theorem 2.1.5, Equation (2.1.10)]{nesterov2003introductory}:
$f\in\LSmthCvx$ if and only if
\begin{align*}
\quad f(y)\geq f(x)+ & \langle\nabla f(x)\mid y-x\rangle+\frac{1}{2L}\|\nabla f(x)-\nabla f(y)\|^{2},\quad
\forall x,y\in\rl^{d}.
\end{align*}
}
\begin{align*}
f_{i} & \geq f_{j}+\langle g_{j}\mid x_{i}-x_{j}\rangle+\frac{1}{2L}\|g_{i}-g_{j}\|^{2},\quad\forall\,i,j\in I.
\end{align*}
\end{lem}

\paragraph{Finite-dimensional maximization problem.}
Using Lemma~\ref{Thm:Interpolation-inequality-FmuL}, reformulate
\eqref{eq:worst-case-pfm} as 
\begin{align}
 & \mathcal{R}(M,\mathcal{E},\mathcal{F},\mathcal{C}) \nonumber \\
= & \left(\begin{array}{l}
\textrm{maximize}\quad\|g_{N}\|^{2}+\mu^{2}\|x_{N}-x_{\star}\|^{2}-2\mu\left\langle g_{N}\mid x_{\star}-x_{N}\right\rangle \\
\textrm{subject to}\\
f_{i}\geq f_{j}+\langle g_{j}\mid x_{i}-x_{j}\rangle+\frac{1}{2(L-\mu)}\|g_{i}-g_{j}\|^{2},\quad i,j\in I_{N}^{\star}:i\neq j,\\
g_{\star}=0,x_{\star}=0,f_{\star}=0,\\
x_{i}=x_{0}\left(1-\frac{\mu}{L}\sum_{j=0}^{i-1}\alpha_{i,j}\right)-\frac{1}{L}\sum_{j=0}^{i-1}\alpha_{i,j}g_{j},\quad i\in[1:N],\\
\|x_{0}-x_{\star}\|^{2}\leq R^{2}.
\end{array}\right)\label{eq:worst-case-pfm-fmuL-3}
\end{align}
Now, the decision variables are $\{x_{i},g_{i},f_{i}\}_{i\in I_{N}^{\star}}\subseteq\rl^{d}\times\rl^{d}\times\rl$,
where $I_{N}^{\star}=\{0,1,\ldots,N,\star\}$, and the optimization
problem is finite-dimensional, although nonconvex. To clarify, we
are applying Lemma~\ref{Thm:Interpolation-inequality-FmuL} on $\widetilde{f}\in\LMinusMuSmthCvx$,
rather than $f\in\LSmthMuStCvx$. However, we use the symbols $\{f_{i}\}_{i\in I_{N}^{\star}}$,
rather than the arguably more consistent $\{\widetilde{f}_{i}\}_{i\in I_{N}^{\star}}$ for the sake of notational conciseness.

\paragraph{Grammian formulation.}

Next, we formulate \eqref{eq:worst-case-pfm} as a convex SDP. Let
\begin{equation}
\begin{alignedat}{1}H & =[x_{0}\mid g_{0}\mid g_{1}\mid\ldots\mid g_{N}]\in\rl^{d\times(N+2)},\\
G & =H^{\top}H\in\mathbb{S}_{+}^{N+2},\\
F & =[f_{0}\mid f_{1}\mid\ldots\mid f_{N}]\in\rl^{1\times(N+1)}.
\end{alignedat}
\label{eq:grammian-mats}
\end{equation}
Note that $\rank G\leq d$. Define the following notation for selecting
columns and elements of $H$ and $F$: 
\begin{equation}
\begin{alignedat}{1} & \mathbf{g}_{\star}=0\in\rl^{N+2},\;\mathbf{g}_{i}=e_{i+2}\in\rl^{N+2},\quad i\in[0:N]\\
 & \mathbf{x}_{0}=e_{1}\in\rl^{N+2},\;\mathbf{x}_{\star}=0\in\rl^{N+2},\\
 & \mathbf{x}_{i}=\mathbf{x}_{0}\left(1-\frac{\mu}{L}\sum_{j=0}^{i-1}\alpha_{i,j}\right)-\frac{1}{L}\sum_{j=0}^{i-1}\alpha_{i,j}\mathbf{g}_{j}\in\rl^{N+2},\quad i\in[1:N]\\
 & \mathbf{f}_{\star}=0\in\rl^{N+1},\;\mathbf{f}_{i}=e_{i+1}\in\rl^{N+1},\quad i\in[0:N].
\end{alignedat}
\label{eq:bold-vec}
\end{equation}
This notation is defined so that 
\[
x_{i}=H\mathbf{x}_{i},\;g_{i}=H\mathbf{g}_{i},\;f_{i}=F\mathbf{f}_{i}
\]
for $i\in I_{N}^{\star}$. Note that $\mathbf{x}_{i}$ depends on
$\{\alpha_{i,j}\}_{j\in[0:i-1]}$ linearly for $i\in[1:N]$. Furthermore,
for $i,j\in I_{N}^{\star}$, define 
\begin{equation}
\begin{alignedat}{1} & A_{i,j}(\alpha)=\mathbf{g}_{j}\odot(\mathbf{x}_{i}-\mathbf{x}_{j})\in\mathbb{S}^{N+2},\\
 & B_{i,j}(\alpha)=(\mathbf{x}_{i}-\mathbf{x}_{j})\odot(\mathbf{x}_{i}-\mathbf{x}_{j})\in\mathbb{S}_{+}^{N+2},\\
 & C_{i,j}=(\mathbf{g}_{i}-\mathbf{g}_{j})\odot(\mathbf{g}_{i}-\mathbf{g}_{j})\in\mathbb{S}_{+}^{N+2},\\
 & a_{i,j}=\mathbf{f}_{j}-\mathbf{f}_{i}\in\mathbb{R}^{N+1}.
\end{alignedat}
\label{eq:ABCa-mat-vec}
\end{equation}
Note that $A_{i,j}(\alpha)$ is affine and $B_{i,j}(\alpha)$ is quadratic
as functions of $\{\alpha_{i,j}\}_{i\in[1:N],j\in[0,i-1]}$. This
notation is defined so that 
\begin{equation}
\begin{alignedat}{1} & \left\langle g_{j}\mid x_{i}-x_{j}\right\rangle =\tr GA_{i,j}(\alpha),\\
 & \|x_{i}-x_{j}\|^{2}=\tr GB_{i,j}(\alpha),\\
 & \|g_{i}-g_{j}\|^{2}=\tr GC_{i,j},
\end{alignedat}
\label{eq:effect-of-grammian-trs}
\end{equation}
for $i,j\in I_{N}^{\star}$.

Using this notation, formulate \eqref{eq:worst-case-pfm} as 
\begin{align*}
 & \mathcal{R}(M,\mathcal{E},\mathcal{F},\mathcal{C})\\
= & \left(\begin{array}{l}
\textrm{maximize}\quad\tr G\left(C_{N,\star}+\mu^{2}B_{N,\star}(\alpha)-2\mu A_{\star,N}(\alpha)\right)\\
\textrm{subject to}\\
Fa_{i,j}+\tr GA_{i,j}(\alpha)+\frac{1}{2(L-\mu)}\tr GC_{i,j}\leq0,\quad i,j\in I_{N}^{\star}:i\neq j,\\
-G\preceq0,\,\rank(G)\le d,\\
\tr GB_{0,\star}\leq R^{2},
\end{array}\right)
\end{align*}
where $F\in\rl^{1\times(N+1)}$ and $G\in\rl^{(N+2)\times(N+2)}$
are the decision variables. The equivalence relies on the fact that given a $G\in\mathbb{S}_{+}^{N+2}$
satisfying $\rank(G)\le d$, there exists a $H\in\rl^{d\times(N+2)}$
such that $G=H^{\top}H$. The argument is further detailed in \cite[$\mathsection$3.2]{taylor2017smooth}.
This formulation is not yet a convex SDP due to the rank constraint
$\rank(G)\le d$.

\paragraph{SDP representation.}

Next, we make the following large-scale assumption. 

\begin{assumption} \label{large-scale-assumption} We have $d\geq N+2$.
\end{assumption} 

Under this assumption, the constraint $\rank G\leq d$ becomes vacuous, since
$G\in\mathbb{S}_{+}^{(N+2)}$. We drop the rank constraint
and formulate \eqref{eq:worst-case-pfm} as a convex SDP 
\begin{align}
 & \mathcal{R}(M,\mathcal{E},\mathcal{F},\mathcal{C})\nonumber \\
= & \left(\begin{array}{l}
\textrm{maximize}\quad\tr G\left(C_{N,\star}+\mu^{2}B_{N,\star}(\alpha)-2\mu A_{\star,N}(\alpha)\right)\\
\textrm{subject to}\\
Fa_{i,j}+\tr GA_{i,j}(\alpha)+\frac{1}{2(L-\mu)}\tr GC_{i,j}\leq0,\quad i,j\in I_{N}^{\star}:i\neq j, \quad \rhd \textsf{\;dual var.\;} \lambda_{i,j}\geq0 \\
-G\preceq0, \quad \rhd \textsf{\;dual var.\;}  Z\succeq0 \\
\tr GB_{0,\star}\leq R^{2}, \quad \rhd \textsf{\;dual var.\;} \nu\geq0
\end{array}\right)\label{eq:worst-case-pfm-sdp-1}
\end{align}
where $F\in\rl^{1\times(N+1)}$ and $G\in\rl^{(N+2)\times(N+2)}$
are the decision variables. We denote the corresponding dual variables on the right hand side of the constraints with $\rhd \textsf{\;dual var.\;}$\ for later use.

 {
 We emphasize that dropping the rank constraint is not a relaxation;
the optimization problem \eqref{eq:worst-case-pfm-sdp-1} and its solution is independent of the dimension $d$, provided that the large-scale assumption $d\ge N+2$ holds. See \cite[$\mathsection$3.3]{taylor2017smooth} for further discussion.}

\paragraph{Dualization.}

Next we use convex duality to formulate \eqref{eq:worst-case-pfm},
originally a maximization problem, as a minimization problem. Take
the dual of \eqref{eq:worst-case-pfm-sdp-1} to get 
\begin{align}
 & \overline{\mathcal{R}}(M,\mathcal{E},\mathcal{F},\mathcal{C})\nonumber \\
 & =\left(\begin{array}{l}
\textrm{minimize}\quad\nu R^{2}\\
\textrm{subject to}\\
\sum_{i,j\in I_{N}^{\star}:i\neq j}\lambda_{i,j}a_{i,j}=0,\\
\nu B_{0,\star}-C_{N,\star}-\mu^{2}B_{N,\star}(\alpha)+2\mu A_{\star,N}(\alpha)+\\
\quad\sum_{i,j\in I_{N}^{\star}:i\neq j}\lambda_{i,j}\left(A_{i,j}(\alpha)+\frac{1}{2(L-\mu)}C_{i,j}\right)=Z,\\
Z\succeq0,\\
\nu\geq0,\;\lambda_{i,j}\geq0,\quad i,j\in I_{N}^{\star}:i\neq j,
\end{array}\right)\label{eq:worst-case-pfm-dual-1}
\end{align}
where $\nu\in\rl$, $\lambda=\{\lambda_{i,j}\}_{i,j\in I_{N}^{\star}:i\neq j}$,
and $Z\in\mathbb{S}_{+}^{N+2}$ are the decision variables. 
(Note, we write $\overline{\mathcal{R}}$ rather than $\mathcal{R}$ here.)
We call $\lambda,\nu,$ and $Z$ the \emph{inner-dual variables}.

By weak duality of convex SDPs, we have
\[
\mathcal{R}(M,\mathcal{E},\mathcal{F},\mathcal{C})\le \overline{\mathcal{R}}(M,\mathcal{E},\mathcal{F},\mathcal{C}).
\]
In convex SDPs, strong duality holds often but not always. For the
sake of simplicity, we assume strong duality holds. \begin{assumption}\label{strong-duality-assumption}
Strong duality holds between \eqref{eq:worst-case-pfm-sdp-1} and
\eqref{eq:worst-case-pfm-dual-1}, i.e.,
\[
\mathcal{R}(M,\mathcal{E},\mathcal{F},\mathcal{C})= \overline{\mathcal{R}}(M,\mathcal{E},\mathcal{F},\mathcal{C}).
\]
\end{assumption}

This assumption can be removed in most cases using the line of reasoning of \cite[Claim~4]{park2021}, but the technique is tedious. Since strong duality for convex SDPs ``usually'' holds, we argue there is little utility in pursuing this direction. Nevertheless, in a strict sense, the assumption constitutes a gap in our mathematical arguments. We detail the implication of this gap in $\mathsection$\ref{appendix:assumption2} of the appendix. In any case, strong duality holds for ``generic'' FSFOMs \cite[Theorem 6]{taylor2017smooth}, so one can safely use the BnB-PEP methodology with confidence that the obtained FSFOM will be optimal among the ``nice'' generic FSFOMs.

Now we have arrived at the end of the formulation based on prior works
\cite{drori2014performance,taylor2017exact}, and, at this point,
our formulations diverge.

\subsubsection{Formulating the outer problem \eqref{eq:main-min-max-problem} as
a QCQP \label{subsec:Converting-minmax-into-a-bnb-pep}}

With \eqref{eq:worst-case-pfm} formulated as a minimization problem,
the outer optimization problem \eqref{eq:main-min-max-problem} becomes
a joint minimization over the inner dual variables and the FSFOM parameters
$\alpha$. However, the outer minimization problem is not convex in
all of the variables, even though the inner problem is.

Prior work circumvents this nonconvexity within the scope of convex
optimization. Prior work on OGM \cite{drori2014performance,kim2016optimized},
ITEM \cite{taylor2021optimal}, ORC-F$_{\flat}$ \cite[$\mathsection$4]{park2021},
and OBL-F$_{\flat}$ \cite[$\mathsection$5.1]{park2021} convexify \eqref{eq:main-min-max-problem}
into an SDP through appropriate relaxations and changes of variables.
The relaxation discards certain inequalities, a process discussed
in $\mathsection$\ref{subsec:Sparsifier}. Due to the relaxation,
the solution FSFOM of the relaxed SDP is not necessarily the exact
optimal FSFOM. Exact optimality of OGM \cite{drori2017exact} and
ITEM \cite{drori2022oracle} were proved through separate exact matching
complexity lower bounds. Even though ORC-F$_{\flat}$ and OBL-F$_{\flat}$
are optimal solutions of the relaxed SDP, they are not optimal FSFOMs,
as their guarantees are worse than that of OGM. Prior work on OGM-G
\cite{kim2021optimizing}, M-OGM-G \cite{ZhouTianSoCheng2021_practical},
OBL-G$_{\flat}$ \cite{park2021}, APPM \cite{kim2021accelerated,lieder2021convergence},
and SM-APPM \cite{park2022} formulate \eqref{eq:main-min-max-problem}
as bi-convex optimization problems, as problems with bilinear matrix
inequalities (BMIs), after relaxations discarding certain inequalities.
Although bi-convex problems (which are nonconvex) do not have provably
efficient algorithms, prior work have obtained FSFOMs using the alternating
minimization heuristic. Exact optimality of APPM and SM-APPM were
proved through separate exact matching complexity lower bounds \cite{park2022}.
OGM-G is presumed but not proven to be exactly optimal.
M-OGM-G and OBL-G$_{\flat}$ are not optimal FSFOMs in the usual sense
as their guarantees are worse than that of OGM-G. For the setups we
consider, especially the setup of $\mathsection$\ref{subsec:Smooth-nonconvex-gradient-reduction}
and $\mathsection$\ref{subsec:pot-bnb-pep-wcvx-Moreau}, these prior
techniques do not apply and the optimization over the FSFOM cannot
be formulated as a convex nor a bi-convex optimization problem (to
the best of our knowledge) even after appropriate relaxations. 

\paragraph{Formulating \eqref{eq:main-min-max-problem} as a QCQP.}

The nonconvex outer optimization problem \eqref{eq:main-min-max-problem}
minimizes over $\alpha$ in addition to the inner dual variables of
\eqref{eq:worst-case-pfm-dual-1}. We confront the nonconvexity directly
by formulating \eqref{eq:main-min-max-problem} as a (nonconvex) QCQP
and solving it with spatial branch-and-bound algorithms. We do not
discard constraints or use any relaxation. To this end, we replace
the semidefinite constraint with a quadratic constraint via the Cholesky
factorization.

\begin{lem}[{\cite[Corollary 7.2.9]{horn2012matrix}\label{Lem:quadratic-characterization-psd-1}}]
A matrix $Z\in\mathbb{S}^{n}$ is positive semidefinite if and only
if it has a Cholesky factorization $PP^{\top}=Z$, where $P\in\rl^{n \times n}$
is lower triangular with nonnegative diagonals.\end{lem} 

In raw index form, the conditions of Lemma~\ref{Lem:quadratic-characterization-psd-1}, applied to the present setup,
have the following equivalent representations:
\begin{align*}
 & \left(\begin{array}{l}
P\textrm{ is lower triangular with nonnegative diagonals},\\
PP^{\top}=Z.
\end{array}\right)\\
\Leftrightarrow & \left(\begin{array}{l}
P_{j,j}\geq0,\quad j\in[1:N+2],\\
P_{i,j}=0,\quad 1\le i<j\le N+2,\\
\sum_{k=1}^{j}P_{i,k}P_{j,k}=Z_{i,j},\quad 1\le j\le i\le N+2.
\end{array}\right)
\end{align*}

We now formulate \eqref{eq:main-min-max-problem}, the problem of
finding an optimal FSFOM, as the following QCQP 
\begin{align}
 &  \mathcal{R}^{\star}(\mathcal{M}_{N},\mathcal{E},\mathcal{F},\mathcal{C})\nonumber \\
= & \left(\begin{array}{l}
\textrm{minimize}\quad\nu R^{2}\\
\textrm{subject to}\\
\sum_{i,j\in I_{N}^{\star}:i\neq j}\lambda_{i,j}a_{i,j}=0,\\
\nu B_{0,\star}-C_{N,\star}-\mu^{2}B_{N,\star}(\alpha)+2\mu A_{\star,N}(\alpha)+\\
\qquad\sum_{i,j\in I_{N}^{\star}:i\neq j}\lambda_{i,j}\left(A_{i,j}(\alpha)+\frac{1}{2(L-\mu)}C_{i,j}\right)=Z,\\
P\text{ is lower triangular with nonnegative diagonals},\\
PP^{\top}=Z,\\
\nu\geq0,\;\lambda_{i,j}\geq0,\quad i,j\in I_{N}^{\star}:i\neq j,
\end{array}\right)\label{eq:BnB-PEP-Preli}
\end{align}
where $\lambda$, $\nu$, $Z$, $P$, and $\alpha$ are the decision
variables. We name this optimization problem the BnB-PEP-QCQP. 

\subsection{Solving the BnB-PEP-QCQP using the BnB-PEP Algorithm \label{subsec:A-globally-optimal-alg-for-gradient-reduction-F-mu-L}}

We now solve the BnB-PEP-QCQP \eqref{eq:BnB-PEP-Preli} to certifiable optimality
in a practical time scale via the BnB-PEP Algorithm, stated as Algorithm~\ref{alg:Alg-1}.
The algorithm has 3 stages: Stage 1 finds a feasible point, Stage
2 uses an interior-point solver to find a locally optimal solution,
and Stage 3 uses a spatial branch-and-bound solver to find a globally
optimal solution.

\begin{algorithm}[h]

\textbf{Stage 1.} Compute a feasible solution. 
\begin{itemize}
\item Set $\alpha_{i,i-1}^{\textrm{init}}\leftarrow1$ and $\alpha_{i,j}^{\textrm{init}}\leftarrow0$
for for $i\in[1:N]$, $j\neq i-1$. 
\item Set $\alpha\leftarrow\alpha^{\textrm{init}}$ in \eqref{eq:worst-case-pfm-dual-1}
and solve the convex SDP. Denote the computed optimal solution to
\eqref{eq:worst-case-pfm-dual-1} by $\{\nu^{\textrm{init}},\lambda^{\textrm{init}},Z^{\textrm{init}}\}$. 
\item Compute Cholesky decomposition $Z^{\textrm{init}}=P^{\textrm{init}}(P^{\textrm{init}})^{\top}$. 
\end{itemize}
\textbf{Stage 2.} Compute a locally optimal solution by warm-starting
at Stage 1 solution. 
\begin{itemize}
\item Warm-start \eqref{eq:BnB-PEP-Preli} with $\{\alpha^{\textrm{init}},\nu^{\textrm{init}},\lambda^{\textrm{init}},Z^{\textrm{init}},P^{\textrm{init}}\}$
and solve it to local optimality using the nonlinear interior-point
method. Denote the solution by $\{\alpha^{\textrm{lopt}},\nu^{\textrm{lopt}},\lambda^{\textrm{lopt}},Z^{\textrm{lopt}},P^{\textrm{lopt}}\}$.
\end{itemize}
\textbf{Stage 3.} Compute a globally optimal solution by warm-starting
at Stage 2 solution. 
\begin{itemize}
\item Warm-start \eqref{eq:BnB-PEP-Preli} with $\{\alpha^{\textrm{lopt}},\nu^{\textrm{lopt}},\lambda^{\textrm{lopt}},Z^{\textrm{lopt}},P^{\textrm{lopt}}\}$
and solve it to global optimality using a customized spatial branch-and-bound
algorithm described in $\mathsection$\ref{sec:Efficient-implementation-and-enhancement}.
Denote the solution by $\{\alpha^{\star},\nu^{\star},\lambda^{\star},Z^{\star},P^{\star}\}$
and the optimal objective value by $p^{\star}$. 
\end{itemize}
\textbf{Return:} $\{\alpha^{\star},\nu^{\star},\lambda^{\star},Z^{\textrm{\ensuremath{\star}}},P^{\textrm{\ensuremath{\star}}}\}$
and $p^{\star}.$ \caption{BnB-PEP Algorithm: Given $(\mu,L,R)$, solves \eqref{eq:BnB-PEP-Preli}
to global optimality}
\label{alg:Alg-1} 
\end{algorithm}

\paragraph{Details of the BnB-PEP Algorithm.}

Stage 1 computes a feasible solution of \eqref{eq:BnB-PEP-Preli}
by taking advantage of the structure that \eqref{eq:worst-case-pfm-dual-1}
is a convex SDP when the FSFOM is fixed. We set the stepsize to represent
gradient descent (GD)
\[
x_{i}=x_{i-1}-\frac{1}{L}\nabla f(x_{i-1}),\quad i\in[1:N]%
\]
which is a suboptimal but reasonable algorithm. By Lemma \ref{Reparametrization-lemma},
this corresponds to $\alpha_{i,i-1}^{\textrm{init}}=1$ for $i\in[1:N]$
and $\alpha_{i,j}^{\textrm{init}}=0$ for $j\neq i-1$. %

Stage 2 computes a locally optimal solution to \eqref{eq:BnB-PEP-Preli}
using an interior-point algorithm, warm-starting at the feasible solution corresponding to GD that was
produced by Stage 1. When a good warm-starting point is provided,
interior-point algorithms can quickly converge to a locally optimal
solution \cite{byrd1997local,wachter2006implementation}, \cite[$\mathsection$3.3]{fiacco1990nonlinear}.
If the interior-point algorithm fails to converge, we return the
feasible solution from Stage 1. Fortunately, we observe that Stage
2 consistently provides a locally optimal solution.

Stage 3 computes a globally optimal solution using a customized spatial
branch-and-bound algorithm, warm-starting at the solution produced
by Stage 2. We detail our BnB-PEP-QCQP-specific customization in $\mathsection$\ref{sec:Efficient-implementation-and-enhancement}.
A good warm-starting point provides a tight upper bound of the optimal
value and, therefore, significantly accelerates the spatial branch-and-bound
algorithm of Stage 3. In our experience, Stage 2 often provides an
excellent, even nearly optimal, warm-starting point.

In principle, one can simply use an off-the-shelf implementation of
spatial branch-and-bound algorithm such as \texttt{Gurobi} \cite{Gurobi}
to solve \eqref{eq:BnB-PEP-Preli}. Spatial branch-and-bound algorithms
do compute globally optimal solutions in ``finite time'', but the
default implementation is impractically slow, as Table~\ref{tab:timing-bnb-pep-grad-red-FmuL}
illustrates. Customizing the spatial branch-and-bound algorithm with
problem-specific insights is essential, as discussed in $\mathsection$\ref{sec:Efficient-implementation-and-enhancement}.

\paragraph{Numerical results.}

We conduct numerical experiments on the computational environment described
in $\mathsection$\ref{subsec:computational_setup} with parameters $\mu=0.1$, $L=1$, and $R=1$. {
Due to the scale invariance discussed in \cite[$\mathsection$3.5]{taylor2017smooth}, it suffices to solve the BnB-PEP-QCQP for $L=1$ and $R=1$ and find the corresponding optimal stepsize vector $\alpha^\star$ (or $h^\star$) and the associated optimal worst-case performance measure $\| \nabla f(x_N) \|^2$. 
More specifically, for any other $L>0$ and $R>0$, the new optimal stepsize vector will be scaled as $\alpha^\star/L$ (or $h^\star/L$) with corresponding performance measure scaled as $L^2 R^2 \| \nabla f(x_N) \|^2$. The homogeneity relations for other performance measures and initial conditions can be found at \cite[$\mathsection$3.5]{taylor2017smooth} and \cite[$\mathsection$4.2.5]{taylor2017convex}.\footnote{The journal version of \cite[$\mathsection$3.5]{taylor2017smooth} contained a typo in the homogeneity relations, which was later corrected in a subsequent arXiv update \url{https://arxiv.org/pdf/1502.05666.pdf}.}}      
Tables~\ref{tab:BnB-PEP-result-1} and \ref{tab:Globally-optimal-stepsize-FmuL-grad-red}
present the results of the BnB-PEP Algorithm. The optimal algorithm indeed
outperforms other known algorithms in terms of the worst-case performance measure $\|\nabla{f(x_N)}\|^2$ with initial condition $\|x_0 - x_\star \| \leq R^2$. Table~\ref{tab:timing-bnb-pep-grad-red-FmuL}
presents runtimes of the BnB-PEP Algorithm. The BnB-PEP-QCQP can be solved in a practical time scale
with the BnB-PEP Algorithm, but only when we use the customized spatial branch-and-bound
solver of $\mathsection$\ref{sec:Efficient-implementation-and-enhancement}.

A notable empirical observation is that the FSFOM produced by Stage
2, expected to be locally optimal, is often globally optimal or near-optimal.
In this case, Stage 3 serves mainly to certify the global optimality
of the warm-starting solution of Stage 2. This fortuitous behavior
was observed consistently in our experiments of $\mathsection$\ref{sec:applications}
as well. Because Stages 1 and 2 tend to be faster than Stage 3 and
because the output of Stage 2 is often globally optimal, one can use
the output of Stage 2 as a heuristic without running Stage 3.
This can be useful when the goal is to obtain a good method, and there is no need to certify that the method is optimal.

Table~\ref{tab:BnB-PEP-result-1} compares the performance of the
optimal method, obtained with the BnB-PEP Algorithm, against {GD (plain gradient desent)}, ITEM \cite{taylor2021optimal}, and
OGM-$\mathcal{F}_{\mu,L}$ \cite{taylor2021optimal}. 
(ITEM and OGM-$\mathcal{F}_{\mu,L}$ are optimal with respect to different performance measures and therefore are suboptimal when the goal is to reduce the gradient magnitude.
The stepsizes of ITEM were taken from page 21
of the first arXiv version of \cite{taylor2021optimal}, and the stepsizes
of OGM-$\mathcal{F}_{\mu,L}$ were taken from \cite[$\mathsection$E.1]{taylor2021optimal}.) We also show the total number of variables and constraints of \eqref{eq:BnB-PEP-Preli} that the BnB-PEP Algorithm works with after the mathematical model described in \texttt{JuMP} gets converted to the \texttt{MathOptInterface} format \cite{Legat2021}, which is the standard data structure for representing optimization models in \texttt{JuMP}. 

Table~\ref{tab:Globally-optimal-stepsize-FmuL-grad-red} shows the globally optimal stepsizes found by the BnB-PEP Algorithm. To clarify, we obtain the optimal $\alpha^{\star}$ from the BnB-PEP Algorithm,
solve for $h^{\star}$ with Lemma~\ref{Reparametrization-lemma},
and present $h^{\star}$ in the table.

\begin{table}
\begin{centering}
\begin{tabular}{>{\centering}b{2em}>{\centering}b{4em}>{\centering}b{4em}>{\centering}b{5em}>{\centering}b{5em}>{\centering}b{5em}>{\centering}b{5em}}
\toprule 
\multirow{2}{2em}{\centering{}{\footnotesize{}$N$}} & \multirow{2}{4em}{\centering{}{\footnotesize{}\# variables}} & \multirow{2}{4em}{\centering{}{\footnotesize{}\# constraints}} & \multicolumn{4}{c}{{\footnotesize{}Worst-case $\|\nabla f(x_{N})\|^{2}$}}\tabularnewline
\cmidrule{4-7} \cmidrule{5-7} \cmidrule{6-7} \cmidrule{7-7} 
 &  &  & {\footnotesize{}Optimal  } & {\footnotesize{} {GD}  } & {\footnotesize{}ITEM } & {\footnotesize{}OGM-$\mathcal{F}_{\mu,L}$}\tabularnewline
\midrule
\midrule 
\centering{}{\footnotesize{}$1$  } & \centering{}{\footnotesize{}$20$} & \centering{}{\footnotesize{}$33$} & {\footnotesize{}$0.1473$  } & {\footnotesize{}$0.2244$  } & {\footnotesize{}$0.6695$  } & {\footnotesize{}$0.2122$}\tabularnewline
\midrule
\centering{}{\footnotesize{}$2$ } & \centering{}{\footnotesize{}$36$} & \centering{}{\footnotesize{}$56$} & {\footnotesize{}$0.0409$  } & {\footnotesize{}$0.0893$  } & {\footnotesize{}$0.3770$  } & {\footnotesize{}$0.0835$}\tabularnewline
\midrule
\centering{}{\footnotesize{}$3$ } & \centering{}{\footnotesize{}$57$} & \centering{}{\footnotesize{}$85$} & {\footnotesize{}$0.0145$  } & {\footnotesize{}$0.0449$  } & {\footnotesize{}$0.1933$  } & {\footnotesize{}$0.0378$}\tabularnewline
\midrule
\centering{}{\footnotesize{}$4$ } & \centering{}{\footnotesize{}$83$} & \centering{}{\footnotesize{}$120$} & {\footnotesize{}$0.005766$  } & {\footnotesize{}$0.0257$  } & {\footnotesize{}$0.0945$  } & {\footnotesize{}$0.0178$}\tabularnewline
\midrule
\centering{}{\footnotesize{}$5$  } & \centering{}{\footnotesize{}$114$} & \centering{}{\footnotesize{}$161$} & {\footnotesize{}$0.002459$  } & {\footnotesize{}$0.0159$  } & {\footnotesize{}$0.0451$  } & {\footnotesize{}$0.0085$}\tabularnewline
\midrule
\centering{}{\footnotesize{}$10$  } & \centering{}{\footnotesize{}$410$} & \centering{}{\footnotesize{}$456$} & {\footnotesize{}$4.89\times 10^{-5}$  } & {\footnotesize{}$2.58 \times 10^{-3}$  } & {\footnotesize{}$1.03 \times 10^{-3}$  } & {\footnotesize{}$1.97 \times 10^{-4}$}\tabularnewline
\midrule
\centering{}{\footnotesize{}$25$  } & \centering{}{\footnotesize{}$2135$} & \centering{}{\footnotesize{}$2241$} & {\footnotesize{}$ 5.42\times 10^{-10}$  } & {\footnotesize{}$5.89 \times 10^{-5}$  } & {\footnotesize{}$5.5 \times 10^{-7}$  } & {\footnotesize{}$7.21 \times 10^{-8}$}\tabularnewline
\bottomrule
\end{tabular}\caption{{Comparison of the optimal method obtained by solving \eqref{eq:BnB-PEP-Preli} with the BnB-PEP Algorithm
against other known methods.
\label{tab:BnB-PEP-result-1}}}
\par\end{centering}
\end{table}

\begin{table}

\begin{centering}
{\footnotesize{}}%
\begin{tabular}{cc}
\toprule 
{\footnotesize{}$N$} & {\footnotesize{}$h^{\star}$}\tabularnewline
\midrule
\midrule 
{\footnotesize{}$1$} & {\footnotesize{}$\left[\begin{array}{c}
1.3837\end{array}\right]$ }\tabularnewline
\midrule
{\footnotesize{}$2$} & {\footnotesize{}$\left[\begin{array}{cc}
1.5018\\
0.0494 & 1.5018
\end{array}\right]$ }\tabularnewline
\midrule
{\footnotesize{}$3$} & {\footnotesize{}$\left[\begin{array}{ccc}
1.5308\\
0.0889 & 1.7229\\
0.0109 & 0.0889 & 1.5308
\end{array}\right]$ }\tabularnewline
\midrule
{\footnotesize{}$4$} & {\footnotesize{}$\left[\begin{array}{cccc}
1.5403\\
0.1038 & 1.7926\\
0.0229 & 0.1751 & 1.7926\\
0.003 & 0.0229 & 0.1038 & 1.5403
\end{array}\right]$ }\tabularnewline
\midrule
{\footnotesize{}$5$} & {\footnotesize{}$\left[\begin{array}{ccccc}
1.5439\\
0.1097 & 1.8187\\
0.0286 & 0.2132 & 1.8842\\
0.0069 & 0.0514 & 0.2132 & 1.8187\\
0.0009 & 0.0069 & 0.0286 & 0.1097 & 1.5439
\end{array}\right]$}\tabularnewline
\midrule
{\footnotesize{}$10$} & {\footnotesize{} $\left[
\begin{array}{cccccccccc}
1.5465 \\
0.1141 & 1.8377 \\
0.033 & 0.2426 & 1.9488\\
0.0107 & 0.0786 & 0.3072 & 1.995\\
0.0036 & 0.0265 & 0.1037 & 0.3357 & 2.0122\\
0.0012 & 0.009 & 0.0352 & 0.114 & 0.3437 & 2.0122\\
0.0004 & 0.003 & 0.0117 & 0.0378 & 0.114 & 0.3357 & 1.995\\
0.0001 & 0.0009 & 0.0036 & 0.0117 & 0.0352 & 0.1037 & 0.3072 & 1.9488 \\
0.0 & 0.0002 & 0.0009 & 0.003 & 0.009 & 0.0265 & 0.0786 & 0.2426 & 1.8377 \\
0.0 & 0.0 & 0.0001 & 0.0004 & 0.0012 & 0.0036 & 0.0107 & 0.033 & 0.1141 & 1.5465 \\
\end{array}
\right]$} \tabularnewline
\midrule
{\footnotesize{}$25$} & {\footnotesize{} See Supplementary Information or \texttt{Github}  repository} \tabularnewline
\bottomrule
\end{tabular}\caption{{Globally optimal stepsizes obtained by solving \eqref{eq:BnB-PEP-Preli} with the BnB-PEP
Algorithm. \label{tab:Globally-optimal-stepsize-FmuL-grad-red}}}
\par\end{centering}
\end{table}

{Table~\ref{tab:timing-bnb-pep-grad-red-FmuL} presents the runtimes
with and without the customized spatial branch-and-bound solver of
$\mathsection$\ref{sec:Efficient-implementation-and-enhancement}.
The off-the-shelf spatial branch-and-bound algorithm of \texttt{Gurobi}
was very slow despite running on the \texttt{MIT Supercloud Computing
Cluster} with 24 Intel-Xeon-Platinum-8260 nodes (has 1152 cores) and
384 GB of RAM running Ubuntu 18.04.6 LTS with Linux 4.14.250-llgrid-10ms
kernel \cite{reuther2018interactive}. On the other hand, our BnB-PEP Algorithm
ran efficiently on both a standard laptop and on \texttt{MIT Supercloud}. For $N=25$, we run both the BnB-PEP Algorithm and the  off-the-shelf \texttt{Gurobi} on the \texttt{MIT Supercloud}. The cases for which the off-the-shelf spatial
branch-and-bound algorithm terminated, the results agreed with the
results of the BnB-PEP Algorithm.}

\begin{table}
\begin{centering}
{\footnotesize{}}%
\begin{tabular}{c>{\centering}p{5em}>{\centering}p{5em}>{\centering}p{5em}>{\centering}p{11em}}
\toprule 
\multirow{2}{*}{{\footnotesize{}Algorithm }} & \multicolumn{3}{c}{{\footnotesize{}BnB-PEP Algorithm runtime}} & \multirow{2}{11em}{\centering{}{\footnotesize{}Off-the-shelf }\texttt{\footnotesize{}Gurobi}{\footnotesize{}
runtime on}
\texttt{\footnotesize{}MIT Supercloud}}\tabularnewline
\cmidrule{2-4} \cmidrule{3-4} \cmidrule{4-4} 
 & \centering{}{\footnotesize{}Stage 1 } & \centering{}{\footnotesize{}Stage 2 } & \centering{}{\footnotesize{}Stage 3} & \tabularnewline
\midrule
\midrule 
{\footnotesize{}$N=1$  } & \centering{}{\footnotesize{}$0.004$ s } & \centering{}{\footnotesize{}$0.130$ s } & \centering{}{\footnotesize{}$0.081$ s  } & \centering{}{\footnotesize{}$5$ h $17$ m}\tabularnewline
\midrule 
{\footnotesize{}$N=2$  } & \centering{}{\footnotesize{}$0.007$ s  } & \centering{}{\footnotesize{}$0.147$ s  } & \centering{}{\footnotesize{}$0.110$ s  } & \centering{}{\footnotesize{}$1$ d $3$ h}\tabularnewline
\midrule 
{\footnotesize{}$N=3$  } & \centering{}{\footnotesize{}$0.007$ s  } & \centering{}{\footnotesize{}$0.153$ s  } & \centering{}{\footnotesize{}$0.512$ s } & \centering{}{\footnotesize{}$4$ d $13$ h}\tabularnewline
\midrule 
{\footnotesize{}$N=4$  } & \centering{}{\footnotesize{}$0.015$ s  } & \centering{}{\footnotesize{}$0.192$ s  } & \centering{}{\footnotesize{}$4.602$ s  } & \centering{}{\footnotesize{}More than a week}\tabularnewline
\midrule 
{\footnotesize{}$N=5$  } & \centering{}{\footnotesize{}$0.017$ s  } & \centering{}{\footnotesize{}$0.330$ s  } & \centering{}{\footnotesize{}$456.685$ s  } & \centering{}{\footnotesize{}More than a week}\tabularnewline
\midrule 
{\footnotesize{}$N=10$  } & \centering{}{\footnotesize{}
$0.26$ s  } & \centering{}{\footnotesize{}$2$ m $37$ s  } & \centering{}{\footnotesize{}
$1$ d $22$ h} & \centering{}{\footnotesize{}Does not finish in 2 weeks}\tabularnewline
\midrule 
{\footnotesize{}$N=25$  } & \centering{}{\footnotesize{}
$3.2$ s  } & \centering{}{\footnotesize{}$6$ m $22$ s  } & \centering{}{\footnotesize{}
$3$ d $10$ h}   & \centering{}{\footnotesize{}Does not finish in 2 weeks}\tabularnewline
\bottomrule
\end{tabular}{\footnotesize\par}
\par\end{centering}
\begin{centering}
 
\par\end{centering}
\caption{{This table compares the runtimes of the BnB-PEP Algorithm executed on a standard laptop with the off-the-shelf spatial
branch-and-bound algorithm of \texttt{Gurobi} executed on \texttt{MIT
Supercloud} for $N=1,\ldots, 5, 10$.
For the case $N=25$, both the BnB-PEP Algorithm and off-the-shelf \texttt{Gurobi} were executed on  \texttt{MIT Supercloud}. 
\label{tab:timing-bnb-pep-grad-red-FmuL}}}

\end{table}

\section{Efficient implementation of the BnB-PEP Algorithm \label{sec:Efficient-implementation-and-enhancement}}

As Table~\ref{tab:timing-bnb-pep-grad-red-FmuL} illustrates, an
off-the-shelf spatial branch-and-bound algorithm applied to BnB-PEP-QCQP
is very slow. In this section, we customize the spatial branch-and-bound
algorithm to exploit specific problem structure and obtain a speedup
that enables us to run the BnB-PEP Algorithm on a laptop.

In $\mathsection$\ref{subsec:Spatial-branch-and-bound-algorithm},
we briefly review the standard spatial branch-and-bound algorithm.
We present the customization that provide significant speedups in
$\mathsection$\ref{subsec:customized-bnb}. In $\mathsection$\ref{subsec:Sparsifier}
we show how to compute the effective index set of the inner-dual-variable and
thereby reduce the size of the BnB-PEP-QCQP without losing optimality.

\subsection{How the spatial branch-and-bound algorithm solves QCQPs \label{subsec:Spatial-branch-and-bound-algorithm}}

We briefly review the standard spatial branch-and-bound algorithm
\cite{Gurobi,locatelli2013global,horst2013global,liberti2008introduction}.
We assume the optimization problem admits a finite optimal value,
as this is the case in the setups we consider.\footnote{The setups of $\mathsection$\ref{sec:QCQO-framework-for-template-problem}
and $\mathsection$\ref{sec:applications} satisfy $p^{\star}<\infty$,
since any FSFOM (such as the method that has all $0$ stepsizes and
therefore does not move) achieves a finite performance measure, and
$0\le p^{\star}$, since the objectives are nonnegative. In general,
however, there could be pathological BnB-PEPs such that $p^{\star}=-\infty$
or $p^{\star}=\infty$.}

The spatial branch-and-bound algorithm uses a divide-and-conquer approach
to solve \eqref{eq:QCQO-template}. The algorithm starts with the
\emph{presolve} phase, solving a linear relaxation of \eqref{eq:QCQO-template}
to obtain valid bounds $l\leq x\leq u$, with $l,u\in\rl^{q}$, that
are satisfied by optimal solutions. Then the algorithm performs \emph{branching},
partitioning the feasible region of \eqref{eq:QCQO-template} into
a finite collection of subregions $F_{1},\ldots,F_{K}$ and considering
the subproblems 
\[
p_{k}^{\star}=\left(\begin{array}{ll}
\underset{x\in\rl^{q}}{\mbox{minimize}} & c^{\top}x+x^{\top}Q_{0}x\\
\textrm{subject to} & a_{i}^{\top}x+x^{\top}Q_{i}x\leq b_{i},\quad i\in[1:m],\\
 & a_{j}^{\top}x+x^{\top}Q_{j}x=b_{j},\quad j\in[m+1:p],\\
 & x\in F_{k},
\end{array}\right)
\]
for $k\in[1:K]$. The best (smallest) among the optimal values $p_{1}^{\star},\ldots,p_{K}^{\star}$
is $p^{\star}$, by definition. The \emph{bounding} part is about
how to efficiently solve these subproblems via solving relaxations
and how to split these subproblems into smaller subproblems if necessary;
we discuss this next. %

The central idea is that, while solving a particular subproblem (also
a QCQP albeit over a smaller region) might be as hard as solving the
original problem, a lower bound and an upper bound of that subproblem
is much easier to solve via linear relaxations. Using this idea, first,
at the root node of the spatial branch-and-bound tree, a linear relaxation
of \eqref{eq:QCQO-template} is constructed and solved, which gives
a lower bound on $p^{\star}$, denoted by $\underline{p}^{\star}$.
The tighter this relaxation, the closer $\underline{p}^{\star}$ is
to $p^{\star}$. In addition to that, the user can \emph{warm-start
}the branch-and-bound algorithm by providing a known initial feasible
solution to \eqref{eq:QCQO-template}, which gives an upper-bound
on $p^{\star},$ denoted by $\overline{p}^{\star}$. Efficient warm-starting
procedure that exploit the problem structure can massively speed up
branch-and-bound-solvers. The branch-and-bound algorithm during its
execution keeps updating $\overline{p}^{\star},\underline{p}^{\star}$.
The difference $\overline{p}^{\star}-\underline{p}^{\star}$ is called
the \emph{gap}, and when this gap is equal to zero (or less than some
tolerance $\epsilon$) at some point of the algorithm, we have found
the globally (near-)optimal value $p^{\star}$ of \eqref{eq:QCQO-template}
along with one (approximately) optimal solution, and the
algorithm is terminated.
We next discuss how the gap is improved over
the course of the algorithm.

Once the subregions have been created, the algorithm picks an active
subregion, say $F_{k}$ (which $k$ to select can be arbitrary, though,
in practice, it is usually done via different heuristics in modern
solvers), and constructs two linear optimization problems on $F_{k}$.
These linear formulations are constructed using the McCormick envelopes
\cite{mccormick1976computability}, which provide lower and upper
bounds for the quadratic objective and constraints in \eqref{eq:QCQO-template},
whereas the the linear constraints in \eqref{eq:QCQO-template} are
kept unaltered. Then three types of linear cuts are added to the linear
optimization problems to remove regions that are certain to not contain
any optimal solutions \cite{sherali1990hierarchy,padberg1989boolean,sherali2002enhancing}.
Solving these linear optimization problems along with the cuts provides
valid lower and upper bounds on the optimal values of \eqref{eq:QCQO-template}
for the active subregion $F_{k}$. Solving these linear optimization
problems on $F_{k}$ leads to one of the three possibilities below: 
\begin{enumerate}
\item If the linear optimization problem associated with the lower bound
is either infeasible or has an objective value greater than the global
upper bound $\overline{p}^{\star}$, then $F_{k}$ cannot contain
the optimal solution to \eqref{eq:QCQO-template}. Hence, without
solving the QCQP on $F_{k}$, we can discard or \emph{prune
}the subregion. Such a pruned subregion becomes a permanent \emph{leaf}
of the branch-and-bound tree. 
\item If both the lower and upper bounds for the subregion are the same,
then without directly solving the QCQP\textsf{ }on $F_{k}$, we have
found this subproblem's optimal solution with optimal value $p_{k}^{\star}$.
This optimal solution on $F_{k}$ is a feasible solution to the main
problem \eqref{eq:QCQO-template}. It is not necessary to branch on this subregion anymore, and it becomes a permanent leaf of the search tree. If the objective value associated with this new feasible solution leads to an improved upper bound $\overline{p}^\star$ compared to the current incumbent, then the feasible solution on $F_k$ becomes the new incumbent solution. Otherwise, updating the incumbent is not necessary and we simply proceed with the search.
\item If 1 or 2 does not happen, then the subregion $F_{k}$ is partitioned
into smaller subregions by branching again, which are then added to
the list of active subregions. 
\end{enumerate}
In addition to that, at any point, the algorithm keeps an updated
value of the lower bound on $\underline{p}^{\star}$ by taking the
minimum of the best objective values of all the current leaf nodes.
On the other hand, the upper bound $\overline{p}^{\star}$ corresponds
to the incumbent solution. As the algorithm explores the active subregions,
the gap $\overline{p}^{\star}-\underline{p}^{\star}$ keeps getting
smaller, and once it is zero or smaller than a certain tolerance $\epsilon$,
we have found the global optimal value $p^{\star}$ of \eqref{eq:QCQO-template}
along with one optimal solution subject to the tolerance, and the
algorithm terminates.

\subsection{Efficient implementation of the spatial branch-and-bound algorithm \label{subsec:customized-bnb}}

We now customize the spatial branch-and-bound algorithm to efficiently
exploit problem structure of the BnB-PEP-QCQP. Our customization of \texttt{Gurobi}'s
branch-and-bound algorithm \cite{Gurobi95} uses solver-independent
\emph{callback functions}, an interface provided by \texttt{JuMP}
\cite{JuMPDunningHuchetteLubin2017}. 

Callback functions are user-defined functions provided to the optimization solver that query or modify the state of the optimization process of a solver. Examples of such callback functions include providing custom heuristics to compute better feasible solutions, changing the default branching decision of the branch-and-bound algorithm, or applying on-demand separators to add new constraints only if they are violated by the current solution.

We discuss the generalization
of our customization of the spatial branch-and-bound algorithm for arbitrary $\mathcal{E}$, $\mathcal{F}$,
and $\mathcal{C}$ in $\mathsection$\ref{subsec:Generalization-of-QCQO-framework}.
We first present the customizations in $\mathsection$\ref{subsec:Compute-bounds},
$\mathsection$\ref{subsec:Computing-tighter-global-bound}, and $\mathsection$\ref{subsec:Custom-heuristic},
and then discuss the observed speedups in $\mathsection$\ref{subsec:ablation}.

\subsubsection{Bounds on optimal solutions\label{subsec:Compute-bounds}}

Branch and bound algorithms require bounds on the optimization variables
to partition the feasible region. If no bound information for a variable
is provided in the original formulation \eqref{eq:QCQO-template},
then the solver obtains a bound by solving a generic linear relaxation
during the presolve phase. However, this bound can be of poor quality
as a generic solver does not have any problem-specific insight, and
a loose bound can cause the solver to waste time in unimportant regions.
We show how to significantly speed up the branch-and-bound algorithm
by exploiting the structure of \eqref{eq:QCQO-template} to obtain
tighter bounds.

\paragraph{Implied linear constraints. }

The constraint $Z=PP^{\top}$ implies that $Z$ is symmetric positive
semidefinite. This in turn implies 
\begin{equation}
\begin{array}{l}
Z=Z^{\top},\\
\textrm{diag}(Z)\geq0,\\
-\frac{Z_{i,i}+Z_{j,j}}{2}\leq Z_{i,j}\leq\frac{Z_{i,i}+Z_{j,j}}{2}.
\end{array}\label{eq:implied_linear_constraints}
\end{equation}
where $\mathrm{diag}(Z)\ge0$ means the $Z_{i,i}\geq0$ for $i\in[1:N+2]$.
To explain, $Z\succeq0$ implies that every $1\times1$ principal
submatrix of $Z$ is positive-semidefinite \cite[Observation 7.1.2]{horn2012matrix},
and this in turn implies the second constraint. Also, $Z\succeq0$
implies that every the $2\times2$ principal submatrix of $Z$ is
positive-semidefinite, and this in turn implies 
\begin{equation}
\lvert Z_{i,j}\rvert\leq\sqrt{Z_{i,i}Z_{j,j}}\Leftrightarrow Z_{i,j}^{2}\leq Z_{i,i}Z_{j,j}\label{eq:soc_constraint}
\end{equation}
for $i,j\in[1:N+2]$. Chaining the AM-GM inequality 
\[
\sqrt{Z_{i,i}Z_{j,j}}\leq\frac{Z_{i,i}+Z_{j,j}}{2},
\]
we get the third constraint.

While these implied constraints are indeed mathematically redundant,
they are algorithmically indispensable as they provide crucial information
that the solver cannot deduce directly. Explicitly incorporating these
implied constraints provides significant speedups. Instead of incorporating the tighter convex second-order cone (SOC) constraint \eqref{eq:soc_constraint}, we opt for the third linear constraint \eqref{eq:implied_linear_constraints} in our BnB-PEP-QCQP formulation. This choice avoids a slowdown in the spatial branch-and-bound algorithm, which solves only \emph{linear} relaxations at each node, as detailed in $\mathsection$\ref{subsec:Spatial-branch-and-bound-algorithm}. The linear relaxations are derived from McCormick convex envelopes, which are constructed without considering underlying convexity and are computationally expensive \cite{locatelli2013global,horst2013global,liberti2008introduction}. Using the SOC constraint \eqref{eq:soc_constraint} would result in the spatial branch-and-bound algorithm treating it as a generic quadratic constraint and spending extra time constructing McCormick convex envelopes for it \cite[pp.\ 10--15]{Gurobi}. Since the positive semi-definiteness of $Z$ is already modeled by the quadratic constraints $Z=PP^{\top}$, and their associated convex envelopes are tighter than the ones for SOC constraints, the additional SOC constraints would ultimately lead to a net slowdown. Conversely, the third constraint in \eqref{eq:implied_linear_constraints} is linear and can be directly incorporated into the linear relaxations at the nodes without any extra processing time. These constraints differ from those automatically generated by the McCormick convex envelopes, ultimately resulting in a speed-up due to their low computational cost and provision of valuable bound information that is not automatically inferred through the McCormick convex envelopes. %

\paragraph{Variable bounds via SDP relaxation of \eqref{eq:BnB-PEP-Preli}.}

Next, we compute bounds $M_{\lambda}$, $M_{\nu}$, $M_{\alpha}$,
and $M_{Z}$ such that 
\begin{equation}
\begin{array}{l}
\lambda_{i,j}\leq M_{\lambda},\quad i,j\in I_{N}^{\star}:i\neq j,\\
\lvert Z_{i,j}\rvert\leq M_{Z},\quad i,j\in[1:N+2],\\
\lvert P_{i,j}\rvert\leq M_{P},\quad i,j\in[1:N+2],\\
\lvert\alpha_{i,j}\rvert\leq M_{\alpha},\quad i\in[1:N],\;j\in[0:i-1],\\
\nu\leq M_{\nu},
\end{array}\label{eq:individual-bound}
\end{equation}
are satisfied by global minimizers of \eqref{eq:BnB-PEP-Preli}.

Let $w=\textrm{vec}(\alpha,\nu,\lambda)$ denote the column vector
stacking the elements of $\alpha$, $\nu$, and $\lambda$. Let $W=ww^{\top}$.
Then we can construct a lifted nonconvex semidefinite representation
of the constraint set of \eqref{eq:BnB-PEP-Preli}, which includes
the nonconvex rank-1 constraint $W=ww^{\top}$ \cite{luo2010semidefinite}.
The specific form is quite tedious, so we present it in $\mathsection$\ref{subsec:Derivation-of-the-rank-1-constraint}
of the appendix. We then relax the rank-1 constraint $W=ww^{\top}$
to an implied convex constraint 
\begin{equation}
W\succeq ww^{\top}\Leftrightarrow\begin{bmatrix}W & w\\
w^{\top} & 1
\end{bmatrix}\succeq0,\label{eq:schur-complement}
\end{equation}
where we have used the Schur complement. Since any feasible (and optimal)
solution of \eqref{eq:BnB-PEP-Preli} must lie in this larger relaxed
convex set, we compute bounds by optimizing over this set as follows.

The feasible point provided by Stage 1 of the BnB-PEP Algorithm establishes
an upper bound $\nu\leq M_{\nu}=\nu^{\textrm{init}}$, since $\nu$
is the scaled objective function. Next, solve 
\begin{equation}
\left(\begin{array}{ll}
\textrm{maximize} & c_{\lambda}M_{\lambda}+c_{Z}M_{Z}+c_{\alpha}M_{\alpha}\\
\textrm{subject to} & \textrm{semidefinite relaxation of }\eqref{eq:BnB-PEP-Preli},\\
 & \textrm{constraint }\eqref{eq:schur-complement},\\
 & \lambda_{i,j}\leq M_{\lambda},\quad i,j\in I_{N}^{\star}:i\neq j,\\
 & \lvert Z_{i,j}\rvert\leq M_{Z},\quad i,j\in[1:N+2],\\
 & \lvert\alpha_{i,j}\rvert\leq M_{\alpha},\quad i\in[1:N],\;j\in[0:i-1],\\
 & \nu\leq\nu^{\textrm{init}},
\end{array}\right)\label{eq:computing_bounds}
\end{equation}
where $\lambda$, $\nu$, $Z$, $\alpha$, $W$, $M_{\lambda} \leq \|\lambda\|_1$, $M_{Z} \leq \tr Z$,
and $M_{\alpha} \leq \|\alpha\|_1$ are the decision variables, with
\[
(c_{\lambda},c_{Z},c_{\alpha})\in\{(1,0,0),(0,1,0),(0,0,1)\}
\]
to obtain $M_{\lambda}$, $M_{Z}$, and $M_{\alpha}$, respectively.\footnote{As a note of caution, solving \eqref{eq:computing_bounds} with $(c_{\lambda},c_{Z},c_{\alpha})=(1,1,1)$
does not provide a valid bound for all $M_{\lambda}$, $M_{Z}$, and
$M_{\alpha}$; maximizing $M_{\lambda}+M_{Z}+M_{\alpha}$ may reduce
one bound below a valid threshold to increase another bound.} (Since we restrict our search to points satisfying $\nu\leq\nu^{\textrm{init}}$,
our bounds may exclude some suboptimal feasible solutions. However,
all optimal solutions will satisfy the bound.) Finally, we set $M_{P}=\sqrt{M_{Z}}$
based on 
\begin{align}
P_{i,j}^{2} & \leq\sum_{k=1}^{i}P_{i,k}^{2}=Z_{ii}\leq M_{Z}
\label{eq:bound_on_P}
\end{align}
for all $i,j\in[1:N+2]$.

To clarify, this approach is a relaxation in the sense that it is guaranteed to produce variable bounds that will include all globally optimal solutions. (However, there is no guarantee on the tightness of the bounds, so the bounds could be very loose and not useful.)

Besides computing valid bounds on the variables, we also investigated the quality of the solutions of the SDP relaxations, for this setup and all other examples in this paper. Unfortunately, we found that the solutions of the SDP relaxations to be of very poor quality in every case: the optimal value of the SDP relaxation was far from the optimal value of the BnB-PEP-QCQP. Additionally, we observed that the SDP relaxations failed to generate feasible solutions for the underlying BnB-PEP-QCQPs, even when we considered a rank-1 projection of the solution matrix. In other words, it was not possible to reconstruct valid first-order methods from the solutions of the SDP relaxations.

\paragraph{Heuristic bounds.}

However, the SDP relaxation to compute the variable bounds is quite
cumbersome. Therefore, we present a simpler alternative, a heuristic
that estimates the variable bounds based on the Stage 2 solution.

The premise of the heuristic is as follows. First, we make the informal
assumption that the Stage 2 solution is near-optimal, which, again,
happened very often in our experiments. In $\mathsection$\ref{subsec:Sparsifier},
we discuss that optimal inner-dual variables $\lambda=\{\lambda_{i,j}\}_{i,j\in I_{N}^{\star}:i\neq j}$
and $Z$ are sometimes not unique and that sparse $\lambda$ and low-rank
$Z$ are more valuable. Following the literature on sparse signal
processing \cite[$\mathsection$2]{hastie2019statistical}, we promote
sparsity of $\lambda$ by reducing its $\ell_{1}$-norm $\sum_{i,j\in I_{N}^{\star}:i\neq j}\lambda_{i,j}$
and low rank of $Z$ by reducing its nuclear norm $\tr Z$. To do
so, we need our variable bounds to include the global solutions with
the minimum $\ell_{1}$-norm of $\lambda$ and minimum nuclear norm
of $Z$.

Based on the constraint set of \eqref{eq:worst-case-pfm-dual-1}, consider the following convex SDP 
\begin{equation}
\left(\begin{array}{ll}
\textrm{maximize} & c_{\lambda}\sum_{i,j\in I_{N}^{\star}:i\neq j}\lambda_{i,j}+c_{Z}\tr Z\\
\textrm{subject to} & \sum_{i,j\in I_{N}^{\star}:i\neq j}\lambda_{i,j}a_{i,j}=0,\\
 & \nu B_{0,\star}-C_{N,\star}-\mu^{2}B_{N,\star}(\alpha^{\textrm{lopt}})+2\mu A_{\star,N}(\alpha^{\textrm{lopt}})+\\
 & \quad\sum_{i,j\in I_{N}^{\star}:i\neq j}\lambda_{i,j}\left(A_{i,j}(\alpha^{\textrm{lopt}})+\frac{1}{2(L-\mu)}C_{i,j}\right)=Z,\\
 & Z\succeq0,\\
 & \lambda_{i,j}\geq0,\quad i,j\in I_{N}^{\star}:i\neq j,\\
 & \nu\geq0,\\
 & \nu R^{2}\leq\nu^{\textrm{lopt}}R^{2},
\end{array}\right)\label{eq:find-sparse-sol-1}
\end{equation}
where $\nu$, $\lambda$, and $Z$ are the decision variables, $(c_{\lambda},c_{Z})\in\{(1,0),(0,1)\}$,
and $\alpha^{\textrm{lopt}}$ are $\nu^{\textrm{lopt}}$ are set to
be values from the Stage 2 solution. Let $\widetilde{M}$ be a user-defined
parameter greater than $1$. With $(c_{\lambda},c_{Z})=(1,0)$, we
\emph{maximize} the $\ell_{1}$ norm of $\lambda$ and get $\lambda^{\textrm{hrstc}}$.
Set 
\[
M_{\lambda}=\widetilde{M}\max_{i,j\in I_{N}^{\star}:i\neq j}\{\lambda_{i,j}^{\textrm{hrstc}}\}.
\]
With $(c_{\lambda},c_{Z})=(0,1)$, we \emph{maximize} the nuclear
norm of $Z$ and get $Z^{\textrm{hrstc}}$. Set 
\[
M_{Z}=\widetilde{M}\max_{i\in[1:N+2]}\{Z_{i,i}^{\textrm{hrstc}}\}
\]
based on the reasoning that \eqref{eq:implied_linear_constraints}
implies that every entry of $Z$ is bounded by the maximum of the
diagonal entries. Set $M_{P}$ from $M_{Z}$ using \eqref{eq:bound_on_P}.
Set 
\[
M_{\alpha}= 5\widetilde{M}\max_{0\le j<i\le N}\{\alpha_{i,j}^{\textrm{lopt}}\}.
\]

A note of caution is that when the Stage 2 solution is far from optimal,
it is unclear whether this heuristic is even likely to produce a
valid bound. When the Stage 2 solution is, in fact, near-optimal,
the heuristic should help the BnB-PEP Algorithm to find the globally optimal
solution quickly, and this is what we observe in our experiments.

Another note of caution is that the heuristic fails silently when
it fails; there is no reliable mechanism to detect whether the heuristic
bounds include or exclude global solutions. After solving \eqref{eq:BnB-PEP-Preli}
using the bounds, we verify if the solution lies within the interior
of those bounds. Furthermore, empirically we always found that the solutions computed using the heuristic-based bound were well within the interior of the imposed bounds, though this is not a guarantee that the associated solution is globally optimal, as there is a possibility that a strictly better
global solution lies far outside of the boundary since the BnB-PEP-QCQP is nonconvex. Finally, in
all our experiments, we additionally verified that the heuristic-based bound produced the same optimal solutions as the SDP-based bounds. 

{
\textbf{Remark.}
We clarify that the heuristic bound offers no guarantee of correctness and that the SDP relaxation, which is guaranteed to be correct, is the superior choice. However, SDP relaxations can be cumbersome to formulate and implement. Therefore, one may first try out the heuristic bound in a prototyping phase and then decide to implement the SDP relaxation if the preliminary results are sufficiently interesting.
}

\subsubsection{Tighter lower bounds via lazy callback \label{subsec:Computing-tighter-global-bound}}

When an incumbent or warm-starting solution is already near-optimal,
i.e., when the upper bound is already good, the work in certifying
global optimality mostly lies in improving the lower bound. Indeed,
in our experiments, Stage 2 of the BnB-PEP Algorithm consistently found near-optimal solutions, and Stage 3 spent most of its time improving
the lower bound to certify or polish the solution of Stage 2.
If we can compute a good lower bound and provide it to the spatial branch-and-bound algorithm, Stage 3 can terminate very quickly as the number of subregions to be explored is substantially reduced. To that goal, we compute a tighter lower bound of \eqref{eq:BnB-PEP-Preli} via the \emph{lazy constraint callback method}. Unlike normal constraints, lazy constraints are not generated upfront but are rather generated and added one by one when needed.

{Consider a variant of \eqref{eq:BnB-PEP-Preli}, where we model $Z=PP^{\top} \Leftrightarrow Z \succeq 0$ equivalently as
\[
\tr\left(Zyy^{\top}\right)\geq0,\qquad \forall\,y\in \rl^{N+2}.
\]
Since this formulation uses an infinite set of linear constraints,
we relax it with a finite set of linear constraints
\begin{equation}
\tr\left(Zyy^{\top}\right)\geq0,\qquad \forall\,y\in Y,\label{eq:starting-set-lazy}
\end{equation}
where $Y$ is initialized to be a randomly generated
set of $2(N+2)^{2}$ unit vectors in $\rl^{N+2}$ following the prescription
of \cite[$\mathsection$5.1]{benson2003solving}.  
In \eqref{eq:BnB-PEP-Preli}, we relax $Z=PP^{\top}$ into the constraint \eqref{eq:starting-set-lazy} and obtain a simpler QCQP.} Then, update $Y$ lazily by
repeating the following steps (i)--(iii) a finite number of times ($1\times10^{6}$ times in our implementation): 
\begin{itemize}
\item[(i)] Solve the relaxation of \eqref{eq:BnB-PEP-Preli}, where \eqref{eq:starting-set-lazy} is used instead of $Z=PP^{\top}$, to obtain $Z$
and a lower bound. 
\item[(ii)] Find the minimum eigenvalue $\textrm{eig}_{\textrm{min}}(Z)$ and
corresponding normalized eigenvector $u$ of $Z$. If $\textrm{eig}_{\textrm{min}}(Z) \geq 0$,
terminate. 
\item[(iii)] If $\textrm{eig}_{\textrm{min}}(Z) < 0$, then add $u$ to $Y$, i.e., add the constraint $\tr(Zuu^{\top})\geq0$.
(Note, $\tr(Zuu^{\top})<0$. So the added constraint makes the current
$Z$ infeasible for the updated relaxation \eqref{eq:starting-set-lazy}.) 
\end{itemize}
In step (ii), if $\textrm{eig}_{\textrm{min}}(Z)\ge0$, then the solution
$Z$ of the relaxation \eqref{eq:starting-set-lazy} is in fact optimal
for the original unrelaxed problem, so we terminate. We use the lazy
constraint callback interface of \texttt{JuMP} to implement this scheme.
After adding one additional linear constraint in step (iii), updating
the solution in step (i) is efficient since \texttt{Gurobi} and all
modern solvers based on the simplex algorithm can quickly update a
solution when one linear constraint is added \cite[pp.\ 205--207]{bertsimas1997introduction}.

\subsubsection{Improved upper bounds via SDP solves \label{subsec:Custom-heuristic}}

As a heuristic to obtain improve upper bounds, we utilize the fact
that the optimization of \eqref{eq:BnB-PEP-Preli} reduces to an SDP
when the stepsize $\alpha$ is fixed. This is a structure that the
branch-and-bound solver cannot infer by itself.

When the branching process reaches a new node, we access (via a callback
feature of \texttt{JuMP}) the solution $(\alpha^{\textrm{rlx}},\nu^{\textrm{rlx}},\lambda^{\textrm{rlx}},Z^{\textrm{rlx}},P^{\textrm{rlx}})$
of the relaxation and quantify its infeasibility with 
\begin{align*}
 & \textrm{merit}(\alpha^{\textrm{rlx}},\nu^{\textrm{rlx}},\lambda^{\textrm{rlx}},Z^{\textrm{rlx}},P^{\textrm{rlx}})\\
= & \|\sum_{i,j\in I_{N}^{\star}:i\neq j}\lambda_{i,j}^{\textrm{rlx}}a_{i,j}\|_{\infty}+\left\Vert \nu^{\textrm{rlx}}B_{0,\star}-C_{N,\star}-\mu^{2}B_{N,\star}(\alpha^{\textrm{rlx}})+2\mu A_{\star,N}(\alpha^{\textrm{rlx}})\right\Vert _{\infty}\\
 & +\lvert\textrm{min}\{\textrm{eig}_{\textrm{min}}(Z^{\textrm{rlx}}),0\}\rvert,
\end{align*}
where $\textrm{eig}_{\textrm{min}}(Z^{\textrm{rlx}})$ is the minimum
eigenvalue of $Z^{\textrm{rlx}}$. If 
\[
\textrm{merit}(\alpha^{\textrm{rlx}},\nu^{\textrm{rlx}},\lambda^{\textrm{rlx}},Z^{\textrm{rlx}},P^{\textrm{rlx}})\leq\epsilon,
\]
then we fix the stepsize in \eqref{eq:worst-case-pfm-dual-1} to $\alpha^{\textrm{rlx}}$
and solve the convex SDP. (We take $\epsilon=0.01$ in our implementation.)
We submit the solution to the SDP as a heuristic solution (via a callback
feature of \texttt{JuMP}). If the heuristic solution improves the
best upper bound $\overline{p}^{\star}$, then it is accepted by the
solver, else it is rejected.

\subsubsection{Numerical evaluation of the customizations \label{subsec:ablation}}

In our experiments, we found that that the customization of $\mathsection$\ref{subsec:Compute-bounds}
provided the largest speedups, $\mathsection$\ref{subsec:Computing-tighter-global-bound}
substantial speedups, and $\mathsection$\ref{subsec:Custom-heuristic}
no speedups. We further describe our observations here.

\paragraph{Variable bounds of $\mathsection$\ref{subsec:Compute-bounds}.}

Tables~\ref{tab:BnB-PEP-bounds} and \ref{tab:BnB-PEP-bounds-heuristic}
show the bounds obtained through the SDP relaxation. As an aside,
we found that these valid bounds substantially improve not only the
branch-and-bound algorithm of Stage 3, but also the local solve of
Stage 2.

\begin{table}
\begin{centering}
\begin{tabular}{ccccccc}
\toprule 
\multicolumn{1}{c}{{\footnotesize{}$N$}} & {\footnotesize{}$M_{\lambda}$}  & {\footnotesize{}$M_{\alpha}$}  & {\footnotesize{}$M_{Z}$}  & {\footnotesize{}$M_{P}$ }  & {\footnotesize{}$M_{\nu}$ }  & {\footnotesize{}Runtime (s)}\tabularnewline
\midrule
\midrule 
{\footnotesize{}$1$}  & {\footnotesize{}$1.00$ }  & {\footnotesize{}$2.00$ }  & {\footnotesize{}$1.00$ }  & {\footnotesize{}$1.00$ }  & {\footnotesize{}$0.2244$ }  & {\footnotesize{}$0.068$}\tabularnewline
\midrule 
{\footnotesize{}$2$ }  & {\footnotesize{}$1.00$ }  & {\footnotesize{}$4.5175$ }  & {\footnotesize{}$1.00$ }  & {\footnotesize{}$1.00$ }  & {\footnotesize{}$0.0893$ }  & {\footnotesize{}$0.181$}\tabularnewline
\midrule 
{\footnotesize{}$3$ }  & {\footnotesize{}$1.00$ }  & {\footnotesize{}$3.672$ }  & {\footnotesize{}$1.00$ }  & {\footnotesize{}$1.00$ }  & {\footnotesize{}$0.0449$ }  & {\footnotesize{}$0.736$}\tabularnewline
\midrule 
{\footnotesize{}$4$ }  & {\footnotesize{}$1.00$ }  & {\footnotesize{}$3.5166$ }  & {\footnotesize{}$1.00$ }  & {\footnotesize{}$1.00$ }  & {\footnotesize{}$0.0257$ }  & {\footnotesize{}$3.173$}\tabularnewline
\midrule 
{\footnotesize{}$5$ }  & {\footnotesize{}$1.00$ }  & {\footnotesize{}$3.7919$ }  & {\footnotesize{}$1.00$ }  & {\footnotesize{}$1.00$ }  & {\footnotesize{}$0.0159$ }  & {\footnotesize{}$11.380$}\tabularnewline
\bottomrule
\end{tabular}
\par\end{centering}
\begin{centering}
\par\end{centering}
\caption{ Valid bounds on the decision variables in \eqref{eq:BnB-PEP-Preli}
obtained via the SDP relaxation of \eqref{eq:computing_bounds}. The
runtime describes to the total time spent compute all the bounds of
the row. \label{tab:BnB-PEP-bounds}}

\end{table}

\begin{table}
\begin{centering}
{\footnotesize{}{}}%
\begin{tabular}{ccccccc}
\toprule 
\multicolumn{1}{c}{{\footnotesize{}{}$N$}} & {\footnotesize{}{}$M_{\lambda}$ }  & {\footnotesize{}{}$M_{\alpha}$ }  & {\footnotesize{}{}$M_{Z}$ }  & {\footnotesize{}{}$M_{P}$ }  & {\footnotesize{}{}$M_{\nu}$ }  & {\footnotesize{}{}Runtime (s)}\tabularnewline
\midrule
\midrule 
{\footnotesize{}{}$1$ }  & {\footnotesize{}{}$0.8789$ }  & {\footnotesize{}{}$7.6105$ }  & {\footnotesize{}{}$0.4233$ }  & {\footnotesize{}{}$0.6506$ }  & {\footnotesize{}{}$0.1473$ }  & {\footnotesize{}{}$0.082$}\tabularnewline
\midrule 
{\footnotesize{}{}$2$ }  & {\footnotesize{}{}$0.9504$ }  & {\footnotesize{}{}$8.2597$ }  & {\footnotesize{}{}$0.1934$ }  & {\footnotesize{}{}$0.4397$ }  & {\footnotesize{}{}$0.0409$ }  & {\footnotesize{}{}$0.093$}\tabularnewline
\midrule 
{\footnotesize{}{}$3$ }  & {\footnotesize{}{}$0.9767$ }  & {\footnotesize{}{}$9.4761$ }  & {\footnotesize{}{}0.1009 }  & {\footnotesize{}{}$0.3177$ }  & {\footnotesize{}{}$0.0145$ }  & {\footnotesize{}{}$0.105$}\tabularnewline
\midrule 
{\footnotesize{}{}$4$ }  & {\footnotesize{}{}$0.9853$ }  & {\footnotesize{}{}$9.8591$ }  & {\footnotesize{}{}$0.0599$ }  & {\footnotesize{}{}$0.2448$ }  & {\footnotesize{}{}$0.005766$ }  & {\footnotesize{}{}$0.114$}\tabularnewline
\midrule 
{\footnotesize{}{}$5$ }  & {\footnotesize{}{}$0.9886$ }  & {\footnotesize{}{}$10.3633$ }  & {\footnotesize{}{}$0.0383$ }  & {\footnotesize{}{}$0.1958$ }  & {\footnotesize{}{}$0.002459$ }  & {\footnotesize{}{}$0.121$}\tabularnewline
\bottomrule
\end{tabular}
\par\end{centering}
\begin{centering}
 
\par\end{centering}
\caption{Heuristic bounds on the decision variables in \eqref{eq:BnB-PEP-Preli}
with $\widetilde{M} = 1.01$. The runtime describes the total time
spent compute all the bounds of the row. Compared to the results of
Table~\ref{tab:BnB-PEP-bounds} the bounds tend to be tighter, the
runtime is faster, and the implementation is much simpler. However,
there is no theoretical guarantee that the bounds are valid. \label{tab:BnB-PEP-bounds-heuristic}}

\end{table}

\paragraph{Tighter lower bound of $\mathsection$\ref{subsec:Computing-tighter-global-bound}.}

Table~\ref{tab:Lower-bound-BnB-PEP} shows the lower-bounds for \eqref{eq:BnB-PEP-Preli}
computed from the lazy constraint callback method. The customization
produces a high quality lower bound, which, combined with the near-optimal solution of Stage 2 of the BnB-PEP Algorithm, enables the branch-and-bound
algorithm to terminate quickly.

\begin{table}
\begin{centering}
\begin{tabular}{ccccc}
\toprule 
\multicolumn{1}{c}{{\footnotesize{}{}$N$}} & {\footnotesize{}{}$\underline{p}^{\star}$ }  & {\footnotesize{}{}$\overline{p}^{\star}$ }  & {\footnotesize{}{}$\overline{p}^{\star}-\underline{p}^{\star}$ }  & {\footnotesize{}{}Runtime (s)}\tabularnewline
\midrule
\midrule 
{\footnotesize{}{}$1$ }  & {\footnotesize{}{}$0.1432$}  & {\footnotesize{}{}$0.1473$ }  & {\footnotesize{}{}$0.0041$ }  & {\footnotesize{}{}$0.135$}\tabularnewline
\midrule 
{\footnotesize{}{}$2$ }  & {\footnotesize{}{}$0.0374$ }  & {\footnotesize{}{}$0.0409$ }  & {\footnotesize{}{}$0.0035$ }  & {\footnotesize{}{}$0.232$}\tabularnewline
\midrule 
{\footnotesize{}{}$3$ }  & {\footnotesize{}{}$0.0121$ }  & {\footnotesize{}{}$0.0145$ }  & {\footnotesize{}{}0.0024 }  & {\footnotesize{}{}$2.550$}\tabularnewline
\midrule 
{\footnotesize{}{}4 }  & {\footnotesize{}{}$0.00178$ }  & {\footnotesize{}{}$0.005766$ }  & {\footnotesize{}{}$0.003986$ }  & {\footnotesize{}{}$72.7$}\tabularnewline
\midrule 
{\footnotesize{}{}5 }  & {\footnotesize{}{}$0.000517$ }  & {\footnotesize{}{}$0.002459$ }  & {\footnotesize{}{}$0.001941$ }  & {\footnotesize{}{}$336.341$}\tabularnewline
\bottomrule
\end{tabular}
\par\end{centering}
\begin{centering}
 
\par\end{centering}
\caption{Lower bound $\underline{p}^{\star}$ of \eqref{eq:BnB-PEP-Preli}
computed from the lazy constraint callback method. The upper bound
$\overline{p}^{\star}$ is the objective value from Stage 2 of the
BnB-PEP Algorithm. \label{tab:Lower-bound-BnB-PEP}}

\end{table}

\paragraph{Improved upper bound of $\mathsection$\ref{subsec:Custom-heuristic}.}

In our experiments, the submitted upper bounds were all rejected by
the solver and therefore provided no speedup. This is not surprising,
as it is likely due to the warm-starting solution from Stage 2 being
near-optimal. To verify this hypothesis, we ran Stage 3 without Stage
2. In this case, the submitted heuristic solution was often accepted
by the solver, but the overall performance was slow as Stage 3 started
from a poor warm-starting solution. We recommend that users of the
BnB-PEP Algorithm always perform Stage 2 before Stage 3. However, when the warm-starting solution is not near-optimal, we can expect this customization
to provide a speedup.

\subsection{Structured inner-dual variables
\label{subsec:Sparsifier}}

The family solutions to \eqref{eq:BnB-PEP-Preli} with $N=1,2,\dots$ exhibits an exploitable structure: the optimal $\lambda^\star$ is sparse and the optimal $Z^\star$ is low-rank.
A computational benefit of this structure is that it reduces the problem size of the BnB-PEP-QCQP.
A theoretical benefit is that the structured inner-dual variable corresponds to simpler and therefore more analytically tractable proofs, which we seek in  $\mathsection$\ref{subsec:pot-bnb-pep-wcvx-Moreau}.

In this section, we describe a heuristic strategy for identifying such structure. The general idea is to solve the problem exactly for smaller values of $N$, say $N=1,\dots,5$, and infer the pattern. This heuristic is based on the expectation that the observed patterns will continue to hold for $N=6,7,\dots$.

\paragraph{Sparsity pattern of $\lambda$.}
Denote the support of $\lambda^{\star}$ as
\[
\mathrm{supp}(\lambda^\star)=\{(i,j)\mid i,j\in I_{N}^{\star},i\neq j,\;\lambda_{i,j}^{\star}>0\}.
\]
(Note that $\lambda_{i,j}^{\star}\ge 0$ for all $i,j$.)
If we know $\mathrm{supp}(\lambda^\star)$ in advance, then we can simplify \eqref{eq:BnB-PEP-Preli} by replacing both instances of
\[
\sum_{i,j\in I_{N}^{\star}:i\neq j}\lambda_{i,j}(\cdots)
\]
with
\[
\sum_{(i,j)\in\mathrm{supp}(\lambda^\star)}\lambda_{i,j}(\cdots)
\]
and obtain a smaller QCQP.

First solve \eqref{eq:BnB-PEP-Preli} for $N=1,\dots,5$ using the BnB-PEP Algorithm.
At this point, solutions may already reveal their pattern in $\mathrm{supp}(\lambda^\star)$.
However, optimal inner-dual variables for a given FSFOM are not always unique (see \cite{TaylorBlog20} or Table~\ref{tab:lambda-non-unique}), and, if so, the solution returned by the spatial branch-and-bound solver will likely not be a sparse one.
Therefore, following the literature on sparse signal processing \cite[$\mathsection$2]{hastie2019statistical}, we promote sparsity of $\lambda$ by reducing its $\ell_1$-norm%
\footnote{
One could consider further advanced approaches for promoting sparsity such as \cite{CandesWakinBoyd2008_enhancing}.
}
as follows
\begin{equation}
\left(\begin{array}{ll}
\textrm{minimize} & \|\lambda\|_{1}=\sum_{i,j\in I_{N}^{\star}:i\neq j}\lambda_{i,j}\\
\textrm{subject to} & \nu R^{2}\leq p^{\star},\\
 & \sum_{i,j\in I_{N}^{\star}:i\neq j}\lambda_{i,j}a_{i,j}=0,\\
 & \nu B_{0,\star}-C_{N,\star}-\mu^{2}B_{N,\star}(\alpha^{\star})+2\mu A_{\star,N}(\alpha^{\star})+\\
 & \quad\sum_{i,j\in I_{N}^{\star}:i\neq j}\lambda_{i,j}\left(A_{i,j}(\alpha^{\star})+\frac{1}{2(L-\mu)}C_{i,j}\right)=Z,\\
 & Z\succeq0,\\
 & \lambda_{i,j}\geq0,\quad i,j\in I_{N}^{\star}:i\neq j,\\
 & \nu\geq0,
\end{array}\right)\label{eq:find-sparse-sol}
\end{equation}
where $\nu$, $\lambda$, and $Z\in\mathbb{S}_{+}^{N+2}$ are the decision variables.
The constraint set of \eqref{eq:find-sparse-sol} is almost identical to \eqref{eq:worst-case-pfm-dual-1}, except we impose the constraint $\nu R^{2}\leq p^{\star}$ and fix $\alpha^{\star}$ to the optimal stepsize computed with the BnB-PEP Algorithm.
This way, our search is confined to the set of optimal solutions to \eqref{eq:BnB-PEP-Preli} and, with the FSFOM fixed, the problem is efficiently solved as an SDP.
Denote the solution to \eqref{eq:find-sparse-sol} by $\{\nu^{\star},\lambda^{\star,\textrm{sparse}},Z^{\star,\textrm{sparse}}\}$.
Hopefully, the optimal $\lambda^{\star,\textrm{sparse}}$ for $N=1,\dots,5$ are sparse and their structure reveals a pattern.

\begin{table}

\centering{}{\footnotesize{}}%
\begin{tabular}{>{\centering}p{2em}>{\centering}p{9em}>{\centering}p{9em}}
\toprule 
\centering{}{\footnotesize{}$N$} & \centering{}{\footnotesize{}Optimal $\lambda$ with minimum $\ell_{1}$
norm} & \centering{}{\footnotesize{}Optimal $\lambda$ with maximum $\ell_{1}$
norm}\tabularnewline
\midrule
\midrule 
\centering{}{\footnotesize{}1} & \centering{}{\footnotesize{}$2.642$} & \centering{}{\footnotesize{}$3.594$}\tabularnewline
\midrule 
\centering{}{\footnotesize{}2} & \centering{}{\footnotesize{}$2.434$} & \centering{}{\footnotesize{}$3.114$}\tabularnewline
\midrule 
\centering{}{\footnotesize{}3} & \centering{}{\footnotesize{}$2.369$} & \centering{}{\footnotesize{}$2.925$}\tabularnewline
\midrule 
\centering{}{\footnotesize{}4} & \centering{}{\footnotesize{}$2.339$} & \centering{}{\footnotesize{}$2.823$}\tabularnewline
\midrule 
\centering{}{\footnotesize{}5} & \centering{}{\footnotesize{}$2.320$} & \centering{}{\footnotesize{}$2.757$}\tabularnewline
\bottomrule
\end{tabular}\caption{Solutions to \eqref{eq:find-sparse-sol} with minimum and maximum $\ell_1$-norm on optimal $\lambda$ for the setup of $\mathsection$\ref{sec:QCQO-framework-for-template-problem} and $\mathsection$\ref{sec:Efficient-implementation-and-enhancement}.
The gap demonstrates that the optimal inner-dual variable is not unique and therefore that the $\ell_1$-norm minimization is necessary for obtaining a sparse solution. \label{tab:lambda-non-unique}}
\end{table}

\paragraph{Low rank of $Z$.}
When $r=\mathrm{rank}(Z^\star)$ with $r<n$, then we can use the factorization $Z=PP^\top$, where $P\in \rl^{n\times n}$ has $n-r$ columns constrained to be zero. Such constraints significantly reduce the effective size of the QCQP.

\begin{lem}[{\cite[Theorem 10.9]{higham2002accuracy}\label{Lem:Cholesky-decomposition-of-low-rank}}]
A matrix $Z\in\mathbb{S}^{n}$ is positive semidefinite with rank $r\leq n$ if and only if it has a Cholesky factorization $Z=PP^{\top}$, where $P\in \rl^{n\times n}$ is lower-triangular, has $r$ positive diagonal entries, and $n-r$ columns containing all zero.
\end{lem}

We solve \eqref{eq:find-sparse-sol} for $N=1,\dots,5$ and infer the rank.
(In principle, one could perform a separate nuclear norm minimization or further advanced approaches such as \cite{Fazel2003} to reduce the rank of $Z$, but this was not necessary in our experiments.)
In our current setup, $Z^{\star,\textrm{sparse}}$ has rank $1$, as Table~\ref{tab:effective_index_set_lambda} indicates.
Other optimized method throughout the literature such as OGM, ITEM, OGM-G \cite{drori2014performance,kim2016optimized,taylor2021optimal,kim2021optimizing} also have corresponding low-rank $Z^\star$.

For $N=6,7,\dots$, we obtain $Z^\star$ and $P^{\star}$ from Stage 2 of the BnB-PEP Algorithm. If we expect a certain value of $r=\mathrm{rank}(Z^\star)$, keep $r$ columns of $P^{\star}$ with the largest magnitude and constrain the remaining $n-r$ columns to be $0$ in the subsequent Stage 3.

\paragraph{Structured inner-dual variables represent simpler proofs.}
A feasible point of the dualized problem \eqref{eq:worst-case-pfm-dual-1} can be interpreted as a convergence proof combining inequalities
\begin{equation}
f_{{i}}\geq f_{{j}}+\langle g_{{j}} \mid
x_{i}-x_{{j}}\rangle+\frac{1}{2(L-\mu)}\|g_{{i}}-g_{{j}}\|^{2}
\label{eq:ineq-coco}
\end{equation}
for $i,j\in I_N^\star$, and the value of $\lambda_{{i},{j}}^{\star}$ corresponds to the value used in forming a weighted combination of the inequalities \cite[$\mathsection$3.3]{ryu2020operator}.
Therefore, $\lambda_{{i},{j}}^{\star}=0$ for some $(i,j)$ is equivalent to not using the corresponding inequality in the convergence proof.
A sparse $\lambda$ corresponds to a proof using fewer inequalities, which tend to be simpler proofs.

On the other hand, $\mathrm{rank}(Z^\star)$ corresponds to the excess quadratic terms arising within a proof.
For convergence proof of FSFOMs of the form,
\[
A\le B-\|c_1\|^2-\dots \|c_r\|^2\le B,
\]
$r=\mathrm{rank}(Z^\star)$ corresponds to the number of quadratic terms $\{\|c_i\|^2\}_{i=1,\dots,r}$, roughly speaking.
Since $\mathrm{rank}(Z^\star)$ corresponds to the number of excess terms to deal with in a proof, an optimal solution $Z^\star$ with small rank tends to correspond to simpler proofs.

\paragraph{Numerical results.}
Table~\ref{tab:effective_index_set_lambda} presents 
$\mathrm{supp}(\lambda^{\star,\textrm{sparse}})$, $\mathrm{rank}(Z^{\star,\textrm{sparse}})$,
and the non-zero columns of $P^{\star,\textrm{sparse}}$ from solving the convex SDP \eqref{eq:find-sparse-sol}.
We use $\mu=0.1$, $L=1$, and $R=1$.
From the results, we infer the pattern
\begin{equation*}
\mathrm{supp}(\lambda^{\star,\textrm{sparse}})=\{(\star,i)\}_{i\in[0:N]}\cup\{(i,i+1)\}_{i\in[0:N-1]}\cup\{N,i\}_{i\in\{\star\}\cup[0:N-1]},
\end{equation*}
which has only $3N+2$ components compared to the $(N+2)(N+1)$ of the full index set.
Furthermore, $\mathrm{rank}(Z^{\star,\textrm{sparse}})=1$.
For $N=1,\dots,5$, we verified that there are globally optimal solutions satisfying these patterns.
For $N=6,7,\dots,25$, we verified that there are locally optimal solutions satisfying these patterns.

\begin{table}
\begin{centering}
{\footnotesize{}}%
\begin{tabular}{cc>{\centering}p{5em}>{\centering}p{5em}>{\centering}p{5em}}
\toprule 
\multicolumn{1}{c}{{\footnotesize{}$N$}} & {\footnotesize{}$\mathrm{supp}(\lambda^{\star})$ } & \centering{}{\footnotesize{}Rank of $Z^{\star,\textrm{sparse}}$ } & \centering{}{\footnotesize{}Index of nonzero column of $P^{\star,\textrm{sparse}}$ } & \centering{}{\footnotesize{}Runtime (s) for solving \eqref{eq:find-sparse-sol}}\tabularnewline
\midrule
\midrule 
{\footnotesize{}$1$  } & {\footnotesize{}$\begin{array}{l}
\{(\star,0),(\star,1),\\
(0,1),\\
(1,\star),(1,0)\}
\end{array}$  } & \centering{}{\footnotesize{}$1$ } & \centering{}{\footnotesize{}Column \# $1$ } & \centering{}{\footnotesize{}$0.0021$}\tabularnewline
\midrule 
{\footnotesize{}$2$  } & {\footnotesize{}$\begin{array}{l}
\{(\star,0),(\star,1),(\star,2),\\
(0,1),(1,2),\\
(2,\star),(2,0),(2,1)\}
\end{array}$  } & \centering{}{\footnotesize{}$1$ } & \centering{}{\footnotesize{}Column \# 1 } & \centering{}{\footnotesize{}$0.0056$}\tabularnewline
\midrule 
{\footnotesize{}$3$  } & {\footnotesize{}$\begin{array}{l}
\{(\star,0),(\star,1),(\star,2),(\star,3),\\
(0,1),(1,2),(2,3),\\
(3,\star),(3,0),(3,1),(3,2)\}
\end{array}$  } & \centering{}{\footnotesize{}$1$  } & \centering{}{\footnotesize{}Column \# 1  } & \centering{}{\footnotesize{}$0.0071$}\tabularnewline
\midrule 
{\footnotesize{}4  } & {\footnotesize{}$\begin{array}{l}
\{(\star,i)\}_{i\in[0:4]}\cup\\
\{(i,i+1)\}_{i\in[0:3]}\cup\\
\{4,i\}_{i\in\{\star\}\cup[0:3]}
\end{array}$  } & \centering{}{\footnotesize{}$1$  } & \centering{}{\footnotesize{}Column \# 1  } & \centering{}{\footnotesize{}$0.0097$}\tabularnewline
\midrule 
{\footnotesize{}5  } & {\footnotesize{}$\begin{array}{l}
\{(\star,i)\}_{i\in[0:5]}\cup\\
\{(i,i+1)\}_{i\in[0:4]}\cup\\
\{5,i\}_{i\in\{\star\}\cup[0:4]}
\end{array}$ } & \centering{}{\footnotesize{}$1$  } & \centering{}{\footnotesize{}Column \# 1  } & \centering{}{\footnotesize{}$0.0140$}\tabularnewline
\bottomrule
\end{tabular}{\footnotesize\par}
\par\end{centering}
{\caption{Structure of the inner-dual variables obtained from the convex SDP \eqref{eq:find-sparse-sol}.
The last column shows the runtime to solve \eqref{eq:find-sparse-sol}.
(Table~\ref{tab:timing-bnb-pep-grad-red-FmuL} shows the runtime to solve \eqref{eq:BnB-PEP-Preli}, a prerequisite for solving \eqref{eq:find-sparse-sol}.) \label{tab:effective_index_set_lambda}}
}
\end{table}

\paragraph{Discussion.}
Prior work on optimized FSFOMs such as OGM, ITEM, and OGM-G \cite{drori2014performance,kim2016optimized,taylor2021optimal,kim2021optimizing} discard certain inequalities in their formulations. The choice of which inequality to discard, which is equivalent to identifying $\mathrm{supp}(\lambda^\star)$, was likely carried out through ad-hoc trial and error. As no reasoning or intuition was provided behind the choice and as the set of discarded inequalities are different from one work to another, the process is opaque. Our approach provides a systematic process for making this choice.

To clarify, we solve the exact, unrelaxed \eqref{eq:BnB-PEP-Preli} with BnB-PEP Algorithm for $N=1,\dots,5$.
The methodology for $N=6,7,\dots$ is a heuristic in the sense that our solution is exactly only under the condition that the observed sparsity pattern continues. If the pattern changes, the QCQP becomes a relaxation, and the produced FSFOM becomes suboptimal. However, one can be reasonably confident in the sparsity pattern as it is based on the exact solutions for $N=1,\dots,5$.

%% file: Sections/section_3.tex
\section{Generalized BnB-PEP methodology
\label{subsec:Generalization-of-QCQO-framework}}
We now discuss the generalization of the BnB-PEP methodology for general
$\mathcal{E}$, $\mathcal{F}$, and $\mathcal{C}$.

\paragraph{Generalized BnB-PEP-QCQP.}
The BnB-PEP-QCQP formulation for general $\mathcal{E}$, $\mathcal{F}$, and $\mathcal{C}$ follows steps analogous to those of $\mathsection$\ref{subsec:Converting-into-QCQO-gradient-reduction-F-mu-L}.
\begin{enumerate}
\item[(i)] \textbf{Infinite-dimensional inner optimization problem.}
Construct an infinite-dimensional representation of \eqref{eq:worst-case-pfm}
analogous to \eqref{eq:worst-case-pfm-fmuL-1}
of $\mathsection$\ref{subsec:Converting-into-QCQO-gradient-reduction-F-mu-L}.
When $x_{\star}$ exists, set $x_{\star}=0$ and $f(x_{\star})=0$ without loss of generality.
\item[(ii)] 
\textbf{Interpolation argument.} 
Using a reparametrization (if necessary) and an interpolation argument, 
formulate the infinite-dimensional inner problem of (i) as a finite-dimensional problem
analogous to \eqref{eq:worst-case-pfm-fmuL-3} of $\mathsection$\ref{subsec:Converting-into-QCQO-gradient-reduction-F-mu-L}.
\item[(iii)]
\textbf{Grammian formulation.}
By introducing Grammian matrices and using a large-scale assumption,
formulate the finite-dimensional inner maximization problem of (ii) as an SDP, analogous to
the problem \eqref{eq:worst-case-pfm-sdp-1} in $\mathsection$\ref{subsec:Converting-into-QCQO-gradient-reduction-F-mu-L}.
When the FSOM is fixed, the SDP is a convex optimization problem.
\item[(iv)]
\textbf{Dualization.}
Form the dual the SDP of (iii) analogous to \eqref{eq:worst-case-pfm-dual-1} of $\mathsection$\ref{subsec:Converting-into-QCQO-gradient-reduction-F-mu-L}.
Assume strong duality.
\item[(v)]
\textbf{Formulating \eqref{eq:main-min-max-problem} as a QCQP.}
Using Lemma~\ref{Lem:quadratic-characterization-psd-1},
replace the SDP constraint $Z\succeq 0$ of the dual SDP with $Z=PP^\top$, where $P$
is lower triangular with nonnegative diagonals.
This formulates \eqref{eq:main-min-max-problem} as a QCQP analogous to
\eqref{eq:BnB-PEP-Preli} of $\mathsection$\ref{subsec:Converting-into-QCQO-gradient-reduction-F-mu-L}.
When $f$ is nonconvex, certain cubic or trilinear terms may arise, whereas for a convex $f$ the nonlinear terms are bilinear or quadratic.
If so, for such a nonconvex $f$, formulate such terms as quadratic or bilinear constraints by introducing dummy variables, a process illustrated in $\mathsection$\ref{subsec:Smooth-nonconvex-gradient-reduction}
and $\mathsection$\ref{subsec:pot-bnb-pep-wcvx-Moreau}.
We call the resultant QCQP the BnB-PEP-QCQP.
The variables of the dual SDP of (iv) are present in the BnB-PEP-QCQP, and we refer to them as the inner-dual-variables.
\end{enumerate}

\paragraph{Generalized BnB-PEP Algorithm.}
We solve the BnB-PEP-QCQP to certifiable global optimality with the following generalized BnB-PEP Algorithm,
a generalization of Algorithm \ref{alg:Alg-1}.
\begin{itemize}
\item \textbf{Stage 1: Compute a feasible solution.}
Fix the stepsizes in the dual SDP of Step (iv) of the formulation of BnB-PEP-QCQP
to a reasonable $h^{\textrm{init}}$
and solve the resultant convex minimization problem to obtain a feasible the BnB-PEP-QCQP.
Table~\ref{tab:Fixed-stepsize-vector-table} lists reasonable stepsizes.
\item \textbf{Stage 2: Compute a locally optimal solution 
by warm-starting at Stage 1 solution.} Warm-start the BnB-PEP-QCQP with
the feasible solution found in Stage 1 and solve the problem to local
optimality using a nonlinear interior-point method.
\item \textbf{Stage 3: Compute a globally optimal solution by 
warm-starting at Stage 2 solution.} Warm-start the BnB-PEP-QCQP with
the locally optimal solution found in Stage 2 and solve the problem
to global optimality using a customized spatial branch-and-bound algorithm
described in the following.
\end{itemize}
\begin{table}
\begin{centering}
\begin{tabular}{>{\centering}p{5em}>{\centering}p{20em}}
\toprule 
{\footnotesize{}Function class} & {\footnotesize{}Fixed stepsize $h^{\textrm{init}}$ for Stage 1 of
the BnB-PEP Algorithm}\tabularnewline
\midrule
\midrule 
{\footnotesize{}$\mathcal{F}_{0,L}$ } & {\footnotesize{}$h_{i,j}^{\textrm{init}}=\begin{cases}
1/L, & \textrm{if }j=i-1,\\
0, & \textrm{else},
\end{cases}\quad0\leq j<i\leq N.$}\tabularnewline
\midrule 
{\footnotesize{}$\mathcal{F}_{\mu,L}$} & {\footnotesize{}Same as $\mathcal{F}_{0,L}.$}\tabularnewline
\midrule 
{\footnotesize{}$\mathcal{F}_{-L,L}$ } & {\footnotesize{}Same as $\mathcal{F}_{0,L}.$}\tabularnewline
\midrule 
{\footnotesize{}$\mathcal{W}_{\rho,L}$} & {\footnotesize{}$h_{i,j}^{\textrm{init}}=\begin{cases}
\frac{R\rho}{L}\frac{1}{\sqrt{N+1}}, & \textrm{if }j=i-1,\\
0, & \textrm{else},
\end{cases}\quad0\leq j<i\leq N.$}\tabularnewline
\bottomrule
\end{tabular}
\par\end{centering}
\caption{Fixed stepsize vector $h^{\textrm{init}}$ to use in step 1 of the
BnB-PEP Algorithm. For $\protect\cpc$, and $\protect\rhoLWcvx,$ $R$\textgreater 0
is the upper bound associated with the initial condition. 
\label{tab:Fixed-stepsize-vector-table}}
\end{table}

\paragraph{Efficient implementation of the generalized BnB-PEP Algorithm.}
We customize the spatial branch-and-bound algorithm to exploit specific problem structure of the generalized BnB-PEP-QCQP.
The techniques are analogous to those described in $\mathsection$\ref{sec:Efficient-implementation-and-enhancement}.
We find bounds on optimal solutions through implied linear constraints, SDP relaxation, and a heuristic.
We find tighter lower bounds via lazy callback by replacing $Z=PP^\top$ with 
\[
\tr\left(Zyy^{\top}\right)\geq0,\qquad \forall\,y\in Y
\]
and lazily updating $Y$.
We improve upper bounds via SDP solves by constructing a merit function to measure the infeasibility at the nodes of the branch-and-bound tree and solving the convex SDP with the stepsizes fixed when the merit function value falls below some tolerance.
We exploit the structure of the inner-dual variables by observing the sparsity and low-rank pattern for small $N$ (e.g., $N \leq 5$) and extrapolating the patterns to larger $N$.

%% file: Sections/section_4.tex
\section{Applications \label{sec:applications}}
In this section, we demonstrate the strength of the BnB-PEP methodology by applying it to three setups for which the prior methodologies do not apply.
Numerical experiments of this section were performed in the computational setup described in $\mathsection$\ref{subsec:computational_setup}. We empirically observed that, among the two approaches of $\mathsection$\ref{subsec:Compute-bounds} for computing variable bounds, the heuristic-based bounds were tighter and lead to runtimes faster by factor of 2--5 compared to using the SDP-based bounds.
We report the faster runtimes in our tables.
In all instances, the two approaches produced the same optimal solutions.

%% file: Sections/acceleration_without_momentum.tex
\subsection{Optimal gradient method without momentum \label{subsec:AWM}}

In optimization folklore, momentum is considered essential for accelerating first-order gradient methods. Indeed, prior FSFOMs minimizing smooth convex functions such as Nesterov's method \cite{Nest83}, OGM \cite{drori2014performance,kim2016optimized}, ITEM \cite{taylor2021optimal}, and many others \cite{lee2021, goujaud2021super} all achieve accelerated rates with momentum. However, a little known fact is that simple gradient descent, without momentum, can achieve an accelerated rate for minimizing strongly convex \emph{quadratics} \cite{young1953,Fabian21}. Whether a similar acceleration without momentum is possible for convex non-quadratic functions is not known.

In this section, we investigate whether the simple gradient descent
method 
\begin{equation}
x_{i}=x_{i-1}-\frac{1}{L}h_{i-1}\nabla f(x_{i-1})\tag{\ensuremath{\mathcal{G}_{N}}}\label{eq:simple-gradient-descent}
\end{equation}
with $i\in[1:N]$ can achieve an accelerated rate for minimizing $L$-smooth convex functions when the stepsize $\{h_{i}\}_{i\in[0:N-1]}$ is chosen optimally.
We denote the class of FSFOM of this form as $\mathcal{G}_{N}\subset\mathcal{M}_N$.

As we discuss in $\mathsection$\ref{subsubsec:lowero1k}, it is relatively straightforward to show that the unaccelerated $\mathcal{O}(1/k)$ rate cannot be surpassed if $\{h_{i}\}_{i\in[0:N-1]}$ stays within the ``standard'' range $(0,2)$.
However, Young's method \cite{young1953} uses long steps satisfying $1 <h_{i}<L/\mu$ for some $i$ to achieve an accelerated rate for $L$-smooth and $\mu$-strongly convex quadratics.
The question is whether a similar use of long steps can provide an acceleration in the smooth convex setup.

Formally, we choose the function class 
\[
\mathcal{F}=\{f\,|\,f\in\LSmthCvx,\,f\text{ has a minimizer }x_{\star}\},
\]
performance
measure $\mathcal{E}=f(x_{N})-f(x_{\star})$, and initial condition $\mathcal{C}=\|x_{0}-x_{\star}\|^{2}-R^{2}\leq0$
with $R>0$. We solve the following instance of \eqref{eq:main-min-max-problem}:
\begin{equation*}
\begin{array}{ll}
\mathcal{R}^{\star}\left(\mathcal{G}_{N},\mathcal{E},\mathcal{F},\mathcal{C}\right)=\underset{M\in\mathcal{G}_{N}}{\mbox{minimize}} & \mathcal{R}\left(M,\mathcal{E},\mathcal{F},\mathcal{C}\right)\end{array}.
\end{equation*}

\paragraph{Derivation of BnB-PEP-QCQP. }

Following $\mathsection$\ref{subsec:Generalization-of-QCQO-framework}  Step (i), formulate the inner optimization problem \eqref{eq:worst-case-pfm} as
\begin{align*}
 & \mathcal{R}\left(M,\mathcal{E},\mathcal{F},\mathcal{C}\right)\nonumber \\
 & =\left(\begin{array}{ll}
\textrm{maximize} & f(x_{N})-f(x_{\star})\\
\textrm{subject to} & f\in\LSmthCvx,\\
 & \nabla f(x_{\star})=0,\\
 & x_{i}=x_{i-1}-h_{i-1}\nabla f(x_{i-1})\quad i\in[1:N],\\
 & \|x_{0}-x_{\star}\|^{2}\leq R^{2},\\
 & x_{\star}=0,\;f(x_{\star})=0,
\end{array}\right)
\end{align*}
where $f$ and $x_{0},\ldots,x_{N}$ are the decision variables. Write $h=\{h_{i}\}_{i\in[0:N-1]}$.
Following $\mathsection$\ref{subsec:Generalization-of-QCQO-framework}  Step (ii), use the interpolation argument to formulate the inner problem as
\begin{align*}
 & \mathcal{R}\left(M,\mathcal{E},\mathcal{F},\mathcal{C}\right)\nonumber \\
= & \left(\begin{array}{l}
\textrm{maximize}\quad f_N-f_\star \\ %
\textrm{subject to}\\
f_{i}\geq f_{j}+\langle g_{j}\mid x_{i}-x_{j}\rangle+\frac{1}{2L}\|g_{i}-g_{j}\|^{2},\quad i,j\in I_{N}^{\star}:i\neq j,\\
g_{\star}=0,x_{\star}=0,f_{\star}=0,\\
x_{i}=x_{0}-(1/L)\sum_{j=0}^{i-1}h_{j}g_{j},\quad i\in[1:N],\\
\|x_{0}-x_{\star}\|^{2}\leq R^{2},
\end{array}\right)
\end{align*}
where $\{x_{i},g_{i},f_{i}\}_{i\in I_{N}^{\star}}\subseteq\rl^{d}\times\rl^{d}\times\rl$ are the decision variables.
Following $\mathsection$\ref{subsec:Generalization-of-QCQO-framework}  Step (iii), implement the Grammian transformation.
Define the Grammian matrices $H\in\rl^{d\times(N+2)}$, $G\in\mathbb{S}_{+}^{N+2}$, and $F\in\rl^{1\times(N+1)}$ using the same equations in \eqref{eq:grammian-mats}, $\{\mathbf{x}_{i},\mathbf{g}_{i},\mathbf{f}_{i}\}_{i\in I_{N}^{\star}}$
using the same encoding as \eqref{eq:bold-vec},
except for $\{{\bf x}_{i}\}_{i\in[1:N]}$, which we define as
\begin{align*}
 & \mathbf{x}_{i}=\mathbf{x}_{0}-(1/L)\sum_{j=0}^{i-1}h_{j}\mathbf{g}_{j}\in\rl^{N+2},\quad i\in[1:N].
\end{align*}
Note, $\mathbf{x}_{i}$ is linearly parameterized by $h$. The matrices
$A_{i,j}$, $B_{i,j}$, $C_{i,j}$, and $a_{i,j}$ are the same as in \eqref{eq:ABCa-mat-vec}
except that they are now parameterized by $h$.
Under the large-scale assumption $d\geq N+2$, we equivalently formulate the inner problem as the SDP 
\begin{align*}
 & \mathcal{R}\left(M,\mathcal{E},\mathcal{F},\mathcal{C}\right)\nonumber \\
= & \left(\begin{array}{l}
\textrm{maximize}\quad Fa_{\star,N}\\
\textrm{subject to}\\
Fa_{i,j}+\tr GA_{i,j}(h)+\frac{1}{2L}\tr GC_{i,j}\leq0,\quad i,j\in I_{N}^{\star}:i\neq j,\quad \rhd \textsf{\;dual var.\;} \lambda_{i,j}\geq0\\
-G\preceq0,\quad \rhd \textsf{\;dual var.\;} Z\succeq0\\
\tr GB_{0,\star}\leq R^{2},\quad \rhd \textsf{\;dual var.\;} \nu\geq0
\end{array}\right)
\end{align*}
where $F\in\rl^{1\times(N+1)}$ and $G\in\rl^{(N+2)\times(N+2)}$ are the decision variables.
Following $\mathsection$\ref{subsec:Generalization-of-QCQO-framework}  Step (iv), construct the dual:
\begin{align*}
 & \mathcal{R}\left(M,\mathcal{E},\mathcal{F},\mathcal{C}\right)\nonumber \\
 & =\left(\begin{array}{l}
\textrm{minimize}\quad\nu R^{2}\\
\textrm{subject to}\\
\sum_{i,j\in I_{N}^{\star}:i\neq j}\lambda_{i,j}a_{i,j}-a_{\star,N}=0,\\
\nu B_{0,\star}+\sum_{i,j\in I_{N}^{\star}:i\neq j}\lambda_{i,j}\left(A_{i,j}(h)+\frac{1}{2L}C_{i,j}\right)=Z,\\
Z\succeq0,\\
\nu\geq0,\;\lambda_{i,j}\geq0,\quad i,j\in I_{N}^{\star}:i\neq j,
\end{array}\right)
\end{align*}
where $\nu$, $\lambda$, and $Z$ are the decision variables.
Assume that strong duality holds.
Finally, following $\mathsection$\ref{subsec:Generalization-of-QCQO-framework}  Step (v), use Lemma~\ref{Lem:quadratic-characterization-psd-1} to pose \eqref{eq:main-min-max-problem} as the following BnB-PEP-QCQP:
\begin{align}
 & \mathcal{R}^{\star}\left(\mathcal{G}_{N},\mathcal{E},\mathcal{F},\mathcal{C}\right)\nonumber \\
= & \left(\begin{array}{l}
\textrm{minimize}\quad\nu R^{2}\\
\textrm{subject to}\\
\sum_{i,j\in I_{N}^{\star}:i\neq j}\lambda_{i,j}a_{i,j}-a_{\star,N}=0,\\
\nu B_{0,\star}+\sum_{i,j\in I_{N}^{\star}:i\neq j}\lambda_{i,j}\left(A_{i,j}(h)+\frac{1}{2L}C_{i,j}\right)=Z,\\
P\text{ is lower triangular with nonnegative diagonals},\\
PP^{\top}=Z,\\
\nu\geq0,\;\lambda_{i,j}\geq0,\quad i,j\in I_{N}^{\star}:i\neq j,
\end{array}\right)\label{eq:BnB-PEP-final-QCQO-AWM}
\end{align}
where $\lambda$, $\nu$, $Z$, $P$, and $h$ are the decision variables.

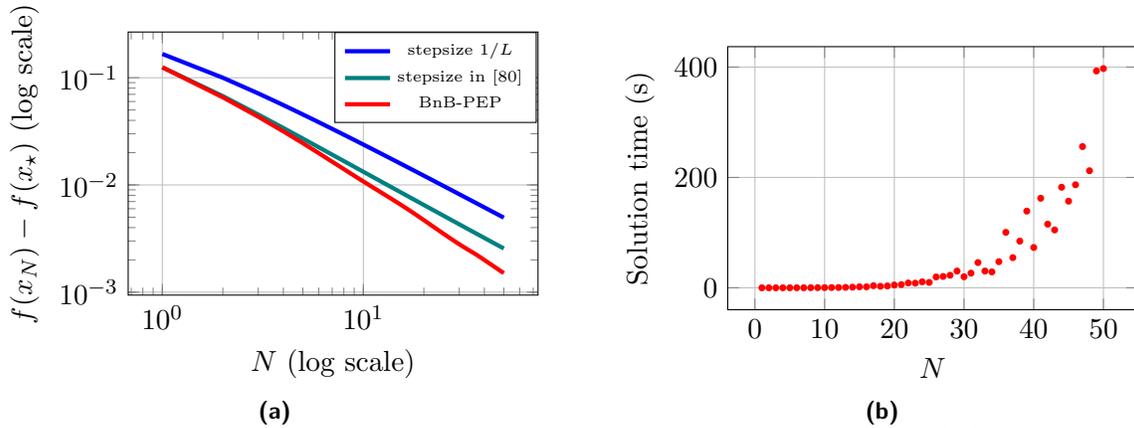
\begin{figure}[htp]%
    \centering
    \subfloat[\centering ]{{\input{Sections/Images/AWM/AWM_vs_Drori_performance_comparison_LogLogScale.tex} }}%
    \qquad
    \subfloat[\centering ]{{\input{Sections/Images/AWM/AWM_timing_performance.tex} }}%
    \caption{
Numerical results for computing locally optimal stepsizes  by solving \eqref{eq:BnB-PEP-final-QCQO-AWM} with the first two stages of the BnB-PEP Algorithm. Global optimality of the stepsizes are verified for $N=1,2,\ldots, 25$.
(Left) Worst-case performance of $f(x_{N})-f_{\star}$ vs.\ iteration count $N$.
(Right) Runtimes of the BnB-PEP Algorithm (including Stages 1 and 2 but excluding Stage 3).
}\label{fig:AWM-performance-Shuvo-vs-Drori}
\end{figure}

\paragraph{Numerical results.}

\begin{table}
\centering{}%
\begin{tabular}{>{\centering}b{2em}>{\centering}b{4em}>{\centering}b{4em}>{\centering}b{5em}>{\centering}b{5em}>{\centering}b{5em}>{\centering}b{7em}}
\toprule 
\multirow{2}{2em}{\centering{}{\footnotesize{}$N$}} & \multirow{2}{4em}{\centering{}{\footnotesize{}\# variables}} & \multirow{2}{4em}{\centering{}{\footnotesize{}\# constraints}} & \multicolumn{3}{c}{{\footnotesize{}Worst-case $f(x_{N})-f(x_{\star})$}} & \multirow{2}{7em}{\centering{}{\footnotesize{}Runtime of the BnB-PEP Algorithm}}\tabularnewline
\cmidrule{4-6} \cmidrule{5-6} \cmidrule{6-6} 
 &  &  & \centering{}{\footnotesize{}Optimal} & {\footnotesize{}For stepsize in \cite[$\mathsection$4.1.1]{taylor2017smooth}} & \centering{}{\footnotesize{}For stepsize $h_{i}=1$} & \tabularnewline
\midrule
\midrule 
\centering{}{\footnotesize{}$1$} & {\footnotesize{}20} & \centering{}{\footnotesize{}$33$} & \centering{}{\footnotesize{}$0.125$} & {\footnotesize{}$0.125$} & \centering{}{\footnotesize{}$0.1667$} & \centering{}{\footnotesize{}$0.03$ s}\tabularnewline
\midrule
\centering{}{\footnotesize{}$2$} & {\footnotesize{}34} & \centering{}{\footnotesize{}$56$} & \centering{}{\footnotesize{}$0.065946$} & {\footnotesize{}$0.067355$} & \centering{}{\footnotesize{}$0.1$} & \centering{}{\footnotesize{}$0.252$ s}\tabularnewline
\midrule
\centering{}{\footnotesize{}$3$} & {\footnotesize{}54} & \centering{}{\footnotesize{}$85$} & \centering{}{\footnotesize{}$0.042893$} & {\footnotesize{}$0.045364$} & \centering{}{\footnotesize{}$0.0714$} & \centering{}{\footnotesize{}$0.375$ s}\tabularnewline
\midrule
\centering{}{\footnotesize{}$4$} & {\footnotesize{}77} & \centering{}{\footnotesize{}$120$} & \centering{}{\footnotesize{}$0.03117$} & {\footnotesize{}$0.033976$} & \centering{}{\footnotesize{}$0.0555$} & \centering{}{\footnotesize{}$17.602$ s}\tabularnewline
\midrule
\centering{}{\footnotesize{}$5$} & {\footnotesize{}104} & \centering{}{\footnotesize{}$161$} & \centering{}{\footnotesize{}$0.024071$} & {\footnotesize{}$0.0270701$} & \centering{}{\footnotesize{}$0.0454$} & \centering{}{\footnotesize{}$86.904$ s}\tabularnewline
\midrule
\centering{}{\footnotesize{}$10$} & {\footnotesize{}$365$} & \centering{}{\footnotesize{}$456$} & \centering{}{\footnotesize{}$0.010622$} & {\footnotesize{}$0.0132692$} & \centering{}{\footnotesize{}$0.0238$} & \centering{}{\footnotesize{}$1$ d $18$ h}\tabularnewline
\midrule
\centering{}{\footnotesize{}$25$} & {\footnotesize{}1835} & \centering{}{\footnotesize{}$2241$} & \centering{}{\footnotesize{}$0.0034757$} & {\footnotesize{}$0.0051754$} & \centering{}{\footnotesize{}$0.0098$} & \centering{}{\footnotesize{}$2$ d $20$ h}\tabularnewline
\bottomrule
\end{tabular}
\caption{{Comparison between the performances of the optimal method obtained by solving \eqref{eq:BnB-PEP-final-QCQO-AWM} with the BnB-PEP Algorithm, the method with constant normalized stepsize $h_{i}=1$, and the method with constant normalized stepsize $h_i$ prescribed in \cite[$\mathsection$4.1.1]{taylor2017smooth}. The BnB-PEP Algorithm was executed on a standard laptop for $N=1,2,\ldots,10$, and on \texttt{MIT Supercloud} for $N=25$.} 
\label{tab:BnB-PEP-result-AWM}}
\end{table}

\begin{table}
\begin{centering}
{\footnotesize{}}%
\begin{tabular}{cc}
\toprule 
{\footnotesize{}$N$} & {\footnotesize{}$h^{\star}$}\tabularnewline
\midrule
\midrule 
{\footnotesize{}$1$} & {\footnotesize{}$\left[\begin{array}{c}
1.5\end{array}\right]$ }\tabularnewline
\midrule
{\footnotesize{}$2$} & {\footnotesize{}$\left[\begin{array}{c}
1.414214\\
1.876768
\end{array}\right]$  }\tabularnewline
\midrule
{\footnotesize{}$3$} & {\footnotesize{}$\left[\begin{array}{c}
1.414215\\
2.414207\\
1.500001
\end{array}\right]$  }\tabularnewline
\midrule
{\footnotesize{}$4$} & {\footnotesize{}$\left[
\begin{array}{c}
1.414214 \\
1.601232 \\
3.005144 \\
1.5 \\
\end{array}
\right]$  }\tabularnewline
\midrule
{\footnotesize{}$5$} & {\footnotesize{}$\left[
\begin{array}{c}
1.414214 \\
2.0 \\
1.414214 \\
3.557647 \\
1.5 \\
\end{array}
\right]$  }\tabularnewline
\midrule
{\footnotesize{}$10$} & {\footnotesize{} See Supplementary Information or \texttt{Github}  repository} \tabularnewline
\midrule
{\footnotesize{}$25$} & {\footnotesize{} See Supplementary Information or \texttt{Github} repository} \tabularnewline
\bottomrule
\end{tabular}\caption{{Globally optimal stepsizes obtained by solving \eqref{eq:BnB-PEP-final-QCQO-AWM} with the BnB-PEP Algorithm. \label{tab:Globally-optimal-stepsize-AWM}}}
\par\end{centering}
\end{table}

Tables~\ref{tab:BnB-PEP-result-AWM} and \ref{tab:Globally-optimal-stepsize-AWM} present the results of solving \eqref{eq:BnB-PEP-final-QCQO-AWM} with $L=1$, $R=1$,
$N=1,\ldots,5,10,25$ using the BnB-PEP Algorithm.
We compare the optimal stepsize with the constant normalized stepsize $h_{i}=1$ for all $i$, which is known to be optimal among $h_i \in (0,1]$
\cite[Corollary 2.8, Theorem 2.9]{drori2014thesis},
and constant normalized stepsize $h$ satisfying 
\begin{equation}
\frac{1}{2Nh+1}=(1-h)^{2N}, \label{eq:adrien_stepsize}
\end{equation}
which is conjectured by Taylor, Hendrickx, and Glineur to be the optimal constant normalized stepsize \cite[$\mathsection$4.1.1]{taylor2017smooth}.
The stepsizes presented in Table~\ref{tab:Globally-optimal-stepsize-AWM} are certified to be globally optimal.
Interestingly, we observe that the locally optimal stepsizes obtained at Stage 2, denoted by $h^{\textrm{lopt}}$, were already near-optimal.

This observation motivates us to apply just the first two stages of the BnB-PEP Algorithm for $N$ upto 50.
In Figure~\ref{fig:AWM-performance-Shuvo-vs-Drori}, we again compare the performance guarantees of the locally optimal stepsizes $h^{\textrm{lopt}}$ with that of the constant normalized stepsizes $h_{i}=1$ and the constant normalized stepsizes $h_i$ satisfying \eqref{eq:adrien_stepsize}. We verified global optimality of these stepsizes $h^{\textrm{lopt}}$ for $N=1,\dots,25$.
While $h^{\textrm{lopt}}$ for $N=26,\dots,50$ are not certifiably globally optimal (although we suspect that they are near-optimal), their computed performances are certifiably accurate.

Figure~\ref{fig:AWM-performance-Shuvo-vs-Drori} shows that the computed stepsizes $h^{\textrm{lopt}}$ outperforms both constant stepsizes.
Figure~\ref{fig:AWM_fit} presents a linear fit of the rate in the log-log scale. The fit $0.156/N^{1.178}$ indicates that the asymptotic rate may be faster than $\mathcal{O}(1/k)$.

Figure~\ref{fig:lopt_step_size_AWM} shows $h^{\textrm{lopt}}$ for $N=5,10,25,50$.
{
The optimal stepsizes $h^{\textrm{lopt}}$ for $N=1,2,\ldots,50$ (global optimality verified for $N=1,2,\ldots,25$) are provided as Supplementary Information and as a data file  at:
\begin{center}
\href{https://github.com/Shuvomoy/BnB-PEP-code/blob/main/Misc/stpszs.jl}{\texttt{https://github.com/Shuvomoy/BnB-PEP-code/blob/main/Misc/stpszs.jl}}
\end{center}
}

However, finding an analytical form of the computed stepsizes seems difficult. Therefore, we leave inconclusive the question of whether acceleration without momentum is possible.

\begin{figure}[htp]%
    \centering
    \subfloat[\centering ]{{\input{Sections/Images/AWM/BnB_PEP_vs_LogLogfit.tex} }}%
    \qquad
    \subfloat[\centering ]{{\input{Sections/Images/AWM/BnB_PEP_vs_fit.tex} }}%
    \caption{
Fitting the worst-case performance of $f(x_{N})-f_{\star}$ corresponding to the locally optimal stepsizes obtained by solving \eqref{eq:BnB-PEP-final-QCQO-AWM} with the first two stages of the BnB-PEP Algorithm yields $0.156/{N^{1.178}}$. 
The asymptotic rate may be faster than $\mathcal{O}(1/k)$.
}\label{fig:AWM_fit}
\end{figure}
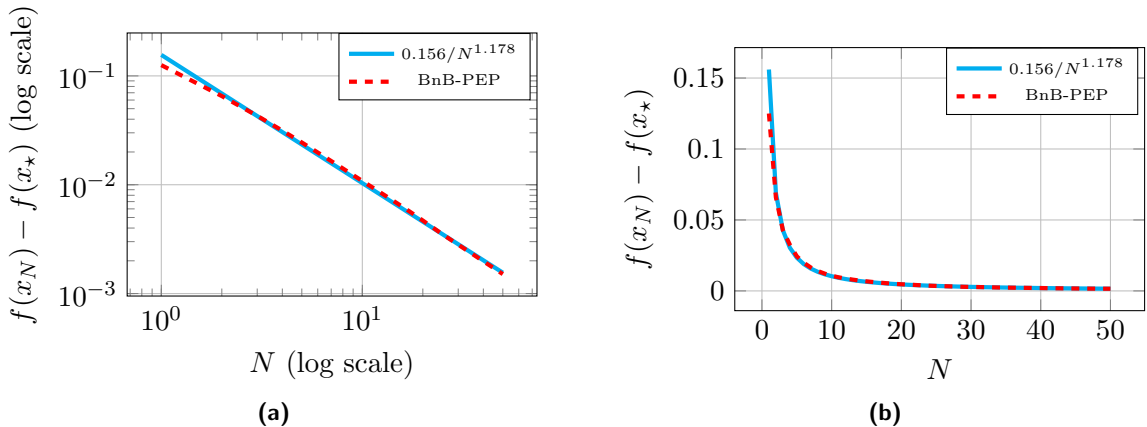

\begin{figure}[htp]
\centering{}\subfloat[$N=5$]{\centering{}\includegraphics[scale=0.75]{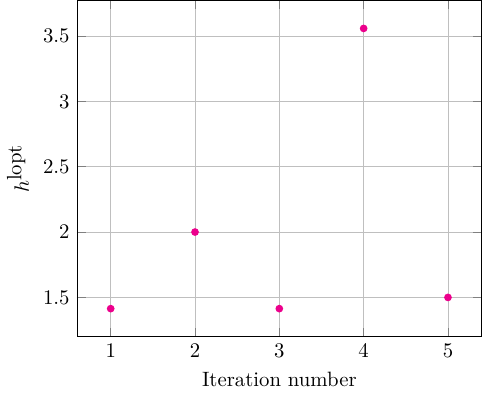}}\hfill{}\subfloat[$N=10$]{\begin{centering}
\includegraphics[scale=0.75]{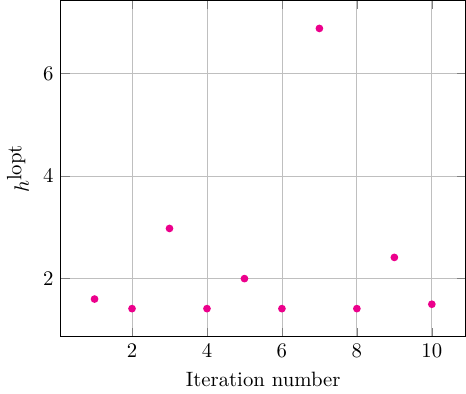}
\par\end{centering}
}\hfill{}\subfloat[$N=25$]{\begin{centering}
\includegraphics[scale=0.75]{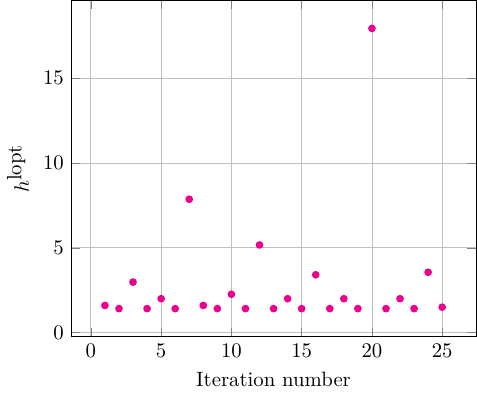}
\par\end{centering}
}\hfill{}\subfloat[$N=50$]{\begin{centering}
\includegraphics[scale=0.75]{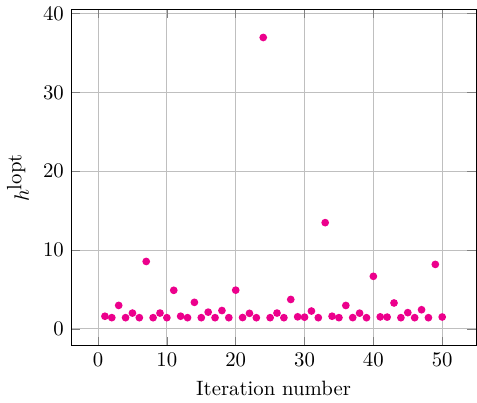}
\par\end{centering}
}\caption{Locally optimal stepsizes $h^{\textrm{lopt}}$
vs.\ iteration number for $N=5,10,25,50$ with $L=1$.
These optimized methods utilize long steps $h_i>2/L$ for some $i$, much alike how Young's method \cite{young1953} uses long steps satisfying $1 <h_{i}<L/\mu$ for some $i$ to achieve an accelerated rate for $L$-smooth and $\mu$-strongly convex quadratics.
\label{fig:lopt_step_size_AWM}}
\end{figure}

{ 
\subsubsection{
Acceleration is impossible without long steps
}
\label{subsubsec:lowero1k}
Consider the univariate functions
\begin{align*}
f_1(x)&=
\left\{
\begin{array}{ll}
\frac{LR}{2Nh+1}|x|-\frac{LR^2}{2(2Nh+1)^2}&\text{if }|x|\ge \frac{R}{2Nh+1}\\
\frac{L}{2}x^2&\text{otherwise}
\end{array}
\right.\\
f_2(x)&=\frac{L}{2}x^2,
\end{align*}
which are $L$-smooth convex functions minimized at $x_\star=0$.

Consider gradient descent $x_{i}=x_{i-1}-(h_{i-1}/L) \nabla f_j(x_{i-1})$ where
for $j=1,2$ and $i\in[1:N]$ with starting point $x_0=R$.
For the sake of simplicity, consider the constant stepsize $h_i = h$ for all $i\in[0:N-1]$. 
It is straightforward to check that the objective values at iteration $N$ are
\begin{align*}
f_1(x_N)&=\frac{LR^2}{2}\frac{1}{2Nh+1}
\quad\text{(for $0\le h\le 2$)}
\\
f_2(x_N)&=\frac{LR^2}{2}(1-h)^2.
\end{align*}
The analysis of $f_1$ shows that acceleration, if possible, would require the algorithm to take long steps exceeding the range $h<2$. On the other hand, the analysis of $f_2$ shows that the constant-step gradient descent cannot use a stepsize exceeding $h<2$ as otherwise one gets divergence. 

For the general case when $h_i $ is not constant, a similar line of reasoning with $f_1$ shows that acceleration without momentum is possible only if $\{h_{i}\}_{i\in[0:N-1]}$ exceeds $h_i<2$ for some $i\in[0:N-1]$.
The reasoning and the counter examples $f_1$ and $f_2$ are based on \cite[Theorem~2]{drori2014performance} and \cite[$\mathsection$4.1.1]{taylor2017smooth}.  
}

%% file: Sections/Images/AWM/AWM_vs_Drori_performance_comparison_LogLogScale.tex
\begin{tikzpicture}
\begin{loglogaxis}[legend style= {at={(0.75,1)},anchor=north, font=\tiny}, scale only axis, height=3.5cm, width=0.33*\textwidth, xmajorgrids, ymajorgrids, xlabel={$N$ (log scale)}, ylabel={$f(x_N) - f(x_\star)$ (log scale)}]
       \addplot+[no marks, style={{ultra thick}}, color={blue}]
        coordinates {
            (1,0.166666662878905)
            (2,0.09999999029850636)
            (3,0.07142857181180012)
            (4,0.05555555588637967)
            (5,0.04545454821049464)
            (6,0.03846153834611883)
            (7,0.03333333052329084)
            (8,0.029411763254097924)
            (9,0.02631578974090786)
            (10,0.023809523691808544)
            (11,0.02173913231439794)
            (12,0.02000000013893638)
            (13,0.018518517118125934)
            (14,0.017241376867162612)
            (15,0.016129032110438105)
            (16,0.015151515790992111)
            (17,0.014285714285486409)
            (18,0.013513513512704389)
            (19,0.012820512687607748)
            (20,0.012195121951011759)
            (21,0.011627906911951374)
            (22,0.011111110560599035)
            (23,0.010638296845389429)
            (24,0.01020408153124342)
            (25,0.009803921568442273)
            (26,0.009433961957771858)
            (27,0.009090909088300776)
            (28,0.008771929260270752)
            (29,0.008474576271040018)
            (30,0.00819671998566875)
            (31,0.007936506977736553)
            (32,0.0076923071745407)
            (33,0.007462686536263421)
            (34,0.007246376509150704)
            (35,0.00704224784556508)
            (36,0.0068493114851270655)
            (37,0.006666666653992947)
            (38,0.006493506493458986)
            (39,0.00632911391715595)
            (40,0.006172839461073482)
            (41,0.00602409619990842)
            (42,0.005882352511171741)
            (43,0.005747124401857555)
            (44,0.005617973156054338)
            (45,0.005494501288822075)
            (46,0.005376339799100511)
            (47,0.005263149977359086)
            (48,0.005154637634104015)
            (49,0.005050502189660276)
            (50,0.004950490301662746)
        }
        ;
    \addlegendentry {stepsize $1/L$}
        \addplot+[no marks, style={{ultra thick}}, color={teal}]
        coordinates {
            (1,0.12500000209151407)
            (2,0.06735532243512904)
            (3,0.0453638496046663)
            (4,0.033975929299897796)
            (5,0.02707013348718289)
            (6,0.022456073997411802)
            (7,0.019164091648483667)
            (8,0.016701361907705815)
            (9,0.01479186685683739)
            (10,0.01326935378961052)
            (11,0.012027599734427385)
            (12,0.010996130751029629)
            (13,0.010125808020736933)
            (14,0.009381998132357525)
            (15,0.008739082410372185)
            (16,0.008177943847311691)
            (17,0.007683949891146401)
            (18,0.007245816222902167)
            (19,0.006855055002110819)
            (20,0.006503212337052077)
            (21,0.0061858880572582845)
            (22,0.005897850351858571)
            (23,0.005635617813669572)
            (24,0.00539505223389345)
            (25,0.005174280671324655)
            (26,0.004970828780148045)
            (27,0.004782632482562268)
            (28,0.004608161950466669)
            (29,0.004445846685000348)
            (30,0.004294556832298904)
            (31,0.00415317993169601)
            (32,0.004020814642716015)
            (33,0.003896551042392898)
            (34,0.003779706260662491)
            (35,0.0036696376337141054)
            (36,0.0035657656265870963)
            (37,0.0034675960838675465)
            (38,0.0033750905324250392)
            (39,0.00328664390085251)
            (40,0.003203097456045503)
            (41,0.003123473540837639)
            (42,0.003047822986360286)
            (43,0.0029757386979332376)
            (44,0.002906974827136179)
            (45,0.002841307608409582)
            (46,0.0027788660025964153)
            (47,0.0027184665198351607)
            (48,0.0026609295200268356)
            (49,0.0026057756865681255)
            (50,0.0025528520660105455)
        }
        ;
    \addlegendentry {stepsize in \cite{taylor2017smooth}}
    \addplot+[no marks, style={{ultra thick}}, color={red}]
        coordinates {
            (1,0.12512881469455195)
            (2,0.06562956614501574)
            (3,0.043276013183385365)
            (4,0.03158895842592049)
            (5,0.024523824372707486)
            (6,0.01979804767938377)
            (7,0.016468321113435616)
            (8,0.014044081884960619)
            (9,0.012199539030785344)
            (10,0.010778290063013797)
            (11,0.00964275453011679)
            (12,0.00871937439141247)
            (13,0.007947117770036769)
            (14,0.007288620256492238)
            (15,0.006723385221382876)
            (16,0.00622107311893578)
            (17,0.0057753204564314625)
            (18,0.005375928275064163)
            (19,0.005018543170206344)
            (20,0.004697259718435158)
            (21,0.0044077434376316015)
            (22,0.004147823040200636)
            (23,0.003915297144929482)
            (24,0.0037071125806800607)
            (25,0.0035171473367649766)
            (26,0.003344400224906669)
            (27,0.003186804114963152)
            (28,0.0030447587482080804)
            (29,0.0029169902903467453)
            (30,0.0027999890538774816)
            (31,0.0026926271354745083)
            (32,0.002593732697356943)
            (33,0.0025048688556175277)
            (34,0.002422290021009871)
            (35,0.0023456103574459593)
            (36,0.0022717150535531755)
            (37,0.002200147478985533)
            (38,0.002128604082478878)
            (39,0.002059714082729395)
            (40,0.0019944183984015908)
            (41,0.0019329859644700803)
            (42,0.0018748211347185994)
            (43,0.0018196855396110986)
            (44,0.0017672773037448156)
            (45,0.0017174557285748935)
            (46,0.0016700768517786942)
            (47,0.001624990076185959)
            (48,0.0015821623545387515)
            (49,0.0015414982595569216)
            (50,0.0015029346188281322)
        }
        ;
    \addlegendentry {$\textrm{BnB-PEP}$}
\end{loglogaxis}
\end{tikzpicture}

%% file: Sections/Images/AWM/AWM_timing_performance.tex
\begin{tikzpicture}
\begin{axis}[legend style= {at={(0.14,1)},anchor=north, font=\tiny}, scale only axis, height=3.5cm, width=0.33*\textwidth, xmajorgrids, ymajorgrids, xlabel={$N$}, ylabel={Solution time (s)}]
    \addplot[only marks, mark size={0.5pt}, style={{ultra thick}}, color={red}, mark={*}]
        coordinates {
            (1,0.027555863)
            (2,0.032388035)
            (3,0.047763634)
            (4,0.062770943)
            (5,0.0803967)
            (6,0.093045279)
            (7,0.120842114)
            (8,0.226563932)
            (9,0.283149162)
            (10,0.338831766)
            (11,0.466318493)
            (12,0.599645667)
            (13,0.83010181)
            (14,1.101309054)
            (15,1.825797928)
            (16,1.77327329)
            (17,3.86729853)
            (18,2.803719523)
            (19,3.357504113)
            (20,5.139137331)
            (21,5.82537542)
            (22,8.740043069)
            (23,8.426102187)
            (24,10.86475903)
            (25,9.833601177)
            (26,19.693708254)
            (27,20.578535247)
            (28,22.761463785)
            (29,30.36708108)
            (30,20.036467622)
            (31,26.558231465)
            (32,45.770337795)
            (33,30.395496869)
            (34,28.858312822)
            (35,47.344663504)
            (36,100.528365044)
            (37,54.69293109)
            (38,84.605542658)
            (39,139.003598042)
            (40,73.057880997)
            (41,162.29801287)
            (42,115.411673982)
            (43,104.848576032)
            (44,182.247982108)
            (45,157.005708348)
            (46,186.735760629)
            (47,255.895705108)
            (48,212.159618103)
            (49,392.951958578)
            (50,397.227323194)
        }
        ;
\end{axis}
\end{tikzpicture}

%% file: Sections/Images/AWM/BnB_PEP_vs_LogLogfit.tex
\begin{tikzpicture}
\begin{loglogaxis}[legend style= {at={(0.75,1)},anchor=north, font=\tiny}, scale only axis, height=3.5cm, width=0.33*\textwidth,  xmajorgrids, ymajorgrids, xlabel={$N$ (log scale)}, ylabel={$f(x_N) - f(x_\star)$ (log scale)}]
    \addplot+[no marks, style={{ultra thick}}, color={cyan}]
        coordinates {
            (1,0.156)
            (2,0.06894634742296636)
            (3,0.0427637478662199)
            (4,0.03047178732672038)
            (5,0.02342814661697403)
            (6,0.018900027035208548)
            (7,0.015761556648458148)
            (8,0.013467425869402605)
            (9,0.011722680330548893)
            (10,0.010354391898269779)
            (11,0.009254735587201773)
            (12,0.008353127117134266)
            (13,0.007601500516302207)
            (14,0.006966036926995903)
            (15,0.006422277916032758)
            (16,0.00595211424926203)
            (17,0.005541862790149907)
            (18,0.005180999941015372)
            (19,0.004861304851269322)
            (20,0.004576266033151643)
            (21,0.004320661759575682)
            (22,0.004090257789121179)
            (23,0.0038815859529707365)
            (24,0.0036917795146547464)
            (25,0.003518449063502768)
            (26,0.0033595877918771166)
            (27,0.0032134983716146683)
            (28,0.0030787359110889256)
            (29,0.0029540630254920345)
            (30,0.002838414131061788)
            (31,0.0027308668338819272)
            (32,0.002630618826479542)
            (33,0.0025369690975617257)
            (34,0.002449302546795418)
            (35,0.0023670773081526882)
            (36,0.002289814243151372)
            (37,0.0022170881840826466)
            (38,0.0021485205974651695)
            (39,0.002083773406949282)
            (40,0.002022543768087136)
            (41,0.001964559628696641)
            (42,0.0019095759408514748)
            (43,0.001857371415953494)
            (44,0.001807745734475917)
            (45,0.0017605171380001097)
            (46,0.001715520344010425)
            (47,0.0016726047342504875)
            (48,0.0016316327758101041)
            (49,0.0015924786409138162)
            (50,0.001555026996937644)
        }
        ;
    \addlegendentry {$0.156/{N^{1.178}}$}
    \addplot+[no marks, style={{ultra thick, dashed}}, color={red}]
        coordinates {
            (1,0.1256566205662219)
            (2,0.06528134265059801)
            (3,0.04296258648300468)
            (4,0.0314092994399884)
            (5,0.02444484984264801)
            (6,0.01978465459518769)
            (7,0.016518292814685556)
            (8,0.014112049067437963)
            (9,0.012267186973534159)
            (10,0.010830264699706547)
            (11,0.009677753326700088)
            (12,0.008735286245867458)
            (13,0.007946575331581582)
            (14,0.007275200944710736)
            (15,0.006700494312853945)
            (16,0.006196404358284833)
            (17,0.005752899867299551)
            (18,0.005358738114527343)
            (19,0.005004690169416533)
            (20,0.004687271228052972)
            (21,0.004401317310560984)
            (22,0.00414467223649091)
            (23,0.0039134176873351426)
            (24,0.0037045194048323793)
            (25,0.0035166570798522375)
            (26,0.0033484846315400985)
            (27,0.0031977642318810834)
            (28,0.00305982860097494)
            (29,0.002933008467692449)
            (30,0.0028144979445567236)
            (31,0.0027045500887292262)
            (32,0.0026017882693560156)
            (33,0.002505827964541633)
            (34,0.0024161886516167585)
            (35,0.002332226836144157)
            (36,0.002253482021947919)
            (37,0.0021794901858584862)
            (38,0.0021098547462944433)
            (39,0.002044205316981468)
            (40,0.0019821837916262066)
            (41,0.0019235149533895614)
            (42,0.0018679153344809173)
            (43,0.0018151472482477063)
            (44,0.0017649926033543619)
            (45,0.0017172587304382092)
            (46,0.0016717721316462052)
            (47,0.0016283731656785)
            (48,0.0015869237770066956)
            (49,0.0015472948146156057)
            (50,0.001509371249819283)
        }
        ;
    \addlegendentry {$\textrm{BnB-PEP}$}
\end{loglogaxis}
\end{tikzpicture}

%% file: Sections/Images/AWM/BnB_PEP_vs_fit.tex
\begin{tikzpicture}
\begin{axis}[legend style= {at={(0.75,1)},anchor=north, font=\tiny}, scale only axis, height=3.5cm, width=0.33*\textwidth, xmajorgrids, ymajorgrids,    y tick label style={
        /pgf/number format/.cd,
        fixed,
        precision=2,
        /tikz/.cd
    }, xlabel={$N$}, ylabel={$f(x_N) - f(x_\star)$}]
    \addplot+[no marks, style={{ultra thick}}, color={cyan}]
        coordinates {
            (1,0.156)
            (2,0.06894634742296636)
            (3,0.0427637478662199)
            (4,0.03047178732672038)
            (5,0.02342814661697403)
            (6,0.018900027035208548)
            (7,0.015761556648458148)
            (8,0.013467425869402605)
            (9,0.011722680330548893)
            (10,0.010354391898269779)
            (11,0.009254735587201773)
            (12,0.008353127117134266)
            (13,0.007601500516302207)
            (14,0.006966036926995903)
            (15,0.006422277916032758)
            (16,0.00595211424926203)
            (17,0.005541862790149907)
            (18,0.005180999941015372)
            (19,0.004861304851269322)
            (20,0.004576266033151643)
            (21,0.004320661759575682)
            (22,0.004090257789121179)
            (23,0.0038815859529707365)
            (24,0.0036917795146547464)
            (25,0.003518449063502768)
            (26,0.0033595877918771166)
            (27,0.0032134983716146683)
            (28,0.0030787359110889256)
            (29,0.0029540630254920345)
            (30,0.002838414131061788)
            (31,0.0027308668338819272)
            (32,0.002630618826479542)
            (33,0.0025369690975617257)
            (34,0.002449302546795418)
            (35,0.0023670773081526882)
            (36,0.002289814243151372)
            (37,0.0022170881840826466)
            (38,0.0021485205974651695)
            (39,0.002083773406949282)
            (40,0.002022543768087136)
            (41,0.001964559628696641)
            (42,0.0019095759408514748)
            (43,0.001857371415953494)
            (44,0.001807745734475917)
            (45,0.0017605171380001097)
            (46,0.001715520344010425)
            (47,0.0016726047342504875)
            (48,0.0016316327758101041)
            (49,0.0015924786409138162)
            (50,0.001555026996937644)
        }
        ;
    \addlegendentry {$0.156/{N^{1.178}}$}
    \addplot+[no marks, style={{dashed, ultra thick}}, color={red}, mark={*}]
        coordinates {
            (1,0.12500000039401646)
            (2,0.06594597669442348)
            (3,0.04289324973901232)
            (4,0.03256516933794558)
            (5,0.02439360612424908)
            (6,0.019814662514600474)
            (7,0.01651332612698534)
            (8,0.01403712301081951)
            (9,0.01218149925018591)
            (10,0.010707803345049002)
            (11,0.009565888232158155)
            (12,0.008697583409315162)
            (13,0.007952102291835897)
            (14,0.007297301195680385)
            (15,0.006730482071052881)
            (16,0.006252646825576575)
            (17,0.005811266656901224)
            (18,0.005375016068984765)
            (19,0.004988492057531461)
            (20,0.004678377044818489)
            (21,0.004432075710848663)
            (22,0.004182361213651744)
            (23,0.003927793416225444)
            (24,0.003685188468346122)
            (25,0.0035042664229482465)
            (26,0.0033476311067778732)
            (27,0.003197565715268286)
            (28,0.003042783220507925)
            (29,0.002897332765325484)
            (30,0.0027996149048139653)
            (31,0.0027113029867317606)
            (32,0.002611154164034634)
            (33,0.0025004251428767336)
            (34,0.0023944918234701054)
            (35,0.0023289238102664417)
            (36,0.002269349268869964)
            (37,0.0022063087714561456)
            (38,0.0021491496850726713)
            (39,0.0020823484078409675)
            (40,0.0020174824338334476)
            (41,0.0019410515750403516)
            (42,0.0018646784364243156)
            (43,0.0018059789720844554)
            (44,0.0017544334125811612)
            (45,0.001701395601716227)
            (46,0.0016528475942262394)
            (47,0.0016101924826737329)
            (48,0.0015764160943991692)
            (49,0.0015510946954243204)
            (50,0.0015327205414233504)
        }
        ;
    \addlegendentry {$\textrm{BnB-PEP}$}
\end{axis}
\end{tikzpicture}

%% file: Sections/smooth_nonconvex_gradient_reduction.tex
\subsection{Optimal method for reducing gradient of smooth nonconvex functions
\label{subsec:Smooth-nonconvex-gradient-reduction}}

In this section, we construct an optimal FSFOM for decreasing the
gradient of $L$-smooth nonconvex functions. Formally, we choose the
function class
\[
\mathcal{F}=\{f\,|\,f\in\LSmth,\,f\text{ has a global minimizer }x_{\star}\},
\]
performance measure\footnote{In the nonconvex setup, performance measures such as $f(x_{N})-f(x_{\star})$
or $\|\nabla f(x_{N})\|^{2}$ may not converge to zero as $N\to\infty$
\cite[page 3, paragraph 2]{davis2019stochastic}\cite[Remark after Theorem 1]{drori2020complexity}.} 
\[
\mathcal{E}=\min_{i\in[0:N]}\|\nabla f(x_{i})\|^{2},
\]
and initial condition $\mathcal{C}=f(x_{0})-f(x_{\star})-(1/2)R^{2}\leq0$
with $R>0$. We parameterize FSFOMs in $\mathcal{M}_{N}$ as 
\begin{equation}
x_{i}=x_{i-1}-\sum_{j=0}^{i-1}\frac{h_{i,j}}{L}\nabla f(x_{j})=x_{0}-\sum_{j=0}^{i-1}\frac{\overline{h}_{i,j}}{L}\nabla f(x_{j}),
\label{eq:ncvx-ogm-g-fsfom}
\end{equation}

for $i\in[1:N]$. We solve the following instance of \eqref{eq:main-min-max-problem}:
\[
\begin{array}{ll}
\mathcal{R}^{\star}\left(\mathcal{M}_{N},\mathcal{E},\mathcal{F},\mathcal{C}\right)=\underset{M\in\mathcal{M}_{N}}{\mbox{minimize}} & \mathcal{R}\left(M,\mathcal{E},\mathcal{F},\mathcal{C}\right)\end{array}.
\]

\paragraph{Derivation of BnB-PEP-QCQP. }

Following $\mathsection$\ref{subsec:Generalization-of-QCQO-framework}  Step (i),
formulate the inner optimization problem \eqref{eq:worst-case-pfm}
as 
\begin{align*}
 & \mathcal{R}\left(M,\mathcal{E},\mathcal{F},\mathcal{C}\right)\\
 & =\left(\begin{array}{ll}
\textrm{maximize} & \min_{i\in[0:N]}\|\nabla f(x_{i})\|^{2}\\
\textrm{subject to} & f\in\mathcal{F}_{-L,L}\\
 & f(x)\geq f(x_{\star}),\quad\textup{ for all }x\in\rl^{d},\\
 & x_{i}=x_{0}-\frac{1}{L}\sum_{j=0}^{i-1}\overline{h}_{i,j}\nabla f(x_{j})\quad i\in[1:N],\\
 & f(x_{0})-f(x_{\star})\leq R^{2},\\
 & x_{\star}=0,\;f(x_{\star})=0,
\end{array}\right)%
\end{align*}
where $f$ and $x_{0},\ldots,x_{N}$ are the decision variables. Write
$\overline{h}=\{\overline{h}_{i,j}\}_{0\leq j<i\leq N}$. To follow
the interpolation argument of $\mathsection$\ref{subsec:Generalization-of-QCQO-framework}  Step (ii),
we use the following interpolation result.
\begin{lem}[Interpolation
inequality for $\mathcal{F}_{-L,L}$\label{Thm:Interpolation-inequality-smth-ncvx}]
Let $I$ be a finite index set, and let
$\{(x_{i},g_{i},f_{i})\}_{i\in I\cup\{\star\}}\subseteq\mathbb{R}^{d}\times\mathbb{R}^{d}\times\mathbb{R}$.
Let $L>0$. There exists $f\in\mathcal{F}_{-L,L}$
satisfying $f(x)\ge f(x_\star)=f_\star$ for all $x\in \rl^d$,
$f(x_{i})=f_{i}$ for all $i\in I$, and $g_{i}=\nabla f(x_{i})$
for all $i\in I$ if and only if
\footnote{
The first condition can be viewed as a discretization of the following condition \cite[Theorem 3.10]{taylor2017exact}:
$f\in\mathcal{F}_{-L,L}$ if and only if
\begin{align*}
f(y) & \geq f(x)-\frac{L}{4}\|x-y\|^{2}+\frac{1}{2}\left\langle \nabla f(x)+\nabla f(y)\mid y-x\right\rangle +\frac{1}{4L}\|\nabla f(x)-\nabla f(y)\|^{2},\quad
\forall x,y\in\rl^{d}.
\end{align*}
}
\[
\begin{alignedat}{1}f_{i} & \geq f_{j}-\frac{L}{4}\|x_{i}-x_{j}\|^{2}+\frac{1}{2}\left\langle g_{i}+g_{j}\mid x_{i}-x_{j}\right\rangle +\frac{1}{4L}\|g_{i}-g_{j}\|^{2},\quad\forall\,i,j\in I\cup\{\star\},\\
f_{\star} & \leq f_{i}-\frac{1}{2L}\|g_{i}\|^{2},\quad \forall\, i\in I,\\
g_\star&=0.
\end{alignedat}
\]
\end{lem}
\begin{proof} The result follows from translating \cite[Theorem 7]{drori2020complexity} into the form of \cite[Theorem 3.10]{taylor2017exact}.
(Note that journal version of \cite[Theorem 3.10]{taylor2017exact} has a sign error that was corrected in its updated arXiv version.)
\end{proof}
Now formulate the inner problem as
\begin{align*}
 & \mathcal{R}\left(M,\mathcal{E},\mathcal{F},\mathcal{C}\right)\\
= & \left(\begin{array}{l}
\textrm{maximize}\quad t\\
\textrm{subject to}\\
t\leq\|\nabla f(x_{i})\|^{2},\quad i\in[0:N],\\
f_{i}\geq f_{j}-\frac{L}{4}\|x_{i}-x_{j}\|^{2}+\frac{1}{2}\left\langle g_{i}+g_{j}\mid x_{i}-x_{j}\right\rangle +\frac{1}{4L}\|g_{i}-g_{j}\|^{2},\quad i,j\in I_{N}^{\star}:i\neq j,\\
f_{\star}\leq f_{i}-\frac{1}{2L}\|g_{i}\|^{2},\quad i\in[0:N],\\
g_{\star}=0,\;x_{\star}=0,\;f_{\star}=0,\\
x_{i}=x_{0}-\frac{1}{L}\sum_{j=0}^{i-1}\overline{h}_{i,j}\nabla f(x_{j}),\quad i\in[1:N],\\
f_{0}-f_{\star}\leq R^{2},
\end{array}\right)%
\end{align*}
where $\{x_{i},g_{i},f_{i}\}_{i\in I_{N}}\subseteq\rl^{d}\times\rl^{d}\times\rl$
and $t\in\rl$ are the decision variables. Following $\mathsection$\ref{subsec:Generalization-of-QCQO-framework}  Step (iii),
implement the Grammian transformation. Define the Grammian matrices
$H\in\rl^{d\times(N+2)}$, $G\in\mathbb{S}_{+}^{N+2}$, and $F\in\rl^{1\times(N+1)}$
using the same equations in \eqref{eq:grammian-mats}, $\{\mathbf{x}_{i},\mathbf{g}_{i},\mathbf{f}_{i}\}_{i\in I_{N}^{\star}}$
using the same encoding as \eqref{eq:bold-vec}, except for $\{{\bf x}_{i}\}_{i\in[1:N]}$,
which we define as 
\begin{align*}
 & \mathbf{x}_{i}=\mathbf{x}_{0}-\frac{1}{L}\sum_{j=0}^{i-1}\overline{h}_{i,j}\mathbf{g}_{j}\in\rl^{N+2},\quad i\in[1:N].
\end{align*}
Note, $\mathbf{x}_{i}$ is linearly parameterized by $\overline{h}$.
The matrices $B_{i,j}$, $C_{i,j}$, and $a_{i,j}$ are the same as
in \eqref{eq:ABCa-mat-vec} except that they are now parameterized
by $\overline{h}$. For $i,j\in I_{N}^{\star},$ define 
\[
\widetilde{A}_{i,j}(\overline{h})=(\mathbf{g}_{i}+\mathbf{g}_{j})\odot(\mathbf{x}_{i}-\mathbf{x}_{j}),
\]
so that 
\[
\tr G\widetilde{A}_{i,j}(\overline{h})=\left\langle g_{i}+g_{j}\mid x_{i}-x_{j}\right\rangle 
\]
holds. Under the large-scale assumption $d\geq N+2$, we equivalently
formulate the inner problem as the SDP:
\begin{align*}
 & \mathcal{R}\left(M,\mathcal{E},\mathcal{F},\mathcal{C}\right)\\
= & \left(\begin{array}{l}
\textrm{maximize}\quad t\\
\textrm{subject to}\\
t\leq\tr GC_{i,\star},\quad i\in[0:N] \quad \rhd \textsf{\;dual var.\;} \eta_{i}\geq0 \\
Fa_{i,j}-\frac{L}{4}\tr GB_{i,j}(\overline{h})+\frac{1}{2}\tr G\widetilde{A}_{i,j}(\overline{h})+\frac{1}{4L}\tr GC_{i,j}\leq0,\quad i,j\in I_{N}^{\star}:i\neq j,\quad \quad \rhd \textsf{\;dual var.\;} \lambda_{i,j}\geq0 \\
Fa_{i,\star}+\frac{1}{2L}\tr GC_{i,\star}\leq0,\quad i\in[0:N], \quad \rhd \textsf{\;dual var.\;} \tau_{i}\geq0 \\
-G\preceq0,\quad \rhd \textsf{\;dual var.\;} Z\succeq0 \\
Fa_{\star,0}\leq R^{2},\quad \quad \rhd \textsf{\;dual var.\;} \nu\geq0
\end{array}\right)%
\end{align*}
where $F\in\rl^{1\times(N+1)}$ and $G\in\rl^{(N+2)\times(N+2)}$
are the decision variables. %
Following
$\mathsection$\ref{subsec:Generalization-of-QCQO-framework}  Step (iv),
construct the dual: 
\begin{align*}
 & \mathcal{R}\left(M,\mathcal{E},\mathcal{F},\mathcal{C}\right)\\
 & =\left(\begin{array}{l}
\textrm{minimize}\quad\nu R^{2}\\
\textrm{subject to}\\
\sum_{i,j\in I_{N}^{\star}:i\neq j}\lambda_{i,j}a_{i,j}+\sum_{i\in I_{N}^{\star}}\tau_{i}a_{i,\star}+\nu a_{\star,0}=0,\\
-\sum_{i\in[0:N]}\eta_{i}C_{i,\star}+\sum_{i,j\in I_{N}^{\star}:i\neq j}\lambda_{i,j}\left(-\frac{L}{4}B_{i,j}(\overline{h})+\frac{1}{2}\widetilde{A}_{i,j}(\overline{h})+\frac{1}{4L}C_{i,j}\right)\\
\quad\quad\quad\quad\quad+\frac{1}{2L}\sum_{i\in[0:N]}\tau_{i}C_{i,\star}=Z,\\
\sum_{i\in[0:N]}\eta_{i}=1,\\
Z\succeq0,\\
\nu\geq0,\;\tau_{i}\geq0,\;\lambda_{i,j}\geq0,\quad i,j\in I_{N}^{\star}:i\neq j,\\
\eta_{i}\geq0,\quad i\in[0:N],
\end{array}\right)%
\end{align*}
where $\nu$, $\lambda$, $\eta$, $\tau,$ and $Z$ are the decision
variables. Assume that strong duality holds. Finally, following $\mathsection$\ref{subsec:Generalization-of-QCQO-framework}  Step (v),
use Lemma~\ref{Lem:quadratic-characterization-psd-1} to pose \eqref{eq:main-min-max-problem}
as the following BnB-PEP-QCQP: 
\begin{align}
 & 
 \mathcal{R}^\star\left(\mathcal{M}_N,\mathcal{E},\mathcal{F},\mathcal{C}\right)
 \nonumber \\
 & =\left(\begin{array}{l}
\textrm{minimize}\quad\nu R^{2}\\
\textrm{subject to}\\
\sum_{i,j\in I_{N}^{\star}:i\neq j}\lambda_{i,j}a_{i,j}+\sum_{i\in I_{N}^{\star}}\tau_{i}a_{i,\star}+\nu a_{\star,0}=0,\\
-\sum_{i\in[0:N]}\eta_{i}C_{i,\star}+\sum_{i,j\in I_{N}^{\star}:i\neq j}\lambda_{i,j}\left(-\frac{L}{4}\Theta_{i,j}+\frac{1}{2}\widetilde{A}_{i,j}(\overline{h})+\frac{1}{4L}C_{i,j}\right)\\
\quad\quad\quad\quad\quad+\frac{1}{2L}\sum_{i\in[0:N]}\tau_{i}C_{i,\star}=Z,\\
\sum_{i\in[0:N]}\eta_{i}=1,\\
P\text{ is lower triangular with nonnegative diagonals},\\
PP^{\top}=Z,\\
\Theta_{i,j}=B_{i,j}(\overline{h}),\quad i,j\in I_{N}^{\star}:i\neq j,\\
\nu\geq0,\;\tau_{i}\geq0,\;\lambda_{i,j}\geq0,\quad i,j\in I_{N}^{\star}:i\neq j,\\
\eta_{i}\geq0,\quad i\in[0:N],
\end{array}\right)\label{eq:BnB-PEP-final-smth-ncvx-grad-reduction}
\end{align}
where $\nu$, $\lambda$, $\eta$, $\tau$, $Z$, $P$, and $\{\Theta_{i,j}\}_{i,j\in I_{N}^{\star}:i\ne j}$
are the decision variables. Note that $\{\Theta_{i,j}\}_{i,j\in I_{N}^{\star}:i\ne j}$
is introduced as a separate decision variable to formulate the cubic
constraints arising from $B_{i,j}(\overline{h})$ as quadratic constraints.

\paragraph{Numerical results. }
Tables~\ref{tab:BnB-PEP-result-smth-ncvx-grad-reduction} and \ref{tab:BnB-PEP-result-smth-ncvx-grad-reduction2} present the results of solving \eqref{eq:BnB-PEP-final-smth-ncvx-grad-reduction} with $L=1$, $R=1$, and $N=1,\ldots,5$ using the BnB-PEP Algorithm.
We compare the computed optimal FSFOM  with the FSFOMs defined by
\begin{equation}
h_{i,j}=\begin{cases}
1/L & \textrm{if }j=i-1,\\
0 & \textrm{if }j\in[0:i-2],
\end{cases}\tag{GD}\label{eq:GD}
\end{equation}
which is gradient descent and
\begin{equation}
h_{i,j}=\begin{cases}
2/(\sqrt{3}L) & \textrm{if }j=i-1,\\
0 & \textrm{if }j\in[0:i-2],
\end{cases}\tag{AKZ}\label{eq:AKZ}
\end{equation}
which was proposed by Abbaszadehpeivasti, de Klerk, and Zamani and has the prior state-of-the-art rate \cite{abbaszadehpeivasti2021exact}.
To clarify, the stepsizes $h=\{h_{i,j}\}_{0\leq j<i\leq N}$ and $\overline{h}=\{\overline{h}_{i,j}\}_{0\leq j<i\leq N}$ are as defined in \eqref{eq:ncvx-ogm-g-fsfom}.
We obtain the optimal $\overline{h}^{\star}$ from the BnB-PEP Algorithm, solve for $h^{\star}$ with \eqref{eq:TriLinSys-General}, and present $h^{\star}$ in the table.
The stepsizes presented in Table~\ref{tab:BnB-PEP-result-smth-ncvx-grad-reduction2} are certified to be globally optimal by Stage 3. We again observe that the stepsizes obtained at Stage 2, denoted by $\overline{h}^{\textrm{lopt}}$, were already near-optimal.

Figure~\ref{fig:Ncvx-grad-red-performance-Shuvo-vs-DeKlerk}(a) shows that the computed stepsizes $h^\star$ outperforms \eqref{eq:AKZ} on the worst-case guarantee of $\min_{i\in[0:N]}\|\nabla f(x_{i})\|^{2}$.
To ensure the comparison is precise, we set the precision of the solver to $10^{-10}$. { We observe in Figure~\ref{fig:Ncvx-grad-red-performance-Shuvo-vs-DeKlerk}(a) that the performance improvement diminishes as $N$ increases, which suggests that it will go to zero as $N\to \infty$. This observation leads conjecture that \eqref{eq:AKZ} has the exact optimal constant for the leading order term. 

\begin{conjecture}
The optimal FSFOM for reducing gradient of smooth nonconvex functions satisfies 
\[
\min_{i \in [0:N]}\|\nabla f(x_i)\|^2
\le\frac{6\sqrt{3} L (f(x_0)-f_\star)}{8N + 3 \sqrt{3} }+o(1/N)
\]
where the leading term corresponds to the rate for \eqref{eq:AKZ} \cite[Theorem~2]{abbaszadehpeivasti2021exact}.
\end{conjecture} }

Also, Figure~\ref{fig:Ncvx-grad-red-performance-Shuvo-vs-DeKlerk}(b)  shows the solution time to compute the locally optimal stepsizes $h^{\textrm{lopt}}$.

\begin{table}
\centering{}{\footnotesize{}}%
\begin{tabular}{>{\centering}b{2em}>{\centering}b{4em}>{\centering}b{4em}>{\centering}b{5em}>{\centering}b{5em}>{\centering}b{5em}>{\raggedright}m{8.5em}}
\toprule 
\multirow{2}{2em}{\centering{}{\footnotesize{}$N$}} & \multirow{2}{4em}{\centering{}{\footnotesize{}\# variables}} & \multirow{2}{4em}{\centering{}{\footnotesize{}\# constraints}} & \multicolumn{3}{c}{{\footnotesize{}Worst-case $\min_{i\in[0:N]}\|\nabla f(x_{i})\|^{2}$}} & \multirow{2}{8.5em}{\centering{}{\footnotesize{}Runtime of the BnB-PEP Algorithm}}\tabularnewline
\cmidrule{4-6} \cmidrule{5-6} \cmidrule{6-6} 
 &  &  & {\footnotesize{}Optimal} & {\footnotesize{}GD} & {\footnotesize{}AKZ} & \tabularnewline
\midrule
\midrule 
\centering{}{\footnotesize{}$1$} & {\footnotesize{}$60$ } & {\footnotesize{}$70$ } & {\footnotesize{}$0.7875254$ } & {\footnotesize{}$0.8$ } & {\footnotesize{}$0.7875254$ } & \centering{}{\footnotesize{}$0.04$ s}\tabularnewline
\midrule
\centering{}{\footnotesize{}$2$} & {\footnotesize{}$162$ } & {\footnotesize{}$177$ } & {\footnotesize{}$0.4902031$ } & {\footnotesize{}$0.5$ } & {\footnotesize{}$0.4902920$ } & \centering{}{\footnotesize{}$0.41$ s}\tabularnewline
\midrule
\centering{}{\footnotesize{}$3$} & {\footnotesize{}$365$ } & {\footnotesize{}$386$ } & {\footnotesize{}$0.3558535$ } & {\footnotesize{}$0.363636$ } & {\footnotesize{}$0.3559478$ } & \centering{}{\footnotesize{}$9.79$ s}\tabularnewline
\midrule
\centering{}{\footnotesize{}$4$} & {\footnotesize{}$723$ } & {\footnotesize{}$751$ } & {\footnotesize{}$0.2793046$ } & {\footnotesize{}$0.285714$ } & {\footnotesize{}$0.2793919$ } & \centering{}{\footnotesize{}$69.2$ s}\tabularnewline
\midrule
\centering{}{\footnotesize{}$5$} & {\footnotesize{}$1302$ } & {\footnotesize{}$1338$ } & {\footnotesize{}$0.2298589$ } & {\footnotesize{}$0.235294$ } & {\footnotesize{}$0.2299378$ } & \centering{}{\footnotesize{}$607.52$ s}\tabularnewline
\midrule
\centering{}{\footnotesize{}$10$} & {\footnotesize{}$4138$ } & {\footnotesize{}$4128$ } & {\footnotesize{}$0.1219308$ } & {\footnotesize{}$0.125$ } & {\footnotesize{}$0.1219809$ } & \centering{}{\footnotesize{} $2$ d  $15$ h}\tabularnewline
\midrule
\centering{}{\footnotesize{}$25$} & {\footnotesize{}$118653$ } & {\footnotesize{}$118433$ } & {\footnotesize{}$0.0506221$ } & {\footnotesize{}$0.051948$ } & {\footnotesize{}$0.0506457$ } & \centering{}{\footnotesize{} $4$ d  $18$ h}\tabularnewline
\bottomrule
\end{tabular}\caption{{ Comparison between the performances of the optimal method obtained by solving \eqref{eq:BnB-PEP-final-smth-ncvx-grad-reduction} with the BnB-PEP Algorithm, \eqref{eq:GD}, and \eqref{eq:AKZ}. 
The BnB-PEP Algorithm was executed on a standard laptop for $N=1,2,\ldots,10$, and on \texttt{MIT Supercloud} for $N=25$. The performance difference between the optimal method and \eqref{eq:AKZ}, while small, is genuine, as the difference is greater than precision of the solver, set to $10^{-10}$. 
\label{tab:BnB-PEP-result-smth-ncvx-grad-reduction}}}
\end{table}

\begin{table}
\centering{}{\footnotesize{}}%
\begin{tabular}{cc}
\toprule 
{\footnotesize{}$N$} & {\footnotesize{}$h^{\star}$}\tabularnewline
\midrule
\midrule 
{\footnotesize{}$1$} & {\footnotesize{}$\left[1.154700\right]$ }\tabularnewline
\midrule
{\footnotesize{}$2$} & {\footnotesize{}$\left[\begin{array}{cc}
1.157583\\
0.023142 & 1.146857
\end{array}\right]$ }\tabularnewline
\midrule
{\footnotesize{}$3$} & {\footnotesize{}$\left[\begin{array}{ccc}
1.15762\\
0.023577 & 1.149576\\
0.003462 & 0.021945 & 1.146719
\end{array}\right]$ }\tabularnewline
\midrule
{\footnotesize{}$4$} & {\footnotesize{}$\left[\begin{array}{cccc}
1.15762\\
0.023584 & 1.149611\\
0.003535 & 0.022356 & 1.149436\\
0.000549 & 0.003276 & 0.021922 & 1.146717
\end{array}\right]$ }\tabularnewline
\midrule
{\footnotesize{}$5$} & {\footnotesize{}$\left[\begin{array}{ccccc}
1.15762\\
0.023586 & 1.149611\\
0.003546 & 0.02236 & 1.149469\\
0.00061 & 0.003334 & 0.022329 & 1.149433\\
0.000149 & 0.000527 & 0.003263 & 0.02192 & 1.146717
\end{array}\right]$ }\tabularnewline

\midrule
{\footnotesize{}$10$} & {\footnotesize{} See Supplementary Information or \texttt{Github} repository}\tabularnewline
\midrule
{\footnotesize{}$25$} & {\footnotesize{} See Supplementary Information or \texttt{Github} repository}\tabularnewline
\bottomrule
\end{tabular}\caption{Globally optimal stepsizes obtained by solving \eqref{eq:BnB-PEP-final-smth-ncvx-grad-reduction} with the BnB-PEP Algorithm. \label{tab:BnB-PEP-result-smth-ncvx-grad-reduction2}} \end{table}

\begin{figure}[htp]%
    \centering
    \subfloat[\centering ]{{\input{Sections/Images/NCVX_GRAD_REDUCTION/smth_ncvx_gap_deKlerk_BnB_PEP.tex} }}%
    \qquad
    \subfloat[\centering ]{{ \input{Sections/Images/NCVX_GRAD_REDUCTION/timing_ncvx_gradient_reduction.tex}  }}%
\caption{
Numerical results associated with the stepsizes by solving \eqref{eq:BnB-PEP-final-smth-ncvx-grad-reduction} with the BnB-PEP Algorithm.
(Left) Improvement in the worst-case guarantee of
$\min_{i\in[0:N]}\|\nabla f(x_{i})\|^{2}$ vs.\ iteration count $N$.
(Right) Runtimes of the BnB-PEP Algorithm to compute locally optimal solutions (including Stages 1 and 2 but excluding Stage 3).
}\label{fig:Ncvx-grad-red-performance-Shuvo-vs-DeKlerk}
\end{figure}
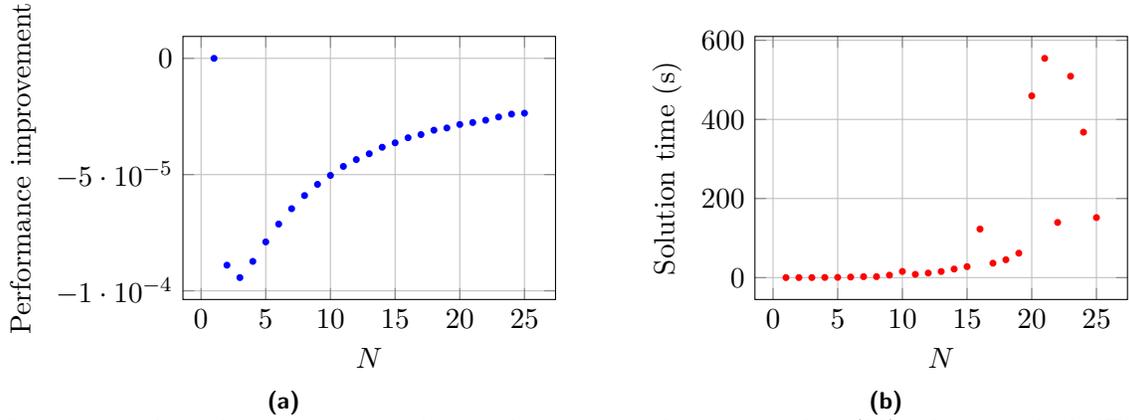

\paragraph{Momentum form of optimal FSFOM.}
An interesting observation is that the optimal FSFOM computed by the BnB-PEP Algorithm
can be equivalently written in the ``momentum form'':
\begin{equation}
\begin{alignedat}{1}y_{i+1} & =x_{i}-\frac{1}{L}\nabla f(x_{i})\\
x_{i+1} & =y_{i+1}+\zeta_{i+1}(y_{i+1}-y_{i})+\eta_{i+1}(y_{i+1}-x_{i})
\end{alignedat}
\label{eq:momentum_form}
\end{equation}
for $i\in[0:N-1]$, with coefficients $\{\zeta_i\}_{i\in[1:N]}$ and $\{\eta_i\}_{i\in[1:N]}$.
Table~\ref{tab:Momentum_form_stepsizes} shows the equivalent optimal coefficients.

The class of FSFOMs in momentum form is a strict subset of the class of FSFOMs \cite[$\mathsection$4.2]{taylor2017smooth}. 
Nesterov's fast gradient method is expressed in the momentum form \eqref{eq:momentum_form} with $\eta_i=0$ for all $i$.
Many other accelerated gradient methods such as OGM \cite{drori2014performance,kim2016optimized}, OGM-G \cite{kim2021optimizing}, Simple-OGM and SC-OGM \cite{parkparkryu2021}, FISTA \cite{beck2009fast}, FISTA-G \cite{lee2021}, EAG \cite{yoon2021accelerated,Tran-Dinh2022}, TMM \cite{van2017fastest,lee2021}, ITEM \cite{taylor2021optimal}, ORC-F$_{\flat}$, OBL-F$_{\flat}$, and OBL-G$_{\flat}$ \cite{park2021}, and M-OGM-G \cite{ZhouTianSoCheng2021_practical,lee2021} can all be expressed in the momentum form \eqref{eq:momentum_form}.

\begin{table}
\centering{}{\footnotesize{}}%
\begin{tabular}{ccc}
\toprule 
{\footnotesize{}$N$} & {\footnotesize{}$\zeta^{\star}$} & {\footnotesize{}$\eta^{\star}$}\tabularnewline
\midrule
\midrule 
{\footnotesize{}$1$} & {\footnotesize{}$[0.0]$} & {\footnotesize{}$[0.1547]$}\tabularnewline
\midrule
{\footnotesize{}$2$} & {\footnotesize{}$[0.0,0.146858]$} & {\footnotesize{}$[0.157583,0]$}\tabularnewline
\midrule
{\footnotesize{}$3$} & {\footnotesize{}$[0.0,0.149583,0.146717]$} & {\footnotesize{}$[0.157619,0,0]$}\tabularnewline
\midrule
{\footnotesize{}$4$} & {\footnotesize{}$[0.0,0.149626,0.149426,0.146702]$} & {\footnotesize{}$[0.15762,0,0,0]$}\tabularnewline
\midrule
{\footnotesize{}$5$} & {\footnotesize{}$[0.0,0.149622,0.149464,0.149417,0.146707]$} & {\footnotesize{}$[0.15762,0,0,0,0]$}\tabularnewline
\bottomrule
\end{tabular}\caption{Momentum form coefficients $\{\zeta_i^\star\}_{i\in[1:N]}$ and $\{\eta_i^\star\}_{i\in[1:N]}$ for \eqref{eq:momentum_form} of the optimal method obtained by solving \eqref{eq:BnB-PEP-final-smth-ncvx-grad-reduction} with the BnB-PEP Algorithm.
\label{tab:Momentum_form_stepsizes}
}
\end{table}

Furthermore, we observe that the optimal FSFOMs for $N=6,\dots,25$ also admit momentum forms. { 
The list of the stepsizes in the momentum form coefficients 
$\{\zeta_i^\star \}_{i\in[1:N]}$ and $\{\eta_i^\star\}_{i\in[1:N]}$
for $N=6,\ldots,25$ are provided as Supplementary Information and also as a data file at:
\begin{center}
\href{https://github.com/Shuvomoy/BnB-PEP-code/blob/main/Misc/zetaeta.jl}{\texttt{https://github.com/Shuvomoy/BnB-PEP-code/blob/main/Misc/zetaeta.jl}}
\end{center} }

%% file: Sections/Images/NCVX_GRAD_REDUCTION/smth_ncvx_gap_deKlerk_BnB_PEP.tex
\begin{tikzpicture}
\begin{axis}[legend style= {at={(0.8,1)}, anchor=north, font=\tiny}, scale only axis, height=3.5cm, width=0.3*\textwidth, xmajorgrids, ymajorgrids, xlabel={$N$}, ylabel={Performance improvement}, yticklabel={$\pgfmathparse{\tick}%
      \pgfkeys{/pgf/number format/.cd,precision=4}\pgfmathprintnumber\pgfmathresult$%
    },scaled y ticks=false]
    \addplot[only marks, mark size={0.5pt}, style={{ultra thick}}, color={blue}]
        coordinates {
            (1,0)
            (2,-8.89588094979965e-5)
            (3,-9.429524646564769e-5)
            (4,-8.732862111487938e-5)
            (5,-7.895243331038082e-5)
            (6,-7.1278683097864e-5)
            (7,-6.465731878069203e-5)
            (8,-5.902088357867541e-5)
            (9,-5.4196648223370225e-5)
            (10,-5.034946299807552e-5)
            (11,-4.64833332882858e-5)
            (12,-4.353571245814536e-5)
            (13,-4.1028940485882126e-5)
            (14,-3.823073601853799e-5)
            (15,-3.6327501492072756e-5)
            (16,-3.4149589816623305e-5)
            (17,-3.278582669641317e-5)
            (18,-3.089963823989239e-5)
            (19,-2.994310645737519e-5)
            (20,-2.843573613922623e-5)
            (21,-2.760938706834748e-5)
            (22,-2.6609798866948642e-5)
            (23,-2.5219031202415587e-5)
            (24,-2.397546110148585e-5)
            (25,-2.358872425551467e-5)
        }
        ;
\end{axis}
\end{tikzpicture}

%% file: Sections/Images/NCVX_GRAD_REDUCTION/timing_ncvx_gradient_reduction.tex
\begin{tikzpicture}
\begin{axis}[legend style= {at={(0.14,1)},anchor=north, font=\tiny}, scale only axis, height=3.5cm, width=0.3*\textwidth,xmajorgrids, ymajorgrids, xlabel={$N$}, ylabel={Solution time (s)}]
    \addplot[only marks, mark size={0.5pt}, style={{ultra thick}}, color={red}, mark={*}]
                coordinates {
            (1,0.070161023)
            (2,0.101102881)
            (3,0.160275843)
            (4,0.399274617)
            (5,0.659336463)
            (6,1.431591811)
            (7,2.419940641)
            (8,2.42578508)
            (9,6.265998392)
            (10,15.62016793)
            (11,8.259732616)
            (12,11.486353509)
            (13,15.461450463)
            (14,21.607581055)
            (15,27.759875428)
            (16,122.775105061)
            (17,36.444603632)
            (18,45.185740894)
            (19,61.803959703)
            (20,459.53442979)
            (21,554.363971515)
            (22,139.440754525)
            (23,509.120910092)
            (24,367.676545806)
            (25,151.821902492)
        }
        ;
\end{axis}
\end{tikzpicture}

%% file: Sections/optimal_algorithm_with_respect_to_a_potential_for_weakly_convex_gradient_reduction.tex
\subsection{Efficient first-order method with respect to a potential function in
weakly convex setup \label{subsec:pot-bnb-pep-wcvx-Moreau}}

Consider the problem of constructing an FSFOM that efficiently reduces the subgradient magnitude of $\rho$-weakly convex functions with $L$-bounded subgradients.
Formally, we choose the function class
\[
\mathcal{F}=\{f\,|\,f\in\rhoLWcvx,\,f\text{ has a global minimizer }x_{\star}\}.
\]
Consider FSFOMs of the form
\begin{equation}
x_{i+1}=x_{i}-\frac{h}{\rho}f^{\prime}(x_{i})
\label{eq:gd-wcvx}
\end{equation}
for $i\in[0:N]$, where $f^{\prime}(x_{i})\in \partial f(x_{i})$ and $h\in \rl$ is the stepsize to be determined.
Let $\widetilde{L}=L/\rho$. Since $f\in\rhoLWcvx\Leftrightarrow f/\rho\in\mathcal{W}_{1,\widetilde{L}}$,
 consider $f\in\mathcal{W}_{1,\widetilde{L}}$ and set $\rho=1$ without loss of generality.

In this section, we show how to construct an efficient FSFOM by obtaining potential function analyses of FSFOMs and minimizing the guarantee. Unlike in previous sections, our goal here is not to construct an optimal FSFOM but rather is to construct an efficient FSFOM with an analytically tractable potential function analysis.

Using optimization to find potential function analyses of FSFOMs has been studied by Lessard, Recht, and Packard \cite{lessard2016analysis} and Taylor, Van Scoy, and Lessard \cite{taylor2018lyapunov} and the philosophy goes further back to the classical Lyapunov stability problem of control theory \cite{lyapunov1892}. Our approach, in particular, closely follows and generalizes the work of Taylor and Bach \cite[Appendix C]{taylor2019stochastic}, which finds potential function analyses of FSFOMs and optimize a certain span-search based relaxation of the FSFOMs. The relaxation retains convexity of the optimization, but it is restricted to the convex minimization setup with performance measure $f(x_{N})-f(x_{\star})$. Our proposed methodology removes this restriction; we can construct efficient FSFOMs in the convex and nonconvex setup with various performance measures. The concrete instance of this section illustrates our methodology and improves upon the prior state-of-the-art rate of Davis and Drusvyatskiy in \cite{davis2018stochasticArxivFirst, davis2018stochasticArxivSecond,davis2019stochastic}.

\subsubsection{Measuring stationarity via Moreau envelope \label{subsec:Measuring-stationarity-via-mr-env-ncvx}}
In the nonsmooth nonconvex setup, performance measures commonly used in the convex setup, such as $f(x_{N})-f(x_{\star})$ or $\textrm{dist}(0;\partial f(x_{N}))$, may not go to zero as $N\to\infty$ \cite[page 3, paragraph 2]{davis2019stochastic}. Therefore, we define a notion of approximate stationarity via the Moreau envelope.

Consider $f\in\mathcal{W}_{1,\widetilde{L}}$ and let $\hat{\rho}>1$.
The proximal operator and Moreau envelope of $f$ are respectively defined as 
\begin{align*}
\prox_{(1/\hat{\rho})f}(x) & =\argmin_{y\in\rl^n}\left\{ f(y)+\frac{\hat{\rho}}{2}\|y-x\|^{2}\right\},
\qquad
f_{(1/\hat{\rho})}(x)=\min_{y\in\rl^n}\left\{ f(y)+\frac{\hat{\rho}}{2}\|y-x\|^{2}\right\}.
\end{align*}
The Moreau envelope $f_{(1/\hat{\rho})}$ is global underestimator of $f$ that is continuously differentiable: with $y=\prox_{(1/\hat{\rho})f}(x)$, we have
\begin{equation*}
x-y=
\frac{1}{\hat{\rho}}\nabla f_{(1/\hat{\rho})}(x)
\in 
\partial f(y)
\end{equation*}
\cite[(2.13), (2.17)]{rockafellar2020characterizing}.
If $x_{\star}$ is a global minimizer of $f$, then
$f_{(1/\hat{\rho})}(x_{\star})=f(x_{\star})$ and $\|\nabla f_{(1/\hat{\rho})}(x_{\star})\|=0$.
The gradient of the Moreau envelope serves as a measure of suboptimality
since, with $y=\prox_{(1/\hat{\rho})f}(x)$, we have
\begin{align*}
&\|y-x\|=\frac{1}{\hat{\rho}}\|\nabla f_{(1/\hat{\rho})}(x)\|,
\nonumber\\
&f(y)\leq f(x),
\\
&\textrm{dist}(0,\partial f(y))\leq\|\nabla f_{(1/\hat{\rho})}(x)\|,
\nonumber
\end{align*}
for any $x\in\rl^{d}$ \cite[page 4]{davis2019stochastic}.
In other words, if $\|\nabla f_{(1/\hat{\rho})}(x)\|$ is small,
then $x$ is near some point $y$ that is nearly stationary for $f$.

We set $\hat{\rho}=2$, the simplest choice.
For a given sequence of iterates $\{x_{i}\}_{i\in[0:N]\cup\{\star\}}$, define
\begin{align}
\begin{alignedat}{1}
y_{i} & =\prox_{(1/2)f}(x_{i})\\
f^{\prime}(y_{i})&=2(x_{i}-y_{i}).
\end{alignedat}
\label{eq:secondary_iterate}
\end{align}
Choose the performance measure
\begin{equation*}
\mathcal{E}=\frac{1}{N+1}\sum_{i=0}^{N}\|\nabla f_{(1/2)}(x_{i})\|^{2}=\frac{1}{N+1}\sum_{i=0}^{N}\|f^{\prime}(y_{i})\|^{2},
\end{equation*}
which was also used in \cite{davis2019stochastic}.
Note that 
\[
\min_{i\in [0:N]}\|\nabla f_{(1/2)}(x_{i})\|^{2}\le \frac{1}{N+1}\sum_{i=0}^{N}\|\nabla f_{(1/2)}\left(x_{i}\right)\|^{2},
\]
so any guarantee on $\mathcal{E}$ translates to a guarantee on $\min_{i\in [0:N]}\|\nabla f_{(1/2)}(x_{i})\|^{2}$.

Finally, we provide a few points of clarification. The parameter $\hat{\rho}$ is not used in the method \eqref{eq:gd-wcvx} and only appears in the analysis of the algorithm. Since $\hat{\rho}$ is strictly larger than $1$, the weak convexity parameter, $y=\prox_{(1/\hat{\rho})f}(x)$ is defined as a minimizer of a strongly convex function and therefore uniquely exists. While $f^{\prime}(x_{i})\in \partial f(x_{i})$ is chosen arbitrarily in the method \eqref{eq:gd-wcvx} (the $i$\nobreakdash-th iteration of the method may use \emph{any} subgradient at $x_{i}$) the choice of $f^{\prime}(y_{i})\in \partial f(y_{i})$ (used in the analysis) is specified by \eqref{eq:secondary_iterate} and therefore is not arbitrary.

\subsubsection{Potential function analysis via BnB-PEP \label{subsec:-potential-function-bnb-pep}}

Consider the potential function
\begin{equation*}
\psi_{k}= b_{k}\left(f_{(1/2)}(x_{k})-f_{(1/2)}(x_{\star})\right)=b_{k}\left(f(y_{k})-f(x_{\star})+\|x_{k}-y_{k}\|^{2}\right),
\qquad k\in [0:N+1],
\end{equation*}
where $x_{\star}$ is a global minimizer of $f$, $\{b_k\}_{k\in [0:N+1]}$ are parameters to be determined, and $\{y_k\}_{k\in [0:N+1]}$ are as defined in \eqref{eq:secondary_iterate}.
Choose the initial condition $\mathcal{C}$ as
\begin{align*}
f_{(1/2)}(x_{0})-f_{(1/2)}(x_{\star}) & =f(y_{0})-f(x_{\star})+\|x_{0}-y_{0}\|^{2}\leq R^{2}.
\end{align*}
Again, let $x_{\star}=0$ and $f(x_{\star})=0$ without loss of generality.
If we show
\begin{equation}
\|f^{\prime}(y_{k})\|^{2}+\psi_{k+1}-\psi_{k}\leq c_{k}\|f^{\prime}(x_{k})\|^{2}
\label{eq:sufficient_condition_psi}
\end{equation}
for $k\in [0:N]$,
where $\{b_k\}_{k\in [0:N+1]}$ and $\{c_k\}_{k\in [0:N]}$ are nonnegative parameters to be determined,
then a telescoping sum provides the rate
\begin{align*}
\frac{1}{N+1}\sum_{i=0}^{N}\|f^{\prime}(y_{i})\|^{2}\leq & \frac{1}{N+1}\left(\widetilde{L}^{2}\sum_{i=0}^{N}c_{i}+\psi_{0}-\psi_{N+1}\right)\\
 & \leq\frac{1}{N+1}\left(\widetilde{L}^{2}\sum_{i=0}^{N}c_{i}+b_{0}R^{2}\right).
\end{align*}

In a potential function analysis, we effectively choose to be oblivious to how $x_k$ was generated and to the method's prior evaluations of $f$; we establish the potential function inequality \eqref{eq:sufficient_condition_psi} one iteration at a time. 
Due to this restriction, our efficient FSFOM is not expected to be optimal, but it is expected to have a simpler analytically tractable analysis.

\paragraph{Potential function analysis.}
Let
\begin{gather*}
\mathcal{V}_k(h)=\left\{(b_{k+1},b_k,c_k)\,|\,\|f^{\prime}(y_{k})\|^{2}+\psi_{k+1}\le \psi_{k}-c_{k}\|f^{\prime}(x_{k})\|^{2},\,\forall\, x_k\in\rl^d,\,f\in\mathcal{W}_{1,\widetilde{L}}\right\},
\end{gather*}
where
$x_{k+1}=x_{k}-hf^{\prime}(x_{k})$, $ y_{k}=x_{k}-\frac{1}{2}f^{\prime}(y_{k})$, and $y_{k+1}=x_{k+1}-\frac{1}{2}f^{\prime}(y_{k+1})$,
be the set of $(b_{k+_1},b_k,c_k)$ such that \eqref{eq:sufficient_condition_psi} holds for all $x_k\in\rl^d$ and $f\in\mathcal{W}_{1,\widetilde{L}}$.
Then, one could consider optimizing the FSFOM by solving
\begin{equation}
\left(\begin{array}{ll}
\underset{}{\mbox{minimize}} & \frac{1}{N+1}\left(\widetilde{L}^2\sum_{i=0}^{N}c_{i}+b_{0}R^{2}\right)\\
\mbox{subject to} & (b_{k+1},b_k,c_k)\in \mathcal{V}_k(h),\quad k\in[0:N+1],
\end{array}\right)\label{eq:pot-pep-infnt-dm}
\end{equation}
where $\{b_{i}\}_{i\in[0:N+1]}$, $\{c_{i}\}_{i\in[0:N]}$, and $h\in \rl$ are the decision variables.

However, checking $(b_{k+_1},b_k,c_k)\in \mathcal{V}_k(h)$ is difficult and \eqref{eq:pot-pep-infnt-dm} is difficult to solve, because $\mathcal{W}_{1,\widetilde{L}}$ is a function class without a known interpolation result.
In the following, we find $\tilde{\mathcal{V}}_k(h)\subseteq \mathcal{V}_k(h)$ such that membership with respect to $\tilde{\mathcal{V}}_k(h)$ is easy to check.
Then we optimize the FSFOM by solving
\begin{equation}
\left(\begin{array}{ll}
\underset{}{\mbox{minimize}} & (1/(N+1))\left(\widetilde{L}^2\sum_{i=0}^{N}c_{i}+b_{0}R^{2}\right)\\
\mbox{subject to} & (b_{k+1},b_k,c_k)\in \tilde{\mathcal{V}}_k(h),\quad k\in[0:N+1]
\end{array}\right)
\label{eq:pot-pep-infnt-dm2}
\end{equation}
where $\{b_{i}\}_{i\in[0:N+1]}$, $\{c_{i}\}_{i\in[0:N]}$, and $h\in \rl$ are the decision variables.
Note, \eqref{eq:pot-pep-infnt-dm} is upper bounded by \eqref{eq:pot-pep-infnt-dm2}, since  $ \mathcal{V}_k(h)\supseteq \tilde{\mathcal{V}}_k(h)$.

{  
In the following, we show that each constraint $(b_{k+1},b_k,c_k)\in \mathcal{V}_k(h)$ in \eqref{eq:pot-pep-infnt-dm2} $ k\in[0:N+1]$ can be formulated into a QCQP feasibility problem. We then apply the feasibility problem formulations $ k\in[0:N+1]$ to obtain the BnB-PEP-QCQP \eqref{eq:potential-bnb-pep}.}

\paragraph{Sufficient SDP for potential function inequality.}
Note, $(b_{k+1},b_k,c_k)\in \mathcal{V}_k(h)$ 
if and only if the optimal value of the following problem is less than or equal to $0$:
\begin{equation}
\left(\begin{array}{ll}
\underset{}{\mbox{maximize}} 
 &\|f^{\prime}(y_{k})\|^{2}+b_{k+1}\left(f(y_{k+1})-f(x_{\star})+\|x_{k+1}-y_{k+1}\|^{2}\right)\\
 & - b_{k}\left(f(y_{k})-f(x_{\star})+\|x_{k}-y_{k}\|^{2}\right)-c_{k}\|f^{\prime}(x_{k})\|^{2}
\\
\mbox{subject to} & x_{k+1}=x_{k}-hf^{\prime}(x_{k})\\
 & y_{k}=x_{k}-\frac{1}{2}f^{\prime}(y_{k})\\
 & y_{k+1}=x_{k+1}-\frac{1}{2}f^{\prime}(y_{k+1}),\\
 & f^{\prime}(x_{\star})=0,\;x_{\star}=0,\;f(x_{\star})=0,\\
 & f(w)\geq f(x_{\star}),\quad w\in\{x_{k},x_{k+1},y_{k},y_{k+1}\},\\
 & f\in\mathcal{W}_{1,\widetilde{L}},
\end{array}\right)\label{eq:pot-pep-infnt-dem-feasiblity}
\end{equation}
where $f\in\mathcal{W}_{1,\widetilde{L}}$ and $x_{k},x_{k+1},y_{k},y_{k+1}\in \rl^d$ are the decision variables.

Define $ \hat{\mathcal{V}}_k(h)\subseteq \mathcal{V}_k(h)$ such that $(b_{k+1},b_k,c_k)\in \hat{\mathcal{V}}_k(h)$ if and only if the optimal value of the following problem is less than or equal to $0$:
\begin{equation}
\left(\begin{array}{ll}
\underset{}{\mbox{maximize}} 
 &\|f^{\prime}(y_{k})\|^{2}+b_{k+1}\left(f(y_{k+1})-f(x_{\star})+\|x_{k+1}-y_{k+1}\|^{2}\right)\\
 & - b_{k}\left(f(y_{k})-f(x_{\star})+\|x_{k}-y_{k}\|^{2}\right)-c_{k}\|f^{\prime}(x_{k})\|^{2}\\
\mbox{subject to} & x_{k+1}=x_{k}-hf^{\prime}(x_{k})\\
 & y_{k}=x_{k}-\frac{1}{2}f^{\prime}(y_{k})\\
 & y_{k+1}=x_{k+1}-\frac{1}{2}f^{\prime}(y_{k+1}),\\
 & f^{\prime}(x_{\star})=0,\;x_{\star}=0,\;f(x_{\star})=0,\\
 & f(w)\geq f(x_{\star}),\quad w\in\{x_{k},x_{k+1},y_{k},y_{k+1}\},\\
 & f(w')\geq f(w)+\langle f'(w)\mid w'-w\rangle-\frac{1}{2}\|w'-w\|^{2},\\
 &\quad\quad\quad\quad\quad\quad\quad\quad w,w'\in\{x_{k},x_{k+1},y_{k},y_{k+1}\},\\
 & \|f'(w)\|^{2}\leq \widetilde{L}^{2},\quad w\in\{x_{k},x_{k+1},y_{k},y_{k+1}\},
\end{array}\right)\label{eq:pot-pep-infnt-dem-feasiblity2}
\end{equation}
where $f$, $x_{k}$, $x_{k+1}$, $y_{k}$, and $y_{k+1}$ are the decision variables.
Since the constraints of \eqref{eq:pot-pep-infnt-dem-feasiblity} imply the constraints of  \eqref{eq:pot-pep-infnt-dem-feasiblity2}, 
the optimal value of
 \eqref{eq:pot-pep-infnt-dem-feasiblity} is upper bounded by   \eqref{eq:pot-pep-infnt-dem-feasiblity2} and $ \mathcal{V}_k(h)\supseteq \hat{\mathcal{V}}_k(h)$.
(Since $\mathcal{W}_{1,\widetilde{L}}$ is a function class with no known interpolation result, two optimal values and the sets are not necessarily equal.)
For notational convenience, define
\begin{align*}
 & (x_{\star},f^{\prime}(x_{\star}),f(x_{\star}))=(w_{\star},g_{\star},f_{\star}),\\
 & (x_{k},f^{\prime}(x_{k}),f(x_{k}))=(w_{0},g_{0},f_{0}),\\
 & (x_{k+1},f^{\prime}(x_{k+1}),f(x_{k+1}))=(w_{1},g_{1},f_{1}),\\
 & (y_{k},f^{\prime}(y_{k}),f(y_{k}))=(w_{2},g_{2},f_{2}),\\
 & (y_{k+1},f^{\prime}(y_{k+1}),f(y_{k+1}))=(w_{3},g_{3},f_{3}).
\end{align*}
Then we can express \eqref{eq:pot-pep-infnt-dem-feasiblity2} equivalently as
\begin{equation}
\left(\begin{array}{ll}
\underset{}{\mbox{maximize}} & \|g_{2}\|^{2}+b_{k+1}\left(f_{3}-f_{\star}+\|w_{1}-w_{3}\|^{2}\right)\\
&-(b_{k}\left(f_{2}-f_{\star}+\|w_{0}-w_{2}\|^{2}\right)-c_{k}\|g_{0}\|^{2}\\
\mbox{subject to} & w_{1}=w_{0}-hg_{0}\\
 & w_{2}=w_{0}-\frac{1}{2}g_{2},\\
 & w_{3}=w_{1}-\frac{1}{2}g_{3},\\
 & (w_{\star},g_{\star},f_{\star})=(0,0,0),\\
 & f(w_{i})\geq f(x_{\star}),\quad i\in[0:3],\\
 & f_{i}\geq f_{j}+\langle g_{j}\mid w_{i}-w_{j}\rangle-\frac{1}{2}\|w_{i}-w_{j}\|^{2},\quad i,j\in[0:3]\cup\{\star\}:i\neq j,\\
 & \|g_{i}\|^{2}\leq \widetilde{L}^{2},\quad\quad i\in[0:3]\cup\{\star\},
\end{array}\right)\label{eq:pot-pep-infnt-dem-feasiblity-wcvx}
\end{equation}
where $\{w_{i},g_{i},f_{i}\}_{i\in[0:3]\cup\{\star\}}$ are the decision
variables.

Next, we use the Grammian formulation to formulate \eqref{eq:pot-pep-infnt-dem-feasiblity-wcvx} as an SDP.
For $k\in[0:N]$, let
\begin{equation}
\begin{alignedat}{1} & H^{[k]}=[w_{0}\mid g_{0}\mid g_{1}\mid g_{2}\mid g_{3}]\in\rl^{d\times5},\\
 & G^{[k]}=H^{[k]\top}H^{[k]}\in\mathbb{S}_{+}^{5},\\
 & F^{[k]}=[f_{0}\mid f_{1}\mid f_{2}\mid f_{3}]\in\rl^{1\times4}.
\end{alignedat}
\end{equation}
Note that $\rank G^{[k]}\leq d$.
Define the following notation for selecting columns and elements of $H^{[k]}$ and $F^{[k]}$:
\begin{equation*}
\begin{alignedat}{1} & \mathbf{g}_{\star}=0\in\rl^{5},\mathbf{g}_{i}=e_{i+2}\in\rl^{5},\;i\in[0:3],\\
 & \mathbf{f}_{\star}=0\in\rl^{4},\mathbf{f}_{i}=e_{i+1}\in\rl^{4},\;i\in[0:3],\\
 & \mathbf{w}_{\star}=0\in\rl^{5},\\
 & \mathbf{w}_{0}=e_{1}\in\rl^{5},\\
 & \mathbf{\mathbf{w}}_{1}=\mathbf{\mathbf{w}}_{0}-h\mathbf{\mathbf{g}}_{0}\in\rl^{5},\\
 & \mathbf{\mathbf{w}}_{2}=\mathbf{\mathbf{w}}_{0}-\frac{1}{2}\mathbf{\mathbf{g}}_{2}\in\rl^{5},\\
 & \mathbf{\mathbf{w}}_{3}=\mathbf{\mathbf{w}}_{1}-\frac{1}{2}\mathbf{\mathbf{g}}_{3}\in\rl^{5}.
\end{alignedat}
\end{equation*}

Furthermore, define
\begin{equation*}
\begin{alignedat}{1} & A_{i,j}(h)=\mathbf{g}_{j}\odot(\mathbf{w}_{i}-\mathbf{w}_{j}),\\
 & B_{i,j}(h)=(\mathbf{w}_{i}-\mathbf{w}_{j})\odot(\mathbf{w}_{i}-\mathbf{w}_{j}),\\
 & C_{i,j}=(\mathbf{g}_{i}-\mathbf{g}_{j})\odot(\mathbf{g}_{i}-\mathbf{g}_{j}),\\
 & a_{i,j}=\mathbf{f}_{j}-\mathbf{f}_{i},
\end{alignedat}
\end{equation*}
for $i,j\in[0:3]\cup\{\star\}$.
Note that $A_{i,j}(h)$ and $B_{i,j}(h)$ are affine and quadratic as functions of $h$.
This notation defined so that
\begin{equation*}
\begin{alignedat}{1} & w_{i}=H^{[k]}\mathbf{w}_{i},\;g_{i}=H^{[k]}\mathbf{g}_{i},\;f_{i}=F^{[k]}\mathbf{f}_{i},\\
 & \left\langle g_{j}\mid w_{i}-w_{j}\right\rangle =\tr G^{[k]}A_{i,j}(h),\\
 & \|w_{i}-w_{j}\|^{2}=\tr G^{[k]}B_{i,j}(h),\textrm{ and}\\
 & \|g_{i}-g_{j}\|^{2}=\tr G^{[k,]}C_{i,j}.
\end{alignedat}
\end{equation*}
for $i,j\in[0:3]\cup\{\star\}$.
Finally, define
\begin{equation*}
\begin{alignedat}{1}Q^{[k]} & = Q^{[k]}(h,b_{k+1},b_{k},c_{k})= C_{2,\star}+b_{k+1}B_{1,3}(h)-b_{k}B_{0,2}(h)-c_{k}C_{0,\star},\\
q^{[k]} & = q^{[k]}(b_{k+1},b_{k})= b_{k+1}a_{\star,3}-b_{k}a_{\star,2}
\end{alignedat}
\end{equation*}
for $k\in [0:N]$.
Assume the large-scale assumption $d\geq5$.
Using the new notation, equivalently formulate \eqref{eq:pot-pep-infnt-dem-feasiblity-wcvx} as
\begin{equation}
\left(\begin{array}{ll}
\underset{}{\mbox{maximize}} & \tr G^{[k]}Q^{[k]}+F^{[k]}q^{[k]}\\
\mbox{subject to} & F^{[k]}a_{i,\star}\leq0,\quad i\in[0:3], \quad \rhd \textsf{\;dual var.\;}\tau_{i}^{[k]}\geq0 \\
 & F^{[k]}a_{i,j}+\tr G^{[k]}\left(A_{i,j}(h)-\frac{1}{2}B_{i,j}(h)\right)\leq0,\quad i,j\in[0:3]\cup\{\star\}:i\neq j, \quad \rhd \textsf{\;dual var.\;}\lambda_{i,j}^{[k]}\geq0 \\
 & \tr G^{[k]}C_{i,\star}\leq \widetilde{L}^{2}\quad i\in[0:3]\cup\{\star\},\quad \rhd \textsf{\;dual var.\;}\eta_{i}^{[k]}\geq0\\
 & -G^{[k]}\preceq0,\quad \rhd \textsf{\;dual var.\;} Z\succeq0
\end{array}\right)\label{eq:pot-pep-primal-sdp-wcvx}
\end{equation}
 where $G^{[k]}\in \mathbb{S}_+^5$ and $F^{[k]}\in\rl^{1\times 4}$ are the decision variables.

Next, we dualize.
Define $ \tilde{\mathcal{V}}_k(h)\subseteq  \mathcal{V}_k(h)$ such that $(b_{k+1},b_k,c_k)\in \tilde{\mathcal{V}}_k(h)$ if and only if the optimal value of the following problem is less than or equal to $0$:
\begin{equation}
\left(\begin{array}{ll}
\underset{}{\mbox{minimize}} & \widetilde{L}^{2}\sum_{i\in[0:3]\cup\{\star\}}\eta_{i}^{[k]}\\
\mbox{subject to} & -Q^{[k]}+\sum_{i,j\in[0:3]\cup\{\star\}:i\neq j}\lambda_{i,j}^{[k]}\left(A_{i,j}(h)-\frac{1}{2}B_{i,j}(h)\right)+\sum_{i\in[0:3]\cup\{\star\}}\eta_{i}C_{i,\star}=Z^{[k]},\\
 & -q^{[k]}+\sum_{i\in[0:3]}\tau_{i}^{[k]}a_{i,\star}+\sum_{i,j\in[0:3]\cup\{\star\}:i\neq j}\lambda_{i,j}^{[k]}a_{i,j}=0,\\
 & \lambda_{i,j}^{[k]}\geq0,\;\eta_{i}^{[k]}\geq0,\quad i,j\in[0:3]\cup\{\star\}:i\neq j,\\
 & \tau_{i}^{[k]}\geq0,\quad i\in[0:3],\\
 & Z^{[k]}\succeq0,
\end{array}\right)\label{eq:pot-pep-dual-sdp-wcvx}
\end{equation}
where $\{\eta^{[k]}\}_{i\in [0:3]\cup\{\star\}}$, $\{\lambda^{[k]}_{i,j}\}_{i,\in [0:3]\cup\{\star\}:i\ne j}$, $\{\tau^{[k]}_i\}_{i\in [0:3]}$, and $Z^{[k]}\in \mathbb{S}^5_+$ are the decision variables.
By weak duality, the optimal value of \eqref{eq:pot-pep-primal-sdp-wcvx} is upper bounded by
\eqref{eq:pot-pep-dual-sdp-wcvx} and $ \hat{\mathcal{V}}_k(h)\supseteq \tilde{\mathcal{V}}_k(h)$.
(While we expect strong duality to usually hold, we do not need to assume it.)
Observe that for the optimal value of \eqref{eq:pot-pep-dual-sdp-wcvx} to be less than equal to zero, we must have $\eta^{[k]}=0$.
Hence, \eqref{eq:pot-pep-dual-sdp-wcvx} simplifies into the feasibility problem
\begin{equation}
\left(\begin{array}{ll}
\mbox{minimize} & 0\\
\mbox{subject to} & -Q^{[k]}+\sum_{i,j\in[0:3]\cup\{\star\}:i\neq j}\lambda_{i,j}^{[k]}\left(A_{i,j}(h)-\frac{1}{2}B_{i,j}(h)\right)=Z^{[k]},\\
 & -q^{[k]}+\sum_{i\in[0:3]}\tau_{i}^{[k]}a_{i,\star}+\sum_{i,j\in[0:3]\cup\{\star\}:i\neq j}\lambda_{i,j}^{[k]}a_{i,j}=0,\\
 & \lambda_{i,j}^{[k]}\geq0,\quad i,j\in[0:3]\cup\{\star\}:i\neq j,\\
 & \tau_{i}^{[k]}\geq0,\quad i\in[0:3],\\
 & Z^{[k]}\succeq0
\end{array}\right)
\label{eq:pot-pep-dual-sdp-wcvx2}
\end{equation}
for $k\in[0:N]$.

\paragraph{Optimizing potential function analysis.}
We have shown that existence of a feasible point for \eqref{eq:pot-pep-dual-sdp-wcvx2} implies $(b_{k+1},b_k,c_k)\in \tilde{\mathcal{V}}_k(h)$, which in turn implies \eqref{eq:sufficient_condition_psi} holds.
Finally, use Lemma \ref{Lem:quadratic-characterization-psd-1} to formulate \eqref{eq:pot-pep-infnt-dm2} as the following BnB-PEP-QCQP:
\begin{equation}
\left(\begin{array}{ll}
\underset{}{\mbox{minimize}} & \widetilde{L}^{2}\left[\frac{1}{N+1}\left(\sum_{i=0}^{N}c_{i}+\frac{R^{2}}{\widetilde{L}^{2}}b_{0}\right)\right]\\
\mbox{subject to} & -Q^{[k]}+\sum_{i,j\in[0:3]\cup\{\star\}:i\neq j}\lambda_{i,j}^{[k]}\left(A_{i,j}(h)-\frac{1}{2}\Theta_{i,j}^{[k]}\right)=Z^{[k]},\quad k\in[0:N],\\
 & -q^{[k]}+\sum_{i\in[0:3]}\tau_{i}^{[k]}a_{i,\star}+\sum_{i,j\in[0:3]\cup\{\star\}:i\neq j}\lambda_{i,j}^{[k]}a_{i,j}=0,\quad k\in[0:N],\\
 & \lambda_{i,j}^{[k]}\geq0,\quad i,j\in[0:3]\cup\{\star\}:i\neq j,\quad k\in[0:N],\\
 & \tau_{i}^{[k]}\geq0,\quad i\in[0:3],\quad k\in[0:N],\\
 &P^{[k]}\text{ is lower triangular with nonnegative diagonals},\quad k\in[0:N],\\
&P^{[k]}(P^{[k]})^{\top}=Z^{[k]},\quad k\in[0:N],\\
 & \Theta_{i,j}^{[k]}=B_{i,j}(h),\quad i,j\in[0:3]\cup\{\star\}:i\neq j,\quad k\in[0:N],
\end{array}\right)\label{eq:potential-bnb-pep}
\end{equation}
where the decision variables are $\{b_{k}\}_{k\in[0:N+1]}$, $\{c_{k}\}_{k\in[0:N]}$, $\{\lambda^{[k]}\}_{k\in[0:N]}$, $\{\tau^{[k]}\}_{k\in[0:N]}$, $\{Z^{[k]}\}_{k\in[0:N]}$, $\{P^{[k]}\}_{k\in[0:N]}$, $\{\Theta^{[k]}\}_{k\in[0:N]}$, and $h$.
We call $\lambda^{[k]}$, $\tau^{[k]}$, and $Z^{[k]}$ the inner-dual variables.
Note that $\{\Theta^{[k]}\}_{k\in[0:N]}$
is introduced as a separate decision variable to formulate the cubic constraints arising from $B_{i,j}(h)$ as quadratic constraints.
A feasible solution of \eqref{eq:potential-bnb-pep} provides a performance guarantee on the FSFOM defined by the stepsize $h$.

\subsubsection{Numerical results and analytical convergence proofs \label{subsec:Numerical-results-ncvx-moreau}}
{Tables~\ref{tab:BnB-PEP-result-potential-PEP} and \ref{tab:Globally-optimal-stepsize-potential} shows the results of solving \eqref{eq:potential-bnb-pep} with $\widetilde{L}=1$, $R=0.1$ and $N=1,\ldots,5,10,25$ using the BnB-PEP Algorithm. Similar to our previous experiments, we empirically observe that the locally optimal stepsizes obtained at Stage 2, denoted by $h^{\textrm{lopt}}$, were already near-optimal. This motivates us to apply just the first two stages of the BnB-PEP Algorithm for $N=26,\ldots,100$.} In Figure~\ref{fig:PotBnBPEP-performance-Shuvo-vs-Davis}, we compare the locally optimal stepsizes $h^{\textrm{lopt}}$ with the stepsize  $h=R / (\widetilde{L} \sqrt{(N+1)})$ reported in the proof of \cite[Theorem 3.1]{davis2019stochastic}. %

In the following, we use the numerical results to obtain an analytical form of the stepsize and its convergence proof.

\begin{figure}[htp]%
    \centering
    \subfloat[\centering ]{{\input{Sections/Images/Pot_BnB_PEP/BnB_PEP_vs_Davis_performance_comparison_LogLogScale.tex} }}%
    \qquad
    \subfloat[\centering ]{{\input{Sections/Images/Pot_BnB_PEP/PotPEP_timing_performance.tex} }}%
    \caption{
{   Numerical results for the locally optimal stepsizes by solving \eqref{eq:potential-bnb-pep} the BnB-PEP Algorithm. We have verified global optimality of the locally optimal stepsizes for $N=1,2,\ldots,25$.
(Left) Objective value $\widetilde{L}^{2}[(\sum_{i=0}^{N}c_{i}+(R^{2}/\widetilde{L}^{2})b_{0})/(N+1)]$  for $h=R / (\widetilde{L} \sqrt{(N+1)})$ (as prescribed by \cite{davis2019stochastic}) and for the stepsizes computed by the BnB-PEP Algorithm   vs.\ iteration count $N$.
(Right) Runtimes of the BnB-PEP Algorithm (including Stages 1 and 2 but excluding Stage 3).}
}\label{fig:PotBnBPEP-performance-Shuvo-vs-Davis}
\end{figure}
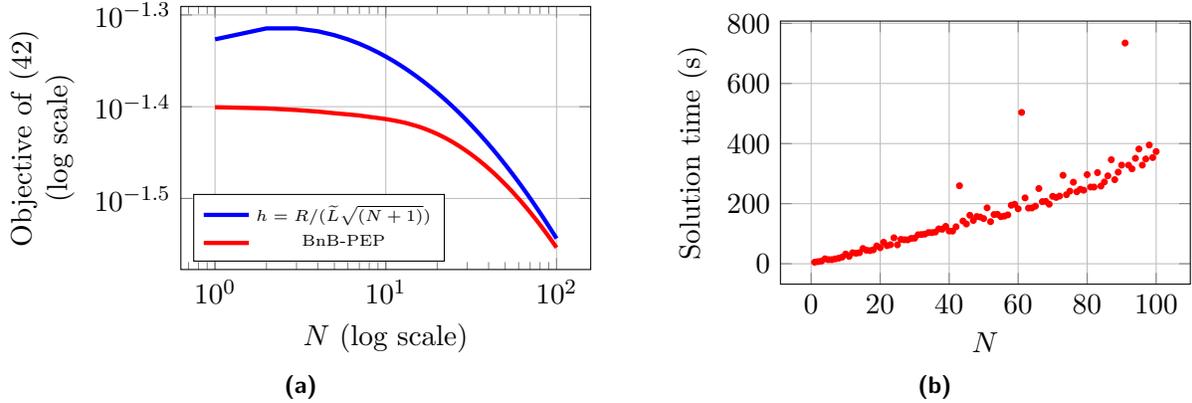

\begin{table}
\centering{}%
\begin{tabular}{>{\centering}b{2em}>{\centering}b{4em}>{\centering}b{4em}>{\centering}b{7em}>{\centering}b{7em}>{\centering}b{8em}}
\toprule 
\multirow{2}{2em}{\centering{}{\footnotesize{}$N$}} & \multirow{2}{4em}{\centering{}{\footnotesize{}\# variables}} & \multirow{2}{4em}{\centering{}{\footnotesize{}\# constraints}} & \multicolumn{2}{c}{{\footnotesize{}Worst-case $\widetilde{L}^{2}[(\sum_{i=0}^{N}c_{i}+(R^{2}/\widetilde{L}^{2})b_{0})/(N+1)]$}} & \multirow{2}{8em}{\centering{}{\footnotesize{}Runtime of the BnB-PEP Algorithm}}\tabularnewline
\cmidrule{4-5} \cmidrule{5-5} 
 &  &  & \centering{}{\footnotesize{}Optimal} & {\footnotesize{}$h=\frac{R}{(\widetilde{L} \sqrt{(N+1)})}$ } & \tabularnewline
\midrule
\midrule 
\centering{}{\footnotesize{}$1$} & {\footnotesize{}$735$ } & {\footnotesize{}$780$ } & {\footnotesize{}0.0398 } & {\footnotesize{}0.04717 } & {\footnotesize{}$0.01$ s}\tabularnewline
\midrule
\centering{}{\footnotesize{}$2$} & {\footnotesize{}$1102$ } & {\footnotesize{}$1170$ } & {\footnotesize{}$0.0396$ } & {\footnotesize{}$0.04837$ } & {\footnotesize{}$0.22$ s}\tabularnewline
\midrule
\centering{}{\footnotesize{}$3$} & {\footnotesize{}$1469$ } & {\footnotesize{}$1560$ } & {\footnotesize{}$0.0394$ } & {\footnotesize{}$0.048415$ } & {\footnotesize{}$0.34$ s}\tabularnewline
\midrule
\centering{}{\footnotesize{}$4$} & {\footnotesize{}$1836$ } & {\footnotesize{}$1950$ } & {\footnotesize{}$0.0392157$ } & {\footnotesize{}$0.0480887$ } & {\footnotesize{}$0.81$ s}\tabularnewline
\midrule
\centering{}{\footnotesize{}$5$} & {\footnotesize{}$2203$ } & {\footnotesize{}$2340$ } & {\footnotesize{}$0.039026$ } & {\footnotesize{}$0.0476315$ } & {\footnotesize{}$2.6$ s}\tabularnewline
\midrule
\centering{}{\footnotesize{}$10$} & {\footnotesize{}$4038$ } & {\footnotesize{}$4290$ } & {\footnotesize{}$0.038113$ } & {\footnotesize{}$0.0451378$ } & {\footnotesize{}$8$ h $40$ m}\tabularnewline
\midrule
\centering{}{\footnotesize{}$25$} & {\footnotesize{}$9543$ } & {\footnotesize{}$10140$ } & {\footnotesize{}$0.035692$ } & {\footnotesize{}$0.0397182$ } & {\footnotesize{}$19$ h $15$ m}\tabularnewline
\bottomrule
\end{tabular}
\caption{{  Comparison between the performances of the optimal method obtained by solving \eqref{eq:potential-bnb-pep} with the BnB-PEP Algorithm and the method with $h=R / (\widetilde{L} \sqrt{(N+1)})$ as prescribed by \cite{davis2019stochastic}. The BnB-PEP Algorithm was executed on a standard laptop for $N=1,2,\ldots,10$, and on \texttt{MIT Supercloud} for $N=25$.}
\label{tab:BnB-PEP-result-potential-PEP}}
\end{table}

\begin{table}
\centering{}{\footnotesize{}}%
\begin{tabular}{cc}
\toprule 
{\footnotesize{}$N$} & {\footnotesize{}$h^{\star}$}\tabularnewline
\midrule
\midrule 
{\footnotesize{}$1$} & {\footnotesize{}$0.01$ }\tabularnewline
\midrule
{\footnotesize{}$2$} & {\footnotesize{}$0.0099664$ }\tabularnewline
\midrule
{\footnotesize{}$3$} & {\footnotesize{}$0.0099331$ }\tabularnewline
\midrule
{\footnotesize{}$4$} & {\footnotesize{}$0.0099$ }\tabularnewline
\midrule
{\footnotesize{}$5$} & {\footnotesize{}$0.00986714$ }\tabularnewline
\midrule
{\footnotesize{}$10$} & {\footnotesize{}$0.0097061$ }\tabularnewline
\midrule
{\footnotesize{}$25$} & {\footnotesize{}$0.0092553$ }\tabularnewline
\bottomrule
\end{tabular}\caption{{  Globally optimal stepsize obtained by solving \eqref{eq:potential-bnb-pep} with
the BnB-PEP Algorithm.}  \label{tab:Globally-optimal-stepsize-potential}}
\end{table}

\paragraph{Structured inner-dual variables.}

To find an analytical proof of an FSFOM defined by $h$, i.e., to find a feasible solution of \eqref{eq:potential-bnb-pep}, we use the methodology of $\mathsection$\ref{subsec:Sparsifier} to we find the following pattern of the optimal inner-dual variables $\lambda^{\star[k]}$, $\tau^{\star[k]}$, and $Z^{\star[k]}$
for all $k\in [0:N]$:
\begin{itemize}
\item Only $\lambda_{2,3}^{\star[k]}$, $\lambda_{2,0}^{\star[k]}$, and $\lambda_{0,2}^{\star[k]}$ are nonzero. Furthermore, $\lambda_{2,0}^{\star[k]}=\lambda_{0,2}^{\star[k]}$. 
\item Only $\tau_{2}^{\star[k]}$ is nonzero.
\item Only $Z_{2,2}^{\star[k]}$, $Z_{4,2}^{\star[k]}=Z_{2,4}^{\star[k]}$, $Z_{5,2}^{\star[k]}=Z_{2,5}^{\star[k]}$, $Z_{5,4}^{\star[k]}=Z_{4,5}^{\star[k]}$, $Z_{4,4}^{[k]}$, and $Z_{5,5}^{[k]}$
are nonzero. Furthermore,
$Z_{5,2}^{\star[k]}=-Z_{4,2}^{\star[k]}$, $Z_{4,4}^{\star[k]}=Z_{5,5}^{\star[k]}$, and
$Z_{4,4}^{\star[k]}=-Z_{5,4}^{\star}$.
\end{itemize}

Since the pattern is the same for all $k$, we enforce this pattern for a given $k$ in the constraint set of \eqref{eq:potential-bnb-pep}. 
This leads to a system of equations, and after some tedious but elementary calculations, we get the simplified form 
\begin{equation}
\begin{alignedat}{1} 
 & 4+(1-2h)b_{k+1}=b_{k},\\
 & \lambda_{2,3}^{[k]}=b_{k+1},\\
 & \lambda_{2,0}^{[k]}=\lambda_{0,2}^{[k]}=2hb_{k+1},\\
 & \tau_{2}^{[k]}=b_{k}-b_{k+1},
\end{alignedat}
\label{eq:assignments-potential-pep}
\end{equation}
for all $k\in[0:N]$.
We share Mathematica code for calculating \eqref{eq:assignments-potential-pep} symbolically as follows:
\begin{center}
\url{https://github.com/Shuvomoy/BnB-PEP-code/blob/main/Misc/potential.nb}
\end{center}
The solution numerically produced by BnB-PEP Algorithm have $b_{N+1}=0$, so we use that terminal value. Furthermore, from the numerical results,  we also empirically observe that $c_k^\star$, $h^\star$, and $b_{k+1}^\star$ follow the relationship $c_{k}=h^{2}b_{k+1}$ for all $k\in [0:N]$.

\paragraph{Convergence proof 1: Analytical solution to the BnB-PEP QCQP}
We restrict%
\footnote{
While unlikely, it is possible that a stepsize $h\notin(0, 1/2]$ achieves a better performance for some parameter values.
(Our numerical experiments indicate that this is not the case.)
If so, our choice of $h\in (0, 1/2]$ excludes this better choice. Regardless, our resulting choice of $h$ and its performance guarantee are valid.
}
our consideration to $h\in (0, 1/2]$.
The recursive relationship $4+(1-2h)b_{k+1}=b_{k}$ of \eqref{eq:assignments-potential-pep} with the terminal condition $b_{N+1}=0$ implies 
\begin{equation}
b_{k}=\frac{2}{h}\left(1-(1-2h)^{N+1-k}\right)\quad\textrm{for }k\in[0:N+1].\label{eq:b_k_closed_form}
\end{equation}
This formula and the resulting values from \eqref{eq:assignments-potential-pep} indeed make
$\{b_{k}\}_{k\in[0:N+1]}$, $\{c_{k}\}_{k\in[0:N]}$, $\{\lambda^{[k]}\}_{k\in[0:N]}$, and $\{\tau^{[k]}\}_{k\in[0:N]}$ non-negative.
Plugging  values from \eqref{eq:assignments-potential-pep} and $c_{k}=h^{2}b_{k+1}$ into \eqref{eq:potential-bnb-pep} we get
\begin{align*}
 & -q^{[k]}+\sum_{i\in[0:3]}\tau_{i}^{[k]}a_{i,\star}+\sum_{i,j\in[0:3]\cup\{\star\}:i\neq j}\lambda_{i,j}^{[k]}a_{i,j}=0,\\
 & -Q^{[k]}+\sum_{i,j\in[0:3]\cup\{\star\}:i\neq j}\lambda_{i,j}^{[k]}\left(A_{i,j}(h)-\frac{1}{2}B_{i,j}(h)\right)=Z^{[k]},\\
 & Z^{[k]}=\left(\begin{array}{ccccc}
0 & 0 & 0 & 0 & 0\\
0 & \frac{1}{2}h^{2}b_{k+1} & 0 & -\frac{1}{4}hb_{k+1} & \frac{1}{4}hb_{k+1}\\
0 & 0 & 0 & 0 & 0\\
0 & -\frac{1}{4}hb_{k+1} & 0 & \frac{1}{8}b_{k+1} & -\frac{1}{8}b_{k+1}\\
0 & \frac{1}{4}hb_{k+1} & 0 & -\frac{1}{8}b_{k+1} & \frac{1}{8}b_{k+1}
\end{array}\right)
\end{align*}
for all $k\in[0:N]$. The eigenvalues of $Z^{[k]}$ are $\{0,0,0,0,(1/4)(1+2h^{2})b_{k+1}\}$,
so $Z^{[k]}\succeq0$.
Thus the values of $\{b,c,\lambda,\tau\}$ defined by  \eqref{eq:assignments-potential-pep} and \eqref{eq:b_k_closed_form} is a feasible solution of \eqref{eq:potential-bnb-pep}, and we have proved the following theorem.

\begin{thm}
\label{thm:weakly-ncvx-main}
Let $N\in\nt$.
Let $f\in\mathcal{W}_{1,\widetilde{L}}$ have a global minimizer $x_\star$.
Let $R>0$, and let $x_{0}\in\rl^{d}$ satisfy the initial condition $f(y_{0})-f(x_{\star})+\|x_{0}-y_{0}\|^{2}\leq R^{2}$.
Consider the method
\[
x_{k+1}=x_{k}-hf^{\prime}(x_{k})
\]
for $k\in[0:N]$, 
where $f^{\prime}(x_{k})$ is an arbitrary subgradient of $f$ at $x_{k}$.
Let $y_{k}=\prox_{(1/2)f}(x_{k})$ and 
\[
\psi_{k}=b_{k}\left(f(y_{k})-f(x_{\star})+\|x_{k}-y_{k}\|^{2}\right)
\]
for $k\in[0:N+1]$.
If $h\in(0,1/2]$,  $4+(1-2h)b_{k+1}=b_{k}$ for $k\in[0:N]$, $b_{N+1}=0$, and $c_{k}=h^{2}b_{k+1}$ for $k\in[0:N]$,
then
\begin{equation*}
\|f^{\prime}(y_{k})\|^{2}+\psi_{k+1}-\psi_{k}\leq c_{k}\|f^{\prime}(x_{k})\|^{2}
\end{equation*}
for all $k\in[0:N]$, and 
\begin{equation*}
\frac{1}{N+1}\sum_{i=0}^{N}\|\nabla f_{(1/2)}\left(x_{i}\right)\|^{2}\leq\frac{1}{N+1}\left(\widetilde{L}^{2}\sum_{i=0}^{N}c_{i}+b_{0}R^{2}\right).
\end{equation*}
\end{thm}

To find the optimal $h^\star$, one can minimize the bound of Theorem~\ref{thm:weakly-ncvx-main}.
For notational simplicity, define $\kappa=R/L$.
With \eqref{eq:b_k_closed_form}, the analytical performance measure minimizes is
\begin{align}
 & \frac{\widetilde{L}^{2}}{N+1}\left(\sum_{i=0}^{N}c_{i}+\kappa^{2}b_{0}\right)\nonumber \\
= & \frac{\widetilde{L}^{2}}{N+1}\left[-1+(1-2h)^{N+1}+2h(N+1)+\kappa^{2}\frac{2}{h}\left(1-(1-2h)^{N+1}\right)\right].\label{eq:eq-to-solve-h-pot-bnb-pep}
\end{align}
As an aside, one can directly verify the nonnegativity of \eqref{eq:eq-to-solve-h-pot-bnb-pep} with Bernoulli's lower bound inequality, which states
that $(1+x)^{r}\geq1+rx$ for any positive integer $r\geq1$ and any
real $x\geq-1$ \cite[page 1]{UsefulInequalities}.
Plotting \eqref{eq:eq-to-solve-h-pot-bnb-pep} for different values $\kappa$ and $N$ reveals that it has a unique minimum in $h$.
Hence, the optimal $h^\star$ can be found by setting the derivative equal to zero:
\[
\frac{2\kappa^{2}\left((1-2h)^{N}-1\right)}{h^{2}}+\frac{4\kappa^{2}N(1-2h)^{N}}{h}-2(N+1)\left((1-2h)^{N}-1\right)=0.
\]
However, this equation does not seem to admit a simple algebraic solution.

To find a simpler analytical stepsize, we construct an upper bound of \eqref{eq:eq-to-solve-h-pot-bnb-pep}
that does admits a closed-form minimizer:
\begin{align}
 & \frac{\widetilde{L}^{2}}{N+1}\left[-1+(1-2h)^{N+1}+2h(N+1)+\kappa^{2}\frac{2}{h}\left(1-(1-2h)^{N+1}\right)\right]\nonumber \\
\overset{a)}{<} & \frac{\widetilde{L}^{2}}{N+1}\left[-1+\frac{1}{2(N+1)h+1}+2h(N+1)+\kappa^{2}\frac{2}{h}\right]\nonumber \\
\overset{b)}{<} & \frac{\widetilde{L}^{2}}{N+1}\left[-1+\frac{1}{2(N+1)h}+2h(N+1)+\kappa^{2}\frac{2}{h}\right]\nonumber \\
= & \frac{\widetilde{L}^{2}}{N+1}\left[-1+\frac{1}{2h}\left(\frac{1}{N+1}+4\kappa^{2}\right)+2h(N+1)\right].\label{eq:upper-bound-for-cost}
\end{align}
Here, $a)$ uses $1-(1-2h)^{N+1}<1$ and
\[
(1-2h)^{N+1}\leq\frac{1}{2(N+1)h+1},
\]
that follows from Bernoulli's upper bound inequality $(1+a)^{r}\leq1/(1-ra)$
for $a\in[-1,0],r\in\nt$ \cite[page 1]{UsefulInequalities} along
with $h\in(0,1/2]$ and $b)$ uses $1/\left(2(N+1)h+1\right)<1/\left(2(N+1)h\right)$.
The minimum of \eqref{eq:upper-bound-for-cost} is achieved at
\[
h_{\textup{ub}}^{\star}=\frac{\sqrt{4\kappa^{2}(N+1)+1}}{2(N+1)}.
\]
Plugging this back into \eqref{eq:upper-bound-for-cost}, we have the following analytical performance guarantee:
\begin{equation*}
\frac{\widetilde{L}^{2}}{N+1}\left(-1+\frac{1}{2h_{\textup{ub}}^{\star}}\left(\frac{1}{N+1}+4\kappa^{2}\right)+2h_{\textup{ub}}^{\star}(N+1)\right)=\frac{\widetilde{L}^{2}\left(2\sqrt{4\kappa^{2}(N+1)+1}-1\right)}{N+1}.
\end{equation*}
We have proved the following corollary.

\begin{cor}
\label{cor:main-cor}
In the setup of Theorem~\ref{thm:weakly-ncvx-main}, the choice
\[
h=\frac{\sqrt{4\kappa^{2}(N+1)+1}}{2(N+1)},
\]
yields the rate
\[
\frac{1}{N+1}\sum_{i=0}^{N}\|\nabla f_{(1/2)}(x_{i})\|^{2}\le \frac{\widetilde{L}^{2}\left(2\sqrt{4\kappa^{2}(N+1)+1}-1\right)}{N+1}.
\]
\end{cor}
The result of Corollary~\ref{cor:main-cor} strictly improves upon the prior state-of-the-art rate \cite[Theorem 3.1]{davis2019stochastic}
\begin{equation*}
\frac{1}{N+1}\sum_{i=0}^{N}\|\nabla f_{(1/2)}(x_{i})\|^{2}
\le 
\widetilde{L}^{2}\frac{4\kappa}{\sqrt{N+1}}
\end{equation*}
for large enough $N$ satisfying $N>(9-64\kappa^{2})/(64\kappa^{2})$.
(This stated rate is obtained by plugging the stepsize found in the proof of \cite[Theorem 3.1]{davis2019stochastic} into \cite[Equation 3.4]{davis2019stochastic} in the noiseless setup. The claimed rate \cite[Equation 3.5]{davis2019stochastic} has an error in the constant.)

\paragraph{Convergence proof 2: Classical analytical proof}
While the previous analytical proof is rigorous, its reliance on the BnB-PEP-QCQP formulation makes it inaccessible to those who do not already understand the PEP methodology. Therefore, we translate the proof into a classical form that does not depend on the PEP methodology. 

{  To clarify, the discovery of the previous proof was assisted by numerical solutions, but its correctness can be verified by humans without the aid of any numerical solvers. The benefit of the following alternate proof is its accessibility; the previous equivalent proof was equally correct and rigorous.}

\begin{proof}[Alternate proof of Theorem~\ref{thm:weakly-ncvx-main}]
The proof forms nonnegative combinations of valid inequalities and organizes the terms to establish the stated result.
The arguably mysterious weights of the nonnegative combinations correspond to the values of the inner-dual-variables listed in \eqref{eq:assignments-potential-pep}.

Note that
\begin{gather*}
f(y_{k+1})-f(y_{k})+\left\langle f^{\prime}(y_{k+1})\mid y_{k}-y_{k+1}\right\rangle -\frac{1}{2}\|y_{k}-y_{k+1}\|^{2}\leq0,\\
f(x_{k})-f(y_{k})+\left\langle f^{\prime}(x_{k})\mid y_{k}-x_{k}\right\rangle -\frac{1}{2}\|y_{k}-x_{k}\|^{2}\leq0,\\
f(y_{k})-f(x_{k})+\left\langle f^{\prime}(y_{k})\mid x_{k}-y_{k}\right\rangle -\frac{1}{2}\|x_{k}-y_{k}\|^{2}\leq0,
\end{gather*}
by weak convexity of $f$, and
\[
f(x_{\star})-f(y_{k})\leq0
\]
by the assumption that $x_\star$ is a global minimizer. Multiplying the last four inequalities with the nonnegative weights $b_{k+1}$, $2hb_{k+1}$, $2hb_{k+1}$, and $b_{k}-b_{k+1}=4-2hb_{k+1}$ (nonnegativity follows from \eqref{eq:b_k_closed_form}), respectively, and then adding them together, we obtain  
\begin{align*}
0\geq & b_{k+1}\left(f(y_{k+1})-f(y_{k})+\left\langle f^{\prime}(y_{k+1})\mid y_{k}-y_{k+1}\right\rangle -\frac{1}{2}\|y_{k}-y_{k+1}\|^{2}\right)+\\
 & 2hb_{k+1}\left(f(x_{k})-f(y_{k})+\left\langle f^{\prime}(x_{k})\mid y_{k}-x_{k}\right\rangle -\frac{1}{2}\|y_{k}-x_{k}\|^{2}\right)+\\
 & 2hb_{k+1}\left(f(y_{k})-f(x_{k})+\left\langle f^{\prime}(y_{k})\mid x_{k}-y_{k}\right\rangle -\frac{1}{2}\|x_{k}-y_{k}\|^{2}\right)+\\
 & (4-2hb_{k+1})\left(f(x_{\star})-f(y_{k})\right)\\
\overset{a)}{=} & b_{k+1}(f(y_{k+1})-f(x_{\star}))-b_{k}(f(y_{k})-f(x_{\star}))+\\
 & \frac{1}{8}b_{k+1}\Big[-4\|f^{\prime}(x_{k})\|^{2}h^{2}-2\left\langle f^{\prime}(y_{k})\mid2hf^{\prime}(x_{k})+f^{\prime}(y_{k+1})\right\rangle \\
 & +4\left\langle f^{\prime}(x_{k})\mid f^{\prime}(y_{k+1})\right\rangle h+(4h-1)\|f^{\prime}(y_{k})\|^{2}+3\|f^{\prime}(y_{k+1})\|^{2}\Big]\\
= & b_{k+1}(f(y_{k+1})-f(x_{\star}))-b_{k}(f(y_{k})-f(x_{\star}))\\
 & +\frac{1}{4}b_{k+1}\left(\|f^{\prime}(y_{k+1})\|^{2}-4h^{2}\|f^{\prime}(x_{k})\|^{2}-(1-2h)\|f^{\prime}(y_{k})\|^{2}\right)+\\
 & b_{k+1}\Big[\frac{1}{8}\Big(-4\|f^{\prime}(x_{k})\|^{2}h^{2}-2\left\langle f^{\prime}(y_{k})\mid2hf^{\prime}(x_{k})+f^{\prime}(y_{k+1})\right\rangle \\
 & +4\left\langle f^{\prime}(x_{k})\mid f^{\prime}(y_{k+1})\right\rangle h+(4h-1)\|f^{\prime}(y_{k})\|^{2}+3\|f^{\prime}(y_{k+1})\|^{2}\Big)-\\
 & \frac{1}{4}\left(\|f^{\prime}(y_{k+1})\|^{2}-4h^{2}\|f^{\prime}(x_{k})\|^{2}-(1-2h)\|f^{\prime}(y_{k})\|^{2}\right)\Big]\\
= & b_{k+1}(f(y_{k+1})-f(x_{\star}))-b_{k}(f(y_{k})-f(x_{\star}))\\
 & +\frac{1}{4}b_{k+1}\left(\|f^{\prime}(y_{k+1})\|^{2}-4h^{2}\|f^{\prime}(x_{k})\|^{2}-(1-2h)\|f^{\prime}(y_{k})\|^{2}\right)\\
 & +\frac{1}{8}\|2hf^{\prime}(x_{k})-f^{\prime}(y_{k})+f^{\prime}(y_{k+1})\|^{2}\\
\overset{b)}{\geq} & b_{k+1}(f(y_{k+1})-f(x_{\star}))-b_{k}(f(y_{k})-f(x_{\star}))\\
 & +\frac{1}{4}b_{k+1}\left(\|f^{\prime}(y_{k+1})\|^{2}-4h^{2}\|f^{\prime}(x_{k})\|^{2}-(1-2h)\|f^{\prime}(y_{k})\|^{2}\right)\\
\overset{c)}{=} & b_{k+1}(f(y_{k+1})-f(x_{\star}))-b_{k}(f(y_{k})-f(x_{\star}))+\|f^{\prime}(y_{k})\|^{2}\\
 & +\frac{b_{k+1}}{4}\|f^{\prime}(y_{k+1})\|^{2}-\frac{b_{k}}{4}\|f^{\prime}(y_{k})\|^{2}-c_{k}\|f^{\prime}(x_{k})\|^{2}\\
= & \|f^{\prime}(y_{k})\|_{2}+b_{k+1}\left(f(y_{k+1})-f(x_{\star})\|+\|y_{k+1}-w_{1}\|^{2}\right)\\
 & -b_{k}\left(f(y_{k})-f(x_{\star})+\|x_{k}-y_{k}\|^{2}\right)-c_{k}\|f^{\prime}(x_{k})\|^{2}\\
\overset{d)}{=} & \|f^{\prime}(y_{k})\|^{2}+\psi_{k+1}-\psi_{k}-c_{k}\|f^{\prime}(x_{k})\|^{2},
\end{align*}
where $a)$ uses \eqref{eq:assignments-potential-pep} and 
\begin{align*}
y_{k}-y_{k+1} & =hf^{\prime}(x_{k})-\frac{1}{2}f^{\prime}(y_{k})+\frac{1}{2}f^{\prime}(y_{k+1}),\\
x_{k}-y_{k} & =\frac{1}{2}f^{\prime}(y_{k}),
\end{align*}
 $b)$ removes the non-negative term in the previous line (the last
term), $c)$ uses \eqref{eq:assignments-potential-pep}, and
$d)$ uses the definition of $\psi_k$.
\end{proof}

{   
\textbf{Remark.}
The convergence proof above nowhere utilizes the assumption that $x_\star $ exists; it only requires that $f_\star=\inf_{x}f(x)>-\infty$. There was no a priori guarantee that we would get such a proof since the BnB-PEP-QCQP was allowed to use the existence of $x_\star$. However, BnB-PEP-QCQP chose not use that assumption in producing an optimal solution.
}

%% file: Sections/Images/Pot_BnB_PEP/BnB_PEP_vs_Davis_performance_comparison_LogLogScale.tex
\begin{tikzpicture}
\begin{loglogaxis}[legend style= {at={(0.25,0.25)},anchor=north, font=\tiny}, scale only axis, height=3.5cm, width=0.33*\textwidth, xmajorgrids, ymajorgrids, xlabel={$N$ (log scale)}, 
ylabel style={align=center}, ylabel=Objective of \eqref{eq:potential-bnb-pep}\\ (log scale),
legend pos={south west}]
    \addplot+[no marks, style={{ultra thick}}, color={blue}]
        coordinates {
            (1,0.04713563001512403)
            (2,0.04845040998952909)
            (3,0.0484500045381293)
            (4,0.04810228615194659)
            (5,0.047634255912512385)
            (6,0.04713128219341468)
            (7,0.04662144568932517)
            (8,0.04611391026157701)
            (9,0.04562013882339357)
            (10,0.04514259186821015)
            (11,0.04468178180541904)
            (12,0.04423821218434921)
            (13,0.043811600519106125)
            (14,0.04340064825710929)
            (15,0.043005317719817235)
            (16,0.04262452681559021)
            (17,0.04225710770334929)
            (18,0.04190243822920164)
            (19,0.04155987358117629)
            (20,0.04122844211729902)
            (21,0.04090758288929477)
            (22,0.04059678062040756)
            (23,0.04029527412486531)
            (24,0.040002612707912596)
            (25,0.039718397725270534)
            (26,0.03944207928899671)
            (27,0.039173272691111)
            (28,0.03891158114945758)
            (29,0.038656664969490745)
            (30,0.03840812192534003)
            (31,0.038165760844385326)
            (32,0.0379292634523038)
            (33,0.03769824170764589)
            (34,0.037472581445060285)
            (35,0.03725202950874836)
            (36,0.03703633054932841)
            (37,0.036825304496470986)
            (38,0.03661877604565877)
            (39,0.03641652880998451)
            (40,0.03621844301755849)
            (41,0.03602434291262462)
            (42,0.03583403978694495)
            (43,0.03564740141835865)
            (44,0.03546431569068986)
            (45,0.035284621465091155)
            (46,0.03510820402817326)
            (47,0.03493497668184052)
            (48,0.034764814885939674)
            (49,0.03459762038260225)
            (50,0.03443331159803701)
            (51,0.034271795265313675)
            (52,0.03411298635007747)
            (53,0.03395680619077775)
            (54,0.0338031758527688)
            (55,0.03365203451991424)
            (56,0.03350328714570988)
            (57,0.03335685928275671)
            (58,0.03321269447474864)
            (59,0.0330707175351753)
            (60,0.032930861292164514)
            (61,0.03279306602635888)
            (62,0.03265728165629139)
            (63,0.032523446575143874)
            (64,0.03239149248535876)
            (65,0.03226139621050733)
            (66,0.032133102595911546)
            (67,0.032006540002241464)
            (68,0.03188167275225543)
            (69,0.031758478749889206)
            (70,0.03163687507773777)
            (71,0.0315168198664522)
            (72,0.03139831994220922)
            (73,0.031281339948580386)
            (74,0.03116584266846454)
            (75,0.031051802558906695)
            (76,0.030939209696983664)
            (77,0.030828023232226547)
            (78,0.030718194071301927)
            (79,0.030609690857569018)
            (80,0.03050247956196323)
            (81,0.030396533666695057)
            (82,0.030291825784850696)
            (83,0.030188327582070337)
            (84,0.030086016441736214)
            (85,0.029984867560810157)
            (86,0.029884857364928255)
            (87,0.029785963227971942)
            (88,0.029688162322491152)
            (89,0.029591434932671638)
            (90,0.029495758024909197)
            (91,0.029401110935393375)
            (92,0.029307474759921577)
            (93,0.02921483004432513)
            (94,0.02912315667414432)
            (95,0.029032436665639032)
            (96,0.028942651888894947)
            (97,0.02885378435756249)
            (98,0.028765816570943657)
            (99,0.02867873164269589)
            (100,0.028592513854800824)
        }
        ;
    \addlegendentry {$h=R / (\widetilde{L} \sqrt{(N+1)})$}
    \addplot+[no marks, style={{ultra thick}}, color={red}]
        coordinates {
            (1,0.039759637697553696)
            (2,0.03963955522616405)
            (3,0.039477384477894)
            (4,0.0393252881780542)
            (5,0.03916615310639766)
            (6,0.03904769206817452)
            (7,0.038937307782696585)
            (8,0.038807204033935117)
            (9,0.03869779541285107)
            (10,0.03859248417052684)
            (11,0.03847476724198924)
            (12,0.03835137783577863)
            (13,0.03822784601241086)
            (14,0.03808713100147117)
            (15,0.03794502173575343)
            (16,0.037800410424409604)
            (17,0.037649262699834214)
            (18,0.037492660526387125)
            (19,0.037334529519285245)
            (20,0.037174136914476526)
            (21,0.03701215212671709)
            (22,0.036849950716402596)
            (23,0.03668744650191996)
            (24,0.03652505087944691)
            (25,0.036363337458757285)
            (26,0.036202285959249896)
            (27,0.036042093085762714)
            (28,0.03588297168754626)
            (29,0.035725051232756794)
            (30,0.0355684247184113)
            (31,0.03541329413868865)
            (32,0.03525970494235729)
            (33,0.035107769741258045)
            (34,0.034957506503890914)
            (35,0.03480890380332273)
            (36,0.03466188753593425)
            (37,0.034516459778486185)
            (38,0.03437267744002289)
            (39,0.03423037010421024)
            (40,0.034089670226034215)
            (41,0.03395058451262851)
            (42,0.03381315536639142)
            (43,0.033677387600047105)
            (44,0.033543224572107545)
            (45,0.03341070027170489)
            (46,0.0332798260240649)
            (47,0.033150515411984494)
            (48,0.03302271490171093)
            (49,0.03289640367716292)
            (50,0.03277155757092534)
            (51,0.03264810461724515)
            (52,0.032526020472729636)
            (53,0.032405279049054436)
            (54,0.03228585638274151)
            (55,0.03216768927487656)
            (56,0.03205083364019659)
            (57,0.03193527542850474)
            (58,0.03182098295401405)
            (59,0.03170796107241068)
            (60,0.03159619541966722)
            (61,0.031485661025966456)
            (62,0.031376329340538016)
            (63,0.031268195227717246)
            (64,0.031161223838774767)
            (65,0.0310553820698103)
            (66,0.03095066695413229)
            (67,0.03084710428336937)
            (68,0.030744654931846718)
            (69,0.030643268993276785)
            (70,0.030543021716880527)
            (71,0.030443902134332738)
            (72,0.03034579759406285)
            (73,0.03024869148638074)
            (74,0.030152571782721918)
            (75,0.030057402574016602)
            (76,0.029963130286215747)
            (77,0.029869754706892244)
            (78,0.029777288130281653)
            (79,0.0296857174832792)
            (80,0.02959503454390068)
            (81,0.0295052226671537)
            (82,0.029416267298326342)
            (83,0.029328156032808354)
            (84,0.029240872152558425)
            (85,0.029154402038358666)
            (86,0.029068731149837938)
            (87,0.02898384587960172)
            (88,0.02889973452914498)
            (89,0.028816380808431782)
            (90,0.028733776231333042)
            (91,0.028651908956348105)
            (92,0.02857076572924659)
            (93,0.028490335473692015)
            (94,0.028410608897771174)
            (95,0.028331574813981124)
            (96,0.02825322318744571)
            (97,0.028175544292044483)
            (98,0.028098529101150953)
            (99,0.028022168432628034)
            (100,0.027946451842920217)
        }
        ;
    \addlegendentry {$\textrm{BnB-PEP}$}
\end{loglogaxis}
\end{tikzpicture}

%% file: Sections/Images/Pot_BnB_PEP/PotPEP_timing_performance.tex
\begin{tikzpicture}
\begin{axis}[legend style= {at={(0.14,1)},anchor=north, font=\tiny}, scale only axis, height=3.5cm, width=0.33*\textwidth, ,xmajorgrids, ymajorgrids, xlabel={$N$}, ylabel={Solution time (s)}]
    \addplot[only marks, mark size={0.5pt}, style={{ultra thick}}, color={red}, mark={*}]
        coordinates {
            (1,5.03)
            (2,7.0819306)
            (3,8.772874)
            (4,16.3932434)
            (5,13.662243600000002)
            (6,13.661380000000001)
            (7,16.2505405)
            (8,18.4204801)
            (9,22.0906018)
            (10,32.4041902)
            (11,24.6387834)
            (12,36.5166363)
            (13,34.9536652)
            (14,36.6288369)
            (15,50.6814173)
            (16,45.030483)
            (17,43.701384)
            (18,46.228912799999996)
            (19,59.23676)
            (20,53.383929699999996)
            (21,71.4536844)
            (22,60.1698847)
            (23,63.296884899999995)
            (24,86.3285233)
            (25,62.93156210000001)
            (26,81.0004401)
            (27,79.82094790000001)
            (28,79.5518086)
            (29,84.3988849)
            (30,85.0204447)
            (31,96.4525628)
            (32,97.2546294)
            (33,98.6614084)
            (34,103.74087890000001)
            (35,103.49795879999999)
            (36,105.4225929)
            (37,115.99866190000002)
            (38,114.7861618)
            (39,124.4150066)
            (40,108.9217708)
            (41,108.5791666)
            (42,122.95919070000001)
            (43,259.6486347)
            (44,142.1474421)
            (45,132.427413)
            (46,161.3268218)
            (47,144.1126442)
            (48,157.1236193)
            (49,155.4406018)
            (50,150.3030252)
            (51,186.1156129)
            (52,140.0802644)
            (53,163.7652453)
            (54,164.4067707)
            (55,156.6864597)
            (56,158.6023222)
            (57,162.5751162)
            (58,194.699457)
            (59,197.9608605)
            (60,182.9809483)
            (61,503.67210079999995)
            (62,219.214967)
            (63,185.3447558)
            (64,185.72916239999998)
            (65,191.6997708)
            (66,250.30001579999998)
            (67,206.7774866)
            (68,207.57624209999997)
            (69,198.0674656)
            (70,224.47835880000002)
            (71,220.738943)
            (72,225.1516027)
            (73,294.3570224)
            (74,229.79976249999999)
            (75,241.5797887)
            (76,271.4422904)
            (77,239.5364949)
            (78,247.9310376)
            (79,244.78103830000003)
            (80,296.6600989)
            (81,255.3423907)
            (82,255.2086727)
            (83,303.2397917)
            (84,258.5449047)
            (85,272.182463)
            (86,292.5976837)
            (87,346.2444026)
            (88,279.88982569999996)
            (89,304.8399546)
            (90,328.2553806)
            (91,734.7230084)
            (92,328.06216159999997)
            (93,315.892359)
            (94,350.96283869999996)
            (95,381.962698)
            (96,328.3098677)
            (97,348.6993787)
            (98,395.08055240000004)
            (99,353.39237460000004)
            (100,373.3837091)
        }
        ;
\end{axis}
\end{tikzpicture}

%% file: Sections/conclusions.tex
\section{Conclusion \label{sec:Conclusion}}

The contribution of the BnB-PEP methodology is threefold. First, BnB-PEP
poses the problem of finding the optimal fixed-step first-order method for convex
or nonconvex, smooth or nonsmooth optimization as a nonconvex but
practically tractable QCQP called BnB-PEP-QCQP. Second, our methodology
presents the BnB-PEP Algorithm that solves the BnB-PEP-QCQP to certifiable
global optimality. Through exploiting specific problem structures,
the BnB-PEP Algorithm outperforms the latest off-the-shelf implementations
by orders of magnitude, reducing the solution-time from hours to seconds
and weeks to minutes. Third, we test the BnB-PEP methodology on a variety
of problem setups for which the prior methodologies failed and obtain
first-order methods with bounds, some numerical and some analytical, that improve upon prior state-of-the-art
results.

As the BnB-PEP offers significantly more flexibility compared to the prior performance estimation methodologies, we expect there to be many fruitful future directions of work utilizing the BnB-PEP methodology.
In particular, using the BnB-PEP to analyze and generate
composite optimization methods \cite{taylor2017exact,taylor2018exact,kim2018another},
randomized and stochastic methods \cite{Shi2017,taylor2019stochastic},
monotone operator and splitting methods \cite{ryu2020operator,gu2020tight,lieder2021convergence,kim2021accelerated},
mirror descent methods \cite{Dragomir2021}, and
adaptive methods  \cite{pmlr-v125-barre20a}
are all interesting directions of future work. Recently, novel worst-case convergence rates for nonlinear conjugate gradient methods were established in \cite{NCGPEP2023} using the BnB-PEP methodology.

\section*{Acknowledgements}

EKR were supported by the National Research Foundation of Korea (NRF) Grant funded by the Korean Government (MSIP) [NRF-2022R1C1C1010010] and the Samsung Science and Technology Foundation (Project Number SSTF-BA2101-02). We thank Yoel Drori, Robert M. Freund, Baptiste Goujaud, and Adrien Taylor for careful reading of the manuscript and constructive feedback.

%% file: Sections/supplementary_stuff.tex
\section{Appendix}

\subsection{Function class}
\label{appendix:function-class}
The BnB-PEP methodology applies to quadratically representable function classes.
{We say
$\mathcal{F}$ is \emph{quadratically representable} if the membership
$f\in\mathcal{F}$ is defined by an inequality of the form 
\[
c_{0}f(y)\geq c_{1}f(x)+q(x,y,u,v),\quad \forall u\in\partial f(x),\,v\in\partial f(y),\,x,y\in\rl^{d},
\]
where 
\begin{align*}
q(x,y,u,v)\triangleq\, & c_{2}\left\langle x\mid x\right\rangle +c_{3}\left\langle y\mid y\right\rangle +c_{4}\left\langle u\mid u\right\rangle +c_{5}\left\langle v\mid v\right\rangle +c_{6}\left\langle x\mid y\right\rangle +c_{7}\left\langle x\mid u\right\rangle \nonumber \\
 & \quad\quad\quad\quad\quad+c_{8}\left\langle x\mid v\right\rangle +c_{9}\left\langle y\mid u\right\rangle +c_{10}\left\langle y\mid v\right\rangle +c_{11}\left\langle u\mid v\right\rangle +c_{12},\label{eq:quadratic_inequality}
\end{align*}
with $c_{i}\in\mathbb{R}$ for $i\in[0:12]$ along with an optional
inequality of the form 
\[
\|u\|_{2}\le M,\quad\forall u\in\partial f(x),\,x\in\rl^{d},
\]
for some $M>0$.} 
Many of the commonly considered the function classes are quadratically representable, and we list a few in the following.
The class of $L$-smooth convex functions $\mathcal{F}_{0,L}$ satisfies \cite[Theorem 2.1.5, Equation (2.1.10)]{nesterov2003introductory}
\[
f(y)\geq f(x)+\left\langle \nabla f(x)\mid y-x\right\rangle +\frac{1}{2L}\|\nabla f(x)-\nabla f(y)\|^{2},\quad
\forall x,y\in\rl^{d}.
\]
($L$-smooth functions are differentiable everywhere.)
The class of $L$-smooth $\mu$-strongly convex functions $\mathcal{F}_{\mu,L}$ satisfies \cite[Theorem 1]{taylor2021optimal}
\begin{align*}\
f(y)&\geq f(x)+\left\langle \nabla f(x)\mid y-x\right\rangle +\frac{1}{2(L-\mu)}\|\nabla f(x)-\nabla f(y)\|^{2}\\
&\quad+\frac{L\mu}{2(L-\mu)}\|x-y\|^{2}-\frac{\mu}{2(L-\mu)}\left\langle \nabla f(x)-\nabla f(y)\mid x-y\right\rangle,\quad
\forall x,y\in\rl^{d}.
\end{align*}
The class of $L$-smooth nonconvex functions $\mathcal{F}_{-L,L}$ satisfies \cite[Theorem~6]{drori2020complexity}
\[
f(y)\geq f(x)+\left\langle \nabla f(x)\mid y-x\right\rangle +\frac{1}{2L}\|\nabla f(x)-\nabla f(y)\|^{2}
-\frac{L}{4}\|x-y-\frac{1}{L}(\nabla f(x)-\nabla f(y))\|^2
,\quad
\forall x,y\in\rl^{d}.
\]
which is also equivalent to \cite[Theorem 3.10]{taylor2017exact}
\[
f(y) \geq f(x)-\frac{L}{4}\|x-y\|^{2}+\frac{1}{2}\left\langle \nabla f(x)+\nabla f(y)\mid y-x \right\rangle +\frac{1}{4L}\|\nabla f(x)-\nabla f(y)\|^{2},\quad
\forall x,y\in\rl^{d}.
\]
The class of $\rho$-weakly convex functions $\mathcal{W}_{\rho,\infty}$ is satisfies \cite[Lemma 2.1]{davis2019stochastic}
\[
f(y)\geq f(x)+\left\langle u\mid y-x\right\rangle -\frac{\rho}{2}\|x-y\|^{2},\quad \forall u\in \partial f(x),\,x,y\in\rl^{d}.
\]
The class of nonsmooth convex functions with $L$-bounded subgradient $\mathcal{F}_{0,\infty}^{L}$ satisfies \cite[Definition 3.1]{taylor2017exact}
\begin{gather*}
f(y)\geq f(x)+\left\langle u\mid y-x\right\rangle,\quad \forall u\in \partial f(x),\,x,y\in\rl^{d},\\
\|u\|_{2}\leq L,\quad \forall\,u\in\partial f(x),\,x\in\rl^{d}.
\end{gather*}
The class of $\rho$-weakly convex functions with $L$-bounded subgradients $\mathcal{W}_{\rho,L}$ satisfies  \cite[Lemma 2.1]{davis2019stochastic}, \cite[$\mathsection$3.1 (c)]{taylor2017exact}
\begin{gather*}
f(y)\geq f(x)+\left\langle u\mid y-x\right\rangle -\frac{\rho}{2}\|x-y\|^{2}
,\quad \forall u\in \partial f(x),\,x,y\in\rl^{d},\\
\|u\|_{2}\leq L,\quad \forall\,u\in\partial f(x),\,x\in\rl^{d}.
\end{gather*}

Some, but not all, of the quadratically representable functions classes have interpolation results analogous to Lemma~\ref{Thm:Interpolation-inequality-FmuL}.
In particular, $\mathcal{F}_{0,L}$, $\mathcal{F}_{\mu,L}$, $\mathcal{F}_{-L,L}$, $\mathcal{W}_{\rho,\infty}$, and $\mathcal{F}_{0,\infty}^{L}$ have interpolation results (\cite[Theorem 4]{taylor2017smooth}, \cite[Theorem 3.10]{taylor2017exact},  \cite[Theorem 3.5]{taylor2017exact}), while $\mathcal{W}_{\rho,L}$ does not.
We can still use the BnB-PEP methodology without an interpolation condition, as we do in $\mathsection$\ref{subsec:pot-bnb-pep-wcvx-Moreau}, but we loose the tightness guarantee.

\subsection{Discussion on the strong duality assumption} \label{appendix:assumption2}

Consider the setup of $\mathsection$\ref{sec:QCQO-framework-for-template-problem}, and let us \emph{not} assume strong duality.
Then, by weak duality, we have
\[
\mathcal{R}(M(\alpha),\mathcal{E},\mathcal{F},\mathcal{C})\le \overline{\mathcal{R}}(M(\alpha),\mathcal{E},\mathcal{F},\mathcal{C}),
\]
where $\mathcal{R}$ is as given by \eqref{eq:worst-case-pfm-sdp-1}, $\overline{\mathcal{R}}$ is as defined in \eqref{eq:worst-case-pfm-dual-1}, and $M(\alpha)$ is the FSFOM parameterized by the stepsize list $\alpha=\{\alpha_{i,j}\}_{0\leq j<i\leq N}$.

Strong duality does provably hold generically.
Let 
\[
A^\mathrm{nice}=\{
\alpha \,|\, \alpha_{i, i-1} \neq 0,\,\forall\,i=1,\dots,N\}
\]
be the set of ``nice'' stepsize lists.
\begin{lem}{\cite[Theorem 6]{taylor2017smooth}}
\label{fact::2}
If $\alpha\in A^\mathrm{nice}$, then strong duality holds, i.e.,
\[
\mathcal{R}(M(\alpha),\mathcal{E},\mathcal{F},\mathcal{C})= \overline{\mathcal{R}}(M(\alpha),\mathcal{E},\mathcal{F},\mathcal{C}),
\quad
\forall\,\alpha\in A^\mathrm{nice}.
\]
\end{lem}
Note that $\alpha$ is an $N(N+1)/2$-dimensional and the set of $\alpha\notin A^\mathrm{nice}$ is lower-dimensional.
In this sense, strong duality holds generically.
When we solve the BnB-PEP-QCQP, we find
\[
\alpha^\star\in\argmin_{\alpha}\overline{\mathcal{R}}(M(\alpha),\mathcal{E},\mathcal{F},\mathcal{C}).
\]
In all of our experiments, we observe, a posteriori, that $\alpha^\star\in A^\mathrm{nice}$. 
So 
\[
\mathcal{R}(M(\alpha^\star),\mathcal{E},\mathcal{F},\mathcal{C})
=
\overline{\mathcal{R}}(M(\alpha^\star),\mathcal{E},\mathcal{F},\mathcal{C}).
\]
However, the actual problem we wish to solve is
\[
\begin{array}{ll}
\underset{\alpha }{\mbox{minimize}} &  \mathcal{R}(M(\alpha),\mathcal{E},\mathcal{F},\mathcal{C})\end{array},
\]
and it is possible that $\alpha^\star\notin \mathcal{R}(M(\alpha),\mathcal{E},\mathcal{F},\mathcal{C})$.
This would be the case if there is $\alpha^{\star,\mathrm{true}}\notin A^\mathrm{nice}$ such that 
\begin{gather}
\mathcal{R}(M(\alpha^{\star,\mathrm{true}}),\mathcal{E},\mathcal{F},\mathcal{C})
<
\overline{\mathcal{R}}(M(\alpha^{\star,\mathrm{true}}),\mathcal{E},\mathcal{F},\mathcal{C})
\label{eq:pathology2}
\\
\mathcal{R}(M(\alpha^{\star,\mathrm{true}}),\mathcal{E},\mathcal{F},\mathcal{C})
<
\mathcal{R}(M(\alpha^\star),\mathcal{E},\mathcal{F},\mathcal{C})
\label{eq:pathology1}
\\
\overline{\mathcal{R}}(M(\alpha^\star),\mathcal{E},\mathcal{F},\mathcal{C})\le 
\overline{\mathcal{R}}(M(\alpha^{\star,\mathrm{true}}),\mathcal{E},\mathcal{F},\mathcal{C})
\label{eq:pathology3}
\end{gather}
Condition\eqref{eq:pathology2} states that for $\alpha^{\star,\mathrm{true}}$, the duality gap is positive.
Condition \eqref{eq:pathology1} states that $\alpha^{\star,\mathrm{true}}$ has a better performance than $\alpha^\star$.
Condition \eqref{eq:pathology3} reaffirms that $\alpha^{\star}$ is optimal for dual problem.
Figure~\ref{appendix:figure1} illustrates this circumstance.
While this pathology is probably extremely unlikely, its possibility is a consequence of the gap in the reasoning.

\begin{figure}[h]
\centering
\includegraphics[width=0.8\textwidth]{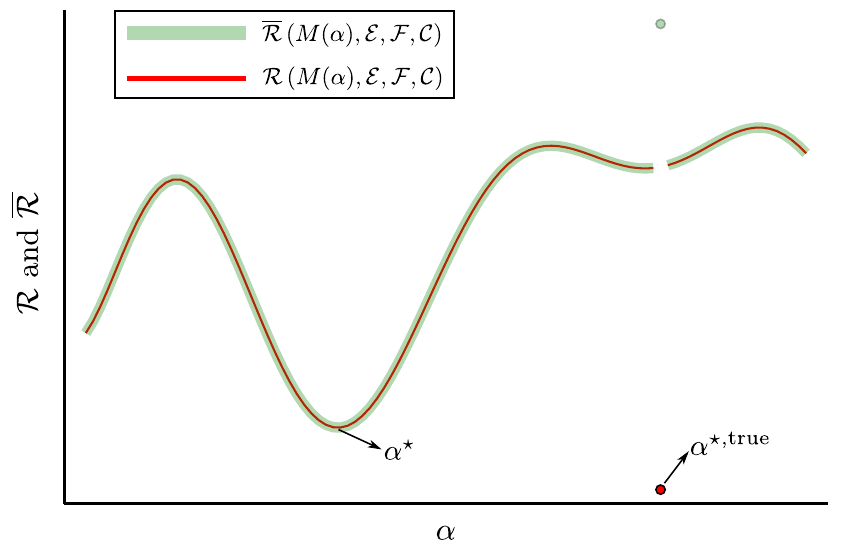}
\caption{
Illustration of $\mathcal{R}(M(\alpha),\mathcal{E},\mathcal{F},\mathcal{C})$ and $\overline{\mathcal{R}}(M(\alpha),\mathcal{E},\mathcal{F},\mathcal{C})$ in a pathological setup.
Even if $\alpha^\star\in \argmin\overline{\mathcal{R}}(M(\alpha),\mathcal{E},\mathcal{F},\mathcal{C})$ satisfies $\alpha^\star\in A^\mathrm{nice}$, it is possible that $\alpha^\star\notin \argmin\mathcal{R}(M(\alpha),\mathcal{E},\mathcal{F},\mathcal{C})$.}
\label{appendix:figure1}
\end{figure}

One can rigorously exclude this pathology in most well-behaved setups using the arguments of Park and Ryu \cite[Claim~4]{park2021}:
\begin{itemize}
\item[(i)] Prove $\mathcal{R}(M(\alpha),\mathcal{E},\mathcal{F},\mathcal{C})$ is a continuous function of $\alpha$.
\item[(ii)] Observe, a posteriori, that $\alpha^\star\in A^\mathrm{nice}$.
\end{itemize}
(One need not show that strong duality holds for $\alpha\notin A^\mathrm{nice}$.)
Then, we have
\begin{align*}
\inf_{\alpha\in \rl^{N(N+1)/2}}
\mathcal{R}(M(\alpha),\mathcal{E},\mathcal{F},\mathcal{C})
&\stackrel{\text{(a)}}{=}
\inf_{\alpha\in A^\mathrm{nice}}
\mathcal{R}(M(\alpha),\mathcal{E},\mathcal{F},\mathcal{C})\\
&\stackrel{\text{(b)}}{=}
\inf_{\alpha\in A^\mathrm{nice}}
\overline{\mathcal{R}}(M(\alpha),\mathcal{E},\mathcal{F},\mathcal{C})\\
&\stackrel{\text{(c)}}{=}
\inf_{\alpha\in \rl^{N(N+1)/2}}
\overline{\mathcal{R}}(M(\alpha),\mathcal{E},\mathcal{F},\mathcal{C})\\
&\stackrel{\text{(d)}}{=}
\overline{\mathcal{R}}(M(\alpha^\star),\mathcal{E},\mathcal{F},\mathcal{C}).
\end{align*}
where (a) follows from continuity of $\mathcal{R}(M(\alpha),\mathcal{E},\mathcal{F},\mathcal{C})$ and denseness of $A^\mathrm{nice}\subseteq\rl^{N(N+1)/2}$,
(b) follows from strong duality on $A^\mathrm{nice}$,
and (c) and (d) follows from the a posteriori observation that $\alpha^\star\in A^\mathrm{nice}$.
In this case, we would have $\alpha^\star\in \argmin\mathcal{R}(M(\alpha),\mathcal{E},\mathcal{F},\mathcal{C})$, even if there is a duality gap, as illustrated in Figure \ref{appendix:figure2}. Showing that $\mathcal{R}(M(\alpha),\mathcal{E},\mathcal{F},\mathcal{C})$ is a continuous function of $\alpha$ can be done by carefully arguing that, for the particular setup, the performance continuously depends on the FSFOM's stepsize.

\begin{figure}[h]
\centering
\includegraphics[width=0.8\textwidth]{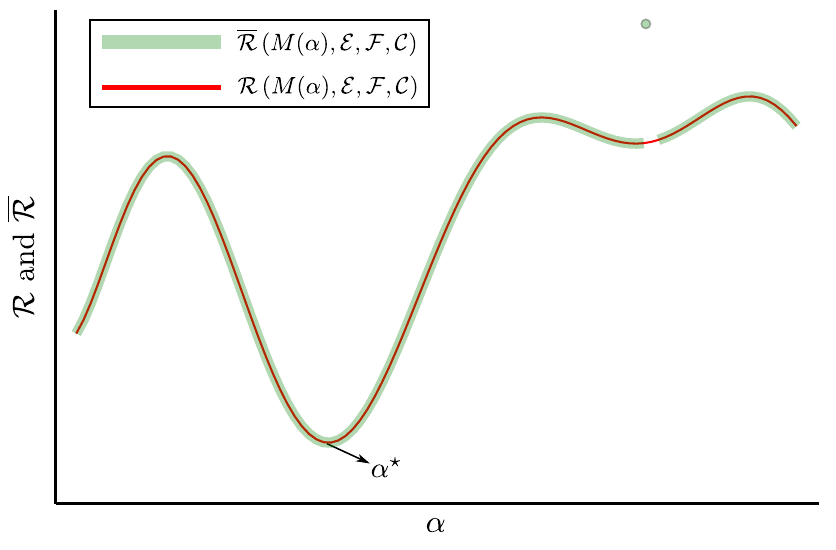}
\caption{
Illustration of $\mathcal{R}(M(\alpha),\mathcal{E},\mathcal{F},\mathcal{C})$ and $\overline{\mathcal{R}}(M(\alpha),\mathcal{E},\mathcal{F},\mathcal{C})$ when  $\mathcal{R}(M(\alpha),\mathcal{E},\mathcal{F},\mathcal{C})$ is continuous.
If $\alpha^\star\in \argmin\overline{\mathcal{R}}(M(\alpha),\mathcal{E},\mathcal{F},\mathcal{C})$ satisfies $\alpha^\star\in A^\mathrm{nice}$, then $\alpha^\star\in \argmin\mathcal{R}(M(\alpha),\mathcal{E},\mathcal{F},\mathcal{C})$, even if there is a duality gap.}
\label{appendix:figure2}
\end{figure}

\subsection{Exact rank-1 nonconvex semidefinite representation of the BnB-PEP-QCQP
\label{subsec:Derivation-of-the-rank-1-constraint}}

In this section, we derive the exact rank-1 nonconvex semidefinite representation
of \eqref{eq:BnB-PEP-Preli}. For the other BnB-PEP-QCQPs, the steps to construct such a  nonconvex semidefinite representation are identical. 

First, we define: 
\begin{equation}
w=\textrm{vec}(\alpha,\nu,\lambda),\label{eq:vectorized_variables}
\end{equation}
which stacks the elements of $\alpha,\nu,\lambda$ in a column vector
$w$, and denote its number of elements by $\lvert w\rvert$. Then
we can define an one-to-one and onto index selector function $\iota(\cdot)$
that takes $\alpha_{i,j}$, $\nu,$ or $\lambda_{i,j}$ as an input,
and provides the unique index of that element in $w$ with the range
$\{1,2,\ldots,\lvert w\rvert\}$ \emph{i.e.},
\[
\alpha_{i,j}=w_{\iota(\alpha_{i,j})},\nu=w_{\iota(\nu)}, \lambda_{i,j}=w_{\iota(\lambda_{i,j})}.
\]

Next, define $W=ww^{\top}\in\mathbb{S}^{\vert w\vert}.$ For notational
convenience define, $\vert\mathbf{x}_{0}\vert=N+2$. Defining the
map $\iota$ mathematically can be quite tedious and does not provide
any insight, but it very easy to implement through the \texttt{Julia
}packages \texttt{OrderedCollections} and \texttt{JuMP}. Recall that
for $i\in[1:N],$ we have:

\begin{align*}
\mathbf{x}_{i} & =\mathbf{x}_{0}\left(1-(\mu/L)\sum_{j=0}^{i-1}\alpha_{i,j}\right)-(1/L)\sum_{j=0}^{i-1}\alpha_{i,j}\mathbf{g}_{j}\\
 & =\underbrace{\left[I_{\vert\mathbf{x}_{0}\vert\times\vert\mathbf{x}_{0}\vert}\mid\frac{-1}{L}\left(\sum_{j=0}^{i-1}(\mu\mathbf{x}_{0}+\mathbf{g}_{j})e_{\iota(\alpha_{i,j})}^{\top}\right)\right]}_{\mathfrak{J}^{[i]}\in\rl^{\vert\mathbf{x}_{0}\vert\times(\vert\mathbf{x}_{0}\vert+\vert w\vert)}}\begin{bmatrix}\mathbf{x}_{0}\\
w
\end{bmatrix}
\end{align*}
 Also, define $\mathfrak{J}^{[\star]}=\mathbf{0}_{\vert\mathbf{x}_{0}\vert\times(\vert\mathbf{x}_{0}\vert+\vert w\vert)},$
and $\mathfrak{J}^{[0]}=[I_{\vert\mathbf{x}_{0}\vert\times\vert\mathbf{x}_{0}\vert}\mid\mathbf{0}_{\vert\mathbf{x}_{0}\vert\times\vert w\vert}].$
Then, for all $i\in I_{N}^{\star},$ we have:
\[
\mathbf{x}_{i}=\mathfrak{J}^{[i]}\begin{bmatrix}\mathbf{x}_{0}\\
w
\end{bmatrix}.
\]
Hence, we have for any $i,j\in I_{N}^{\star}$,
\begin{align*}
\mathbf{x}_{i}-\mathbf{x}_{j} & =(\mathfrak{J}^{[i]}-\mathfrak{J}^{[j]})\begin{bmatrix}\mathbf{x}_{0}\\
w
\end{bmatrix}\\
 & =\left[\underbrace{\mathfrak{G}^{[i,j]}}_{\in\rl^{\vert\mathbf{x}_{0}\vert\times\vert\mathbf{x}_{0}\vert}}\mid\underbrace{\mathfrak{H}^{[i,j]}}_{\in\rl^{\vert\mathbf{x}_{0}\vert\times\vert w\vert}}\right]\begin{bmatrix}\mathbf{x}_{0}\\
w
\end{bmatrix}\\
 & =\mathfrak{G}^{[i,j]}\mathbf{x}_{0}+\mathfrak{H}^{[i,j]}w\\
 & =\left[\underbrace{\mathfrak{g}^{[i,j][k]\top}\mathbf{x}_{0}}_{c^{[i,j][k]}}+\mathfrak{h}^{[i,j][k]\top}w\right]_{k=1}^{\vert\mathbf{x}_{0}\vert}\\
 & =\left[c^{[i,j][k]}+\mathfrak{h}^{[i,j][k]\top}w\right]_{k=1}^{\vert\mathbf{x}_{0}\vert},
\end{align*}
where $\mathfrak{h}^{[i,j][k]\top}$ and $\mathfrak{g}^{[i,j][k]\top}$
correspond to the $k$-th rows of $\mathfrak{H}^{[i,j]}$ and $\mathfrak{G}^{[i,j]}$,
respectively. Thus, for any $i,j\in I_{N}^{\star}$, $k,\ell\in[1:\vert\mathbf{x}_{0}\vert]:$
\begin{align}
[B_{i,j}(\alpha)]_{k,\ell} & =\left[(\mathbf{x}_{i}-\mathbf{x}_{j})\odot(\mathbf{x}_{i}-\mathbf{x}_{j})\right]_{k,\ell}\nonumber \\
 & =\left[\mathbf{x}_{i}-\mathbf{x}_{j}\right]_{k}\left[\mathbf{x}_{i}-\mathbf{x}_{j}\right]_{\ell}\nonumber \\
 & =\left[\mathfrak{G}^{[i,j]}\mathbf{x}_{0}+\mathfrak{H}^{[i,j]}w\right]_{k}\left[\mathfrak{G}^{[i,j]}\mathbf{x}_{0}+\mathfrak{H}^{[i,j]}w\right]_{\ell}\nonumber \\
 & =\left[c^{[i,j][k]}+\mathfrak{h}^{[i,j][k]\top}w\right]\left[c^{[i,j][\ell]}+\mathfrak{h}^{[i,j][\ell]\top}w\right]\nonumber \\
 & =c^{[i,j][k]}c^{[i,j][\ell]}+c^{[i,j][k]}\mathfrak{h}^{[i,j][\ell]\top}w+c^{[i,j][\ell]}\mathfrak{h}^{[i,j][k]\top}w\nonumber \\
 & \qquad+(\mathfrak{h}^{[i,j][k]\top}w)(\mathfrak{h}^{[i,j][\ell]\top}w)\nonumber \\
 & =c^{[i,j][k]}c^{[i,j][\ell]}+c^{[i,j][k]}\mathfrak{h}^{[i,j][\ell]\top}w+c^{[i,j][\ell]}\mathfrak{h}^{[i,j][k]\top}w\nonumber \\
 & \qquad+\sum_{\widetilde{i}=1}^{\lvert w\rvert}\sum_{\widetilde{j}=1}^{\lvert w\rvert}\mathfrak{h}_{\widetilde{i}}^{[i,j][k]}\mathfrak{h}_{\widetilde{j}}^{[i,j][\ell]}w_{\widetilde{i}}w_{\widetilde{j}}\nonumber \\
 & =c^{[i,j][k]}c^{[i,j][\ell]}+c^{[i,j][k]}\mathfrak{h}^{[i,j][\ell]\top}w+c^{[i,j][\ell]}\mathfrak{h}^{[i,j][k]\top}w\nonumber \\
 & \qquad+w^{\top}H^{[i,j][k,\ell]}w\nonumber \\
 & =c^{[i,j][k]}c^{[i,j][\ell]}+c^{[i,j][k]}\mathfrak{h}^{[i,j][\ell]\top}w+c^{[i,j][\ell]}\mathfrak{h}^{[i,j][k]\top}w\nonumber \\
 & \qquad+\tr\left(H^{[i,j][k,\ell]}W\right),\label{eq:lifted-eq-1}
\end{align}
 where $H^{[i,j][k,\ell]}\in\mathbb{S}^{\vert w\vert}$ with
its entries defined by
\[
H_{\widetilde{i},\widetilde{j}}^{[i,j][k,\ell]}=\frac{1}{2}\left(\mathfrak{h}_{\widetilde{i}}^{[i,j][k]}\mathfrak{h}_{\widetilde{j}}^{[i,j][\ell]}+\mathfrak{h}_{\widetilde{j}}^{[i,j][k]}\mathfrak{h}_{\widetilde{i}}^{[i,j][\ell]}\right),
\]
for $\widetilde{i},\widetilde{j}\in[1:\vert w\vert]$. Also, for $i,j\in I_{N}^{\star}$,
and $k,\ell\in[1:\vert\mathbf{x}_{0}\vert]:$
\begin{align*}
[A_{i,j}(\alpha)]_{k,\ell} & =\left[\mathbf{g}_{j}\odot(\mathbf{x}_{i}-\mathbf{x}_{j})\right]_{k,\ell}\\
 & =\left[\frac{1}{2}\mathbf{g}_{j}(\mathbf{x}_{i}-\mathbf{x}_{j})^{\top}+\frac{1}{2}(\mathbf{x}_{i}-\mathbf{x}_{j})\mathbf{g}_{j}^{\top}\right]_{k,\ell}\\
 & =\frac{1}{2}\left[\mathbf{g}_{j}\right]_{k}\left[\mathbf{x}_{i}-\mathbf{x}_{j}\right]_{\ell}+\frac{1}{2}\left[\mathbf{x}_{i}-\mathbf{x}_{j}\right]_{k}\left[\mathbf{g}_{j}\right]_{\ell}\\
 & =\frac{1}{2}\left[\mathbf{g}_{j}\right]_{k}(c^{[i,j][\ell]}+\mathfrak{h}^{[i,j][\ell]\top}w)+\frac{1}{2}(c^{[i,j][k]}+\mathfrak{h}^{[i,j][k]\top}w)\left[\mathbf{g}_{j}\right]_{\ell}\\
 & =\frac{1}{2}\left[c^{[i,j][\ell]}\left[\mathbf{g}_{j}\right]_{k}+c^{[i,j][k]}\left[\mathbf{g}_{j}\right]_{\ell}\right]+\\
 & \frac{1}{2}\left[\mathbf{g}_{j}\right]_{k}(\mathfrak{h}^{[i,j][\ell]\top}w)+\frac{1}{2}(\mathfrak{h}^{[i,j][k]\top}w)\left[\mathbf{g}_{j}\right]_{\ell}\\
 & =\overbrace{\frac{1}{2}\left[c^{[i,j][\ell]}\left[\mathbf{g}_{j}\right]_{k}+c^{[i,j][k]}\left[\mathbf{g}_{j}\right]_{\ell}\right]}^{=\widetilde{c}^{[i,j][k,\ell]}}+\\
 & \sum_{\widetilde{i}=1}^{\vert w\vert}\left(\frac{1}{2}\left[\mathbf{g}_{j}\right]_{k}\mathfrak{h}_{\widetilde{i}}^{[i,j][\ell]}+\frac{1}{2}\left[\mathbf{g}_{j}\right]_{\ell}\mathfrak{h}_{\widetilde{i}}^{[i,j][k]}\right)w_{\widetilde{i}}\\
 & =\widetilde{c}^{[i,j][k,\ell]}+\sum_{\widetilde{i}=1}^{\vert w\vert}\underbrace{\left(\frac{1}{2}\left[\mathbf{g}_{j}\right]_{k}\mathfrak{h}_{\widetilde{i}}^{[i,j][\ell]}+\frac{1}{2}\left[\mathbf{g}_{j}\right]_{\ell}\mathfrak{h}_{\widetilde{i}}^{[i,j][k]}\right)}_{=\widetilde{q}_{\widetilde{i}}^{[i,j][k,\ell]}}w_{\widetilde{i}}\\
 & =\widetilde{c}^{[i,j][k,\ell]}+\sum_{\widetilde{i}=1}^{\vert w\vert}\widetilde{q}_{\widetilde{i}}^{[i,j][k,\ell]}w_{\widetilde{i}}.
\end{align*}
Next, denoting $e_{\iota(\lambda_{i,j})}=\widetilde{d}^{[i,j]},$
we have for $k,\ell\in[1:\vert\mathbf{x}_{0}\vert]$
\begin{align}
 & \lambda_{i,j}\left[A_{i,j}(\alpha)\right]_{k,\ell}\nonumber \\
 & =(\widetilde{d}^{[i,j]\top}w)\left(\widetilde{c}^{[i,j][k,\ell]}+\sum_{\widetilde{i}=1}^{\vert w\vert}\widetilde{q}_{\widetilde{i}}^{[i,j][k,\ell]}w_{\widetilde{i}}\right)\nonumber \\
 & =\widetilde{c}^{[i,j][k,\ell]}(\widetilde{d}^{[i,j]\top}w)+\left(\sum_{\widetilde{j}=1}^{\vert w\vert}\widetilde{d}_{\widetilde{j}}^{[i,j]}w_{\widetilde{j}}\right)\left(\sum_{\widetilde{i}=1}^{\vert w\vert}\widetilde{q}_{\widetilde{i}}^{[i,j][k,\ell]}w_{\widetilde{i}}\right)\nonumber \\
 & =\widetilde{c}^{[i,j][k,\ell]}(\widetilde{d}^{[i,j]\top}w)+\sum_{\widetilde{i}=1}^{\vert w\vert}\sum_{\widetilde{j}=1}^{\vert w\vert}\left[\widetilde{d}_{\widetilde{j}}^{[i,j]}\widetilde{q}_{\widetilde{i}}^{[i,j][k,\ell]}\right]w_{\widetilde{i}}w_{\widetilde{j}}\nonumber \\
 & =\widetilde{c}^{[i,j][k,\ell]}(\widetilde{d}^{[i,j]\top}w)+w^{\top}S^{[i,j][k,\ell]}w\nonumber \\
 & =\widetilde{c}^{[i,j][k,\ell]}(\widetilde{d}^{[i,j]\top}w)+\tr\left(S^{[i,j][k,\ell]}W\right),\label{eq:lifted-eq-2}
\end{align}
 where: $S^{[i,j][k,\ell]}\in\mathbb{S}^{\vert w\vert}$
with its entries defined by
\begin{align*}
S^{[i,j][k,\ell]}[\widetilde{i},\widetilde{j}] & =\frac{1}{2}\left(\left[\widetilde{d}_{\widetilde{j}}^{[i,j]}\widetilde{q}_{\widetilde{i}}^{[i,j][k,\ell]}\right]+\left[\widetilde{d}_{\widetilde{i}}^{[i,j]}\widetilde{q}_{\widetilde{j}}^{[i,j][k,\ell]}\right]\right),
\end{align*}
 for $\widetilde{i},\widetilde{j}\in[1:\vert w\vert]$. Hence, we
have 
\[
\sum_{i,j\in I_{N}^{\star}:i\neq j}\lambda_{i,j}\left[A_{i,j}(\alpha)\right]_{k,\ell}=\left(\sum_{i,j\in I_{N}^{\star}:i\neq j}\widetilde{c}^{[i,j][k,\ell]}\widetilde{d}^{[i,j]\top}\right)w+\tr\left(\sum_{i,j\in I_{N}^{\star}}S^{[i,j][k,\ell]}\right)W.
\]
Using \eqref{eq:lifted-eq-1} and \eqref{eq:lifted-eq-2}, we have
the following nonconvex semidefinite representation of \eqref{eq:BnB-PEP-Preli}: 

\begin{align}
 & \mathcal{R}^{\star}\left(\mathcal{M}_{N}, \mathcal{E}, \mathcal{F},\mathcal{C}\right)\nonumber \\
= & \left(\begin{array}{ll}
\textrm{minimize}\quad\nu R^{2}\\
\textrm{subject to}\\
\sum_{(i,j)\in I_{N}^{\star}}\lambda_{i,j}a_{i,j}=0,\\
\nu\left[B_{0,\star}\right]{}_{k,\ell}-\left[C_{N,\star}\right]{}_{k,\ell}-\mu^{2}\left[B_{N,\star}(\alpha)\right]{}_{k,\ell}+\\
\qquad2\mu\left[A_{\star,N}(\alpha)\right]{}_{k,\ell}+\sum_{(i,j)\in I_{N}^{\star}}\lambda_{i,j}\left[A_{i,j}(\alpha)\right]_{k,\ell}+\\
\qquad\frac{1}{2(L-\mu)}\sum_{(i,j)\in I_{N}^{\star}}\lambda_{i,j}\left[C_{i,j}\right]{}_{k,\ell}=Z_{k,\ell},\quad k\in[1:\vert\mathbf{x}_{0}\vert],\;\ell\in[1:k],\\
\left[B_{N,\star}(\alpha)\right]{}_{k,\ell}=c^{[N,\star][k]}c^{[N,\star][\ell]}+c^{[N,\star][k]}\mathfrak{h}^{[N,\star][\ell]\top}w+c^{[N,\star][\ell]}\mathfrak{h}^{[N,\star][k]\top}w+\\
\qquad\tr\left(H^{[N,\star][k,\ell]}W\right),\quad k\in[1:\vert\mathbf{x}_{0}\vert],\;\ell\in[1:k],\\
\sum_{i,j\in I_{N}^{\star}:i\neq j}\lambda_{i,j}\left[A_{i,j}(\alpha)\right]_{k,\ell}=\left(\sum_{i,j\in I_{N}^{\star}:i\neq j}\widetilde{c}^{[i,j][k,\ell]}\widetilde{d}^{[i,j]\top}\right)w+\\
\qquad\tr\left(\sum_{i,j\in I_{N}^{\star}}S^{[i,j][k,\ell]}\right)W,\quad k\in[1:\vert\mathbf{x}_{0}\vert],\;\ell\in[1:k],\\
Z\succeq0,\\
W=ww^{\top},\\
\left(\forall i,j\in I_{N}^{\star}\right)\quad\lambda_{i,j}\geq0,\;\nu\geq0,
\end{array}\right)\label{eq:BnB-PEP-Preli-1-ncvx-SDP}
\end{align}
where $\lambda,\nu,Z,$ and $W$ are the decision variables, and $w=\textrm{vec}(\alpha,\nu,\lambda)$
as defined in \eqref{eq:vectorized_variables}. The constraint $W=ww^{\top}$
is nonconvex, but if we replace this constraint with the implied constraint
$W\succeq ww^{\top}$, then by using Schur complement, a convex semidefinite
relaxation of \eqref{eq:BnB-PEP-Preli} is given by: 

\begin{align}
& \left(\begin{array}{ll}
\textrm{minimize}\quad\nu R^{2}\\
\textrm{subject to}\\
\sum_{i,j\in I_{N}^{\star}}\lambda_{i,j}a_{i,j}=0,\\
\nu\left[B_{0,\star}\right]{}_{k,\ell}-\left[C_{N,\star}\right]{}_{k,\ell}-\mu^{2}\left[B_{N,\star}(\alpha)\right]{}_{k,\ell}+\\
\qquad2\mu\left[A_{\star,N}(\alpha)\right]{}_{k,\ell}+\sum_{(i,j)\in I_{N}^{\star}}\lambda_{i,j}\left[A_{i,j}(\alpha)\right]_{k,\ell}+\\
\qquad\frac{1}{2(L-\mu)}\sum_{(i,j)\in I_{N}^{\star}}\lambda_{i,j}\left[C_{i,j}\right]{}_{k,\ell}=Z_{k,\ell},\quad k\in[1:\vert\mathbf{x}_{0}\vert],\;\ell\in[1:k],\\
\left[B_{N,\star}(\alpha)\right]{}_{k,\ell}=c^{[N,\star][k]}c^{[N,\star][\ell]}+c^{[N,\star][k]}\mathfrak{h}^{[N,\star][\ell]\top}w+c^{[N,\star][\ell]}\mathfrak{h}^{[N,\star][k]\top}w+\\
\qquad\tr\left(H^{[N,\star][k,\ell]}W\right),\quad k\in[1:\vert\mathbf{x}_{0}\vert],\;\ell\in[1:k],\\
\sum_{i,j\in I_{N}^{\star}:i\neq j}\lambda_{i,j}\left[A_{i,j}(\alpha)\right]_{k,\ell}=\left(\sum_{i,j\in I_{N}^{\star}:i\neq j}\widetilde{c}^{[i,j][k,\ell]}\widetilde{d}^{[i,j]\top}\right)w+\\
\qquad\tr\left(\sum_{i,j\in I_{N}^{\star}}S^{[i,j][k,\ell]}\right)W\quad k\in[1:\vert\mathbf{x}_{0}\vert],\;\ell\in[1:k],\\
Z\succeq0,\\
\begin{bmatrix}W & w\\
w^{\top} & 1
\end{bmatrix}\succeq0,\\
\left(\forall i,j\in I_{N}^{\star}\right)\quad\lambda_{i,j}\geq0,\;\nu\geq0,
\end{array}\right)\label{eq:BnB-PEP-Preli-1-cvx-relaxation}
\end{align}
where $\lambda,\nu,Z,$ and $W$ are the decision variables. The optimal
objective value of \eqref{eq:BnB-PEP-Preli-1-cvx-relaxation} will
provide a lower bound to \eqref{eq:BnB-PEP-Preli-1-ncvx-SDP}.

%% file: BnB_PEP_manuscript.bbl
\begin{thebibliography}{10}

\bibitem{Gurobi95}
{Gurobi 10: New Advances}.
\newblock 2021.
\newblock
  \url{https://www.gurobi.com/products/gurobi-optimizer/whats-new-current-release/}.

\bibitem{abbaszadehpeivasti2021exact}
H.~Abbaszadehpeivasti, E.~de~Klerk, and M.~Zamani.
\newblock The exact worst-case convergence rate of the gradient method with
  fixed step lengths for $l$-smooth functions.
\newblock {\em Optimization Letters}, 16(6):1649--1661, 2022.

\bibitem{Gurobi91}
T.~Achterberg.
\newblock {Non-Convex MIQCP in Gurobi 9.1: New Advances}.
\newblock 2020.
\newblock
  \url{https://cdn.gurobi.com/wp-content/uploads/2020/12/Non-Convex-MIQCP-in-Gurobi-9.1-New-Advances.pdf}.

\bibitem{Gurobi}
T.~Achterberg and E.~Towle.
\newblock {Non-Convex Quadratic Optimization: Gurobi 9.0}.
\newblock 2020.
\newblock
  \url{https://www.gurobi.com/resource/non-convex-quadratic-optimization/}.

\bibitem{barre2020}
M.~Barr\'e, A.~B. Taylor, and F.~Bach.
\newblock Principled analyses and design of first-order methods with inexact
  proximal operators.
\newblock {\em Mathematical Programming}, 2022.

\bibitem{pmlr-v125-barre20a}
M.~Barr\'e, A.~Taylor, and A.~d'Aspremont.
\newblock Complexity guarantees for {P}olyak steps with momentum.
\newblock {\em Conference on Learning Theory}, 2020.

\bibitem{Bauschke2021}
H.~H. Bauschke, W.~M. Moursi, and X.~Wang.
\newblock Generalized monotone operators and their averaged resolvents.
\newblock {\em Mathematical Programming}, 189(1--2):55--74, 2021.

\bibitem{beck2009fast}
A.~Beck and M.~Teboulle.
\newblock A fast iterative shrinkage-thresholding algorithm for linear inverse
  problems.
\newblock {\em SIAM Journal on Imaging Sciences}, 2(1):183--202, 2009.

\bibitem{benson2003solving}
H.~Y. Benson and R.~J. Vanderbei.
\newblock Solving problems with semidefinite and related constraints using
  interior-point methods for nonlinear programming.
\newblock {\em Mathematical Programming}, 95(2):279--302, 2003.

\bibitem{bertsimas2019machine}
D.~Bertsimas and J.~Dunn.
\newblock {\em Machine Learning Under a Modern Optimization Lens}.
\newblock Dynamic Ideas LLC, 2019.

\bibitem{bertsimas1997introduction}
D.~Bertsimas and J.~N. Tsitsiklis.
\newblock {\em Introduction to Linear Optimization}, volume~6.
\newblock Athena Scientific Belmont, MA, 1997.

\bibitem{bertsimas2005optimization}
D.~Bertsimas and R.~Weismantel.
\newblock {\em Optimization Over Integers}, volume~13.
\newblock Dynamic Ideas Belmont, MA, 2005.

\bibitem{BorweinLewis2006}
J.~Borwein and A.~S. Lewis.
\newblock {\em Convex Analysis and Nonlinear Optimization, 2nd Edition}.
\newblock Springer, New York, 2006.

\bibitem{boyd2004convex}
S.~Boyd and L.~Vandenberghe.
\newblock {\em Convex Optimization}.
\newblock Cambridge University Press, 2004.

\bibitem{byrd1997local}
R.~H. Byrd, G.~Liu, and J.~Nocedal.
\newblock On the local behavior of an interior point method for nonlinear
  programming.
\newblock {\em Numerical Analysis}, 1997:37--56, 1997.

\bibitem{byrd2006k}
R.~H. Byrd, J.~Nocedal, and R.~A. Waltz.
\newblock {KNITRO: An integrated package for nonlinear optimization}.
\newblock In G.~D. Pillo and M.~Roma, editors, {\em Large-Scale Nonlinear
  Optimization}, pages 35--59. Springer, 2006.

\bibitem{CandesWakinBoyd2008_enhancing}
E.~J. Cand{\`e}s, M.~B. Wakin, and S.~P. Boyd.
\newblock Enhancing sparsity by reweighted {$\ell_1$} minimization.
\newblock {\em Journal of Fourier Analysis and Applications}, 14(5):877--905,
  2008.

\bibitem{clarke1990optimization}
F.~H. Clarke.
\newblock {\em Optimization and Nonsmooth Analysis}.
\newblock SIAM, 1990.

\bibitem{cyrus2018}
S.~Cyrus, B.~Hu, B.~Van~Scoy, and L.~Lessard.
\newblock A robust accelerated optimization algorithm for strongly convex
  functions.
\newblock {\em American Control Conference}, 2018.

\bibitem{NCGPEP2023}
S.~Das~Gupta, R.~M. Freund, X.~A. Sun, and A.~B. Taylor.
\newblock Nonlinear conjugate gradient methods: worst-case convergence rates
  via computer-assisted analyses.
\newblock {\em arXiv preprint arXiv:2301.01530}, 2023.

\bibitem{davis2018stochasticArxivSecond}
D.~Davis and D.~Drusvyatskiy.
\newblock Stochastic model-based minimization of weakly convex functions.
\newblock {\em arXiv preprint arXiv:1803.06523}, 2018.

\bibitem{davis2018stochasticArxivFirst}
D.~Davis and D.~Drusvyatskiy.
\newblock Stochastic subgradient method converges at the rate $o(k^{-1/4})$ on
  weakly convex functions.
\newblock {\em arXiv preprint arXiv:1802.02988}, 2018.

\bibitem{davis2019stochastic}
D.~Davis and D.~Drusvyatskiy.
\newblock Stochastic model-based minimization of weakly convex functions.
\newblock {\em SIAM Journal on Optimization}, 29(1):207--239, 2019.

\bibitem{de2017worst}
E.~{de Klerk}, F.~Glineur, and A.~B. Taylor.
\newblock On the worst-case complexity of the gradient method with exact line
  search for smooth strongly convex functions.
\newblock {\em Optimization Letters}, 11(7):1185--1199, 2017.

\bibitem{de2020worst}
E.~de~Klerk, F.~Glineur, and A.~B. Taylor.
\newblock Worst-case convergence analysis of inexact gradient and {N}ewton
  methods through semidefinite programming performance estimation.
\newblock {\em SIAM Journal on Optimization}, 30(3):2053--2082, 2020.

\bibitem{Dragomir2021}
R.-A. Dragomir, A.~B. Taylor, A.~d'Aspremont, and J.~Bolte.
\newblock Optimal complexity and certification of {B}regman first-order
  methods.
\newblock {\em Mathematical Programming}, 194(1--2):41--83, 2022.

\bibitem{drori2014thesis}
Y.~Drori.
\newblock {\em Contributions to the Complexity Analysis of Optimization
  Algorithms.}
\newblock PhD thesis, Tel-Aviv University, Tel Aviv, Israel, 2014.

\bibitem{drori2017exact}
Y.~Drori.
\newblock The exact information-based complexity of smooth convex minimization.
\newblock {\em Journal of Complexity}, 39:1--16, 2017.

\bibitem{drori2020complexity}
Y.~Drori and O.~Shamir.
\newblock The complexity of finding stationary points with stochastic gradient
  descent.
\newblock {\em International Conference on Machine Learning}, 2020.

\bibitem{drori2020efficient}
Y.~Drori and A.~B. Taylor.
\newblock Efficient first-order methods for convex minimization: A constructive
  approach.
\newblock {\em Mathematical Programming}, 184(1):183--220, 2020.

\bibitem{drori2022oracle}
Y.~Drori and A.~B. Taylor.
\newblock On the oracle complexity of smooth strongly convex minimization.
\newblock {\em Journal of Complexity}, 68:101590, 2022.

\bibitem{drori2014performance}
Y.~Drori and M.~Teboulle.
\newblock Performance of first-order methods for smooth convex minimization: A
  novel approach.
\newblock {\em Mathematical Programming}, 145(1-2):451--482, 2014.

\bibitem{JuMPDunningHuchetteLubin2017}
I.~Dunning, J.~Huchette, and M.~Lubin.
\newblock {JuMP}: A modeling language for mathematical optimization.
\newblock {\em SIAM Review}, 59(2):295--320, 2017.

\bibitem{Fazel2003}
M.~Fazel, H.~Hindi, and S.~Boyd.
\newblock Log-det heuristic for matrix rank minimization with applications to
  {H}ankel and {E}uclidean distance matrices.
\newblock {\em American Control Conference}, 2003.

\bibitem{fiacco1990nonlinear}
A.~V. Fiacco and G.~P. McCormick.
\newblock {\em Nonlinear Programming: Sequential Unconstrained Minimization
  Techniques}.
\newblock SIAM, 1990.

\bibitem{goujaud2021super}
B.~Goujaud, D.~Scieur, A.~Dieuleveut, A.~B. Taylor, and F.~Pedregosa.
\newblock Super-acceleration with cyclical step-sizes.
\newblock {\em International Conference on Artificial Intelligence and
  Statistics}, 2022.

\bibitem{gu2020tight}
G.~Gu and J.~Yang.
\newblock Tight sublinear convergence rate of the proximal point algorithm for
  maximal monotone inclusion problems.
\newblock {\em SIAM Journal on Optimization}, 30(3):1905--1921, 2020.

\bibitem{hastie2019statistical}
T.~Hastie, R.~Tibshirani, and M.~Wainwright.
\newblock {\em Statistical Learning with Sparsity: The Lasso and
  Generalizations}.
\newblock Chapman and Hall/CRC, 2019.

\bibitem{higham2002accuracy}
N.~J. Higham.
\newblock {\em Accuracy and Stability of Numerical Algorithms}.
\newblock SIAM, 2002.

\bibitem{horn2012matrix}
R.~A. Horn and C.~R. Johnson.
\newblock {\em Matrix Analysis}.
\newblock Cambridge University Press, 2012.

\bibitem{horst2013global}
R.~Horst and H.~Tuy.
\newblock {\em Global Optimization: Deterministic Approaches}.
\newblock Springer Science \& Business Media, 2013.

\bibitem{kim2021accelerated}
D.~Kim.
\newblock Accelerated proximal point method for maximally monotone operators.
\newblock {\em Mathematical Programming}, 190(1--2):57--87, 2021.

\bibitem{kim2016optimized}
D.~Kim and J.~A. Fessler.
\newblock Optimized first-order methods for smooth convex minimization.
\newblock {\em Mathematical Programming}, 159(1):81--107, 2016.

\bibitem{kim2018another}
D.~Kim and J.~A. Fessler.
\newblock Another look at the fast iterative shrinkage/thresholding algorithm
  ({F}{I}{S}{T}{A}).
\newblock {\em SIAM Journal on Optimization}, 28(1):223--250, 2018.

\bibitem{kim2021optimizing}
D.~Kim and J.~A. Fessler.
\newblock Optimizing the efficiency of first-order methods for decreasing the
  gradient of smooth convex functions.
\newblock {\em Journal of Optimization Theory and Applications},
  188(1):192--219, 2021.

\bibitem{UsefulInequalities}
L.~Kozma.
\newblock {Useful Inequalities}.
\newblock 2021.
\newblock \url{https://www.lkozma.net/inequalities_cheat_sheet/ineq.pdf}.

\bibitem{lee2021}
J.~Lee, C.~Park, and E.~K. Ryu.
\newblock A geometric structure of acceleration and its role in making
  gradients small fast.
\newblock {\em Neural Information Processing Systems}, 2021.

\bibitem{Legat2021}
B.~Legat, O.~Dowson, J.~D. Garcia, and M.~Lubin.
\newblock {MathOptInterface}: A data structure for mathematical optimization
  problems.
\newblock {\em INFORMS Journal on Computing}, 2021.

\bibitem{lessard2016analysis}
L.~Lessard, B.~Recht, and A.~Packard.
\newblock Analysis and design of optimization algorithms via integral quadratic
  constraints.
\newblock {\em SIAM Journal on Optimization}, 26(1):57--95, 2016.

\bibitem{liberti2008introduction}
L.~Liberti.
\newblock Introduction to global optimization.
\newblock {\em Ecole Polytechnique}, 2008.

\bibitem{lieder2021convergence}
F.~Lieder.
\newblock On the convergence rate of the halpern-iteration.
\newblock {\em Optimization Letters}, 15(2):405--418, 2021.

\bibitem{locatelli2013global}
M.~Locatelli and F.~Schoen.
\newblock {\em Global Optimization: Theory, Algorithms, and Applications}.
\newblock SIAM, 2013.

\bibitem{luo2010semidefinite}
Z.-Q. Luo, W.-K. Ma, A.~M.-C. So, Y.~Ye, and S.~Zhang.
\newblock Semidefinite relaxation of quadratic optimization problems.
\newblock {\em IEEE Signal Processing Magazine}, 27(3):20--34, 2010.

\bibitem{lyapunov1892}
A.~M. Lyapunov.
\newblock The general problem of the stability of motion.
\newblock {\em Communications of the Mathematical Society of Kharkov}, 1892.

\bibitem{mccormick1976computability}
G.~P. McCormick.
\newblock {Computability of global solutions to factorable nonconvex programs:
  Part I—Convex underestimating problems}.
\newblock {\em Mathematical Programming}, 10(1):147--175, 1976.

\bibitem{mordukhovich2006PartI}
B.~S. Mordukhovich.
\newblock {\em Variational Analysis and Generalized Differentiation I: Basic
  Theory}.
\newblock Springer, 2006.

\bibitem{mosek}
{\relax MOSEK ApS}.
\newblock {\em MOSEK Optimizer API for C 9.3.6}, 2019.

\bibitem{nemirovski1995information}
A.~Nemirovski.
\newblock Information-based complexity of convex programming.
\newblock {\em Lecture Notes, {T}echnion - {I}srael {I}nstitute of
  {T}echnology}, 1995.

\bibitem{nemirovsky1992information}
A.~Nemirovsky.
\newblock Information-based complexity of linear operator equations.
\newblock {\em Journal of Complexity}, 8(2):153--175, 1992.

\bibitem{Nest83}
Y.~Nesterov.
\newblock A method of solving a convex programming problem with convergence
  rate ${O}(1/k^2)$.
\newblock {\em Soviet Mathematics Doklady}, 27(2):372--376, 1983.

\bibitem{nesterov2003introductory}
Y.~Nesterov.
\newblock {\em Lectures on Convex Optimization}, volume 137.
\newblock second edition, 2018.

\bibitem{padberg1989boolean}
M.~Padberg.
\newblock The boolean quadric polytope: some characteristics, facets and
  relatives.
\newblock {\em Mathematical Programming}, 45(1):139--172, 1989.

\bibitem{parkparkryu2021}
C.~Park, J.~Park, and E.~K. Ryu.
\newblock Factor-{$\sqrt{2}$} acceleration of accelerated gradient methods.
\newblock {\em Applied Mathematics \& Optimization}, 2023.

\bibitem{park2021}
C.~Park and E.~K. Ryu.
\newblock Optimal first-order algorithms as a function of inequalities.
\newblock {\em arXiv preprint arXiv:2110.11035}, 2021.

\bibitem{park2022}
J.~Park and E.~K. Ryu.
\newblock Exact optimal accelerated complexity for fixed-point iterations.
\newblock {\em International Conference on Machine Learning}, 2022.

\bibitem{Fabian21}
F.~Pedregosa.
\newblock {On the link between optimization and polynomials: acceleration
  without Momentum}.
\newblock 2021.
\newblock \url{http://fa.bianp.net/blog/2021/no-momentum/}.

\bibitem{reuther2018interactive}
A.~Reuther, J.~Kepner, C.~Byun, S.~Samsi, W.~Arcand, D.~Bestor, B.~Bergeron,
  V.~Gadepally, M.~Houle, M.~Hubbell, M.~Jones, A.~Klein, L.~Milechin,
  J.~Mullen, A.~Prout, A.~Rosa, C.~Yee, and P.~Michaleas.
\newblock Interactive supercomputing on 40,000 cores for machine learning and
  data analysis.
\newblock In {\em 2018 IEEE High Performance extreme Computing Conference
  (HPEC)}, pages 1--6. IEEE, 2018.

\bibitem{rockafellar2020characterizing}
R.~T. Rockafellar.
\newblock Characterizing firm nonexpansiveness of prox mappings both locally
  and globally.
\newblock {\em Journal of Nonlinear and Convex Analysis}, 22(5), 2021.

\bibitem{rockafellar2009variational}
R.~T. Rockafellar and R.~J.-B. Wets.
\newblock {\em Variational Analysis}.
\newblock Springer Science \& Business Media, 2009.

\bibitem{ryu2020operator}
E.~K. Ryu, A.~B. Taylor, C.~Bergeling, and P.~Giselsson.
\newblock Operator splitting performance estimation: Tight contraction factors
  and optimal parameter selection.
\newblock {\em SIAM Journal on Optimization}, 30(3):2251--2271, 2020.

\bibitem{RyuYin2022_largescale}
E.~K. Ryu and W.~Yin.
\newblock {\em Large-Scale Convex Optimization via Monotone Operators}.
\newblock {Cambridge University Press}, 2022.

\bibitem{sherali1990hierarchy}
H.~D. Sherali and W.~P. Adams.
\newblock A hierarchy of relaxations between the continuous and convex hull
  representations for zero-one programming problems.
\newblock {\em SIAM Journal on Discrete Mathematics}, 3(3):411--430, 1990.

\bibitem{sherali2002enhancing}
H.~D. Sherali and B.~M. Fraticelli.
\newblock Enhancing {RLT} relaxations via a new class of semidefinite cuts.
\newblock {\em Journal of Global Optimization}, 22(1):233--261, 2002.

\bibitem{Shi2017}
Z.~Shi and R.~Liu.
\newblock Better worst-case complexity analysis of the block coordinate descent
  method for large scale machine learning.
\newblock {\em International Conference on Machine Learning and Applications},
  2017.

\bibitem{taylor2017convex}
A.~B. Taylor.
\newblock {\em Convex Interpolation and Performance Estimation of First-Order
  Methods for Convex Optimization}.
\newblock PhD thesis, Catholic University of Louvain, Louvain-la-Neuve,
  Belgium, 2017.

\bibitem{TaylorBlog20}
A.~B. Taylor.
\newblock {Computer-aided analyses in optimization}.
\newblock 2020.
\newblock \url{https://francisbach.com/computer-aided-analyses/}.

\bibitem{taylor2019stochastic}
A.~B. Taylor and F.~Bach.
\newblock Stochastic first-order methods: Non-asymptotic and computer-aided
  analyses via potential functions.
\newblock {\em Conference on Learning Theory}, 2019.

\bibitem{taylor2021optimal}
A.~B. Taylor and Y.~Drori.
\newblock An optimal gradient method for smooth strongly convex minimization.
\newblock {\em Mathematical Programming}, pages 1--38, 2022.

\bibitem{taylor2017exact}
A.~B. Taylor, J.~M. Hendrickx, and F.~Glineur.
\newblock Exact worst-case performance of first-order methods for composite
  convex optimization.
\newblock {\em SIAM Journal on Optimization}, 27(3):1283--1313, 2017.

\bibitem{taylor2017smooth}
A.~B. Taylor, J.~M. Hendrickx, and F.~Glineur.
\newblock Smooth strongly convex interpolation and exact worst-case performance
  of first-order methods.
\newblock {\em Mathematical Programming}, 161(1--2):307--345, 2017.

\bibitem{taylor2018exact}
A.~B. Taylor, J.~M. Hendrickx, and F.~Glineur.
\newblock Exact worst-case convergence rates of the proximal gradient method
  for composite convex minimization.
\newblock {\em Journal of Optimization Theory and Applications},
  178(2):455--476, 2018.

\bibitem{taylor2018lyapunov}
A.~B. Taylor, B.~Van~Scoy, and L.~Lessard.
\newblock Lyapunov functions for first-order methods: Tight automated
  convergence guarantees.
\newblock {\em International Conference on Machine Learning}, 2018.

\bibitem{Tran-Dinh2022}
Q.~Tran-Dinh.
\newblock The connection between {N}esterov's accelerated methods and {H}alpern
  fixed-point iterations.
\newblock {\em arXiv preprint arXiv:2203.04869}, 2022.

\bibitem{van2017fastest}
B.~Van~Scoy, R.~A. Freeman, and K.~M. Lynch.
\newblock The fastest known globally convergent first-order method for
  minimizing strongly convex functions.
\newblock {\em IEEE Control Systems Letters}, 2(1):49--54, 2017.

\bibitem{wachter2006implementation}
A.~W{\"a}chter and L.~T. Biegler.
\newblock On the implementation of an interior-point filter line-search
  algorithm for large-scale nonlinear programming.
\newblock {\em Mathematical Programming}, 106(1):25--57, 2006.

\bibitem{IpoptLocConv}
A.~Wächter and L.~T. Biegler.
\newblock Line search filter methods for nonlinear programming: Local
  convergence.
\newblock {\em SIAM Journal on Optimization}, 16(1):32--48, 2005.

\bibitem{yoon2021accelerated}
T.~Yoon and E.~K. Ryu.
\newblock Accelerated algorithms for smooth convex-concave minimax problems
  with $\mathcal{O}(1/k^2)$ rate on squared gradient norm.
\newblock {\em International Conference on Machine Learning}, 2021.

\bibitem{young1953}
D.~Young.
\newblock On {R}ichardson's method for solving linear systems with positive
  definite matrices.
\newblock {\em Journal of Mathematics and Physics}, 32(1--4):243--255, 1953.

\bibitem{ZhouTianSoCheng2021_practical}
K.~Zhou, L.~Tian, A.~M.-C. So, and J.~Cheng.
\newblock Practical schemes for finding near-stationary points of convex
  finite-sums.
\newblock {\em International Conference on Artificial Intelligence and
  Statistics}, 2022.

\end{thebibliography}
